\documentclass[reqno]{amsart}
   
\usepackage[utf8]{inputenc}
\usepackage[T1]{fontenc}
\usepackage[svgnames,hyperref]{xcolor}
\usepackage{lmodern}
\usepackage{amssymb} 
\usepackage{xspace} 
\usepackage[nomessages]{fp} 
\usepackage{ifthen} 
 
\newcommand{\nicecolor}{Navy}
\usepackage[pdfauthor={J. Blanc, S. Lamy \& S. Zimmermann}, backref=page, colorlinks, linktocpage, citecolor=\nicecolor, linkcolor=\nicecolor, urlcolor=\nicecolor]{hyperref}
\usepackage[all]{hypcap} 

\usepackage[shortlabels]{enumitem}
\setlist[1]{wide}
\setlist[2]{leftmargin=15mm}
\setlist[enumerate]{label=\rm{(\arabic*)}}
\setlist[enumerate,2]{label=\rm({\it\roman*}), }
\setlist[itemize]{label=\raisebox{0.25ex}{\tiny$\bullet$}}
\definecolor{grisclair}{rgb}{0.9,0.9,0.9}
\def\vertsubseteq{\hbox{\vrule\hskip0.2ex\hbox{$\cap$}}}
\usepackage{tikz}
\usetikzlibrary{cd,arrows,positioning}
\tikzset{>=stealth}
\tikzcdset{arrow style=tikz}
\tikzset{link/.style={column sep=1.8cm,row sep=0.16cm}}
\tikzset{map/.style={row sep=0em, column sep=0em}}

\DeclareFontFamily{U}{mathb}{\hyphenchar\font45}
\DeclareFontShape{U}{mathb}{m}{n}{
      <5> <6> <7> <8> <9> <10> gen * mathb
      <10.95> mathb10 <12> <14.4> <17.28> <20.74> <24.88> mathb12
      }{}
\DeclareSymbolFont{mathb}{U}{mathb}{m}{n}
\DeclareMathSymbol{\bigast}{1}{mathb}{"06}

\DeclareFontFamily{U}{mathx}{\hyphenchar\font45}
\DeclareFontShape{U}{mathx}{m}{n}{<-> mathx10}{}
\DeclareSymbolFont{mathx}{U}{mathx}{m}{n}
\DeclareMathAccent{\widebar}{0}{mathx}{"73}

\renewcommand{\to}{ \, \tikz[baseline=-.6ex] \draw[->,line width=.5] (0,0) -- +(.5,0); \, }
\renewcommand{\rightarrow}{ \, \tikz[baseline=-.6ex] \draw[->,line width=.5] (0,0) -- +(.5,0); \, }
\newcommand{\longto}{ \, \tikz[baseline=-.6ex] \draw[->,line width=.5] (0,0) -- +(.9,0); \, }

\newcommand{\rat}{ \, \tikz[baseline=-.6ex] \draw[->,densely dashed,line width=.5] (0,0) -- +(.5,0); \, }

\newcommand{\longrat}{ \, \tikz[baseline=-.6ex] \draw[->,densely dashed,line width=.5] (0,0) -- +(.9,0); \, }
\newcommand{\ps}{ \, \tikz[baseline=-.6ex] \draw[->,dotted,line width=.6] (0,0) -- +(.5,0); \,}
\newcommand{\dottedline}{ \, \tikz[baseline=-.6ex] \draw[dotted,line width=.6] (0,0) -- +(.5,0); \,}
\newcommand{\mapsrat}{ \, \tikz[baseline=-.6ex] \draw[|->,densely dashed,line width=.5] (0,0) -- +(.5,0); \, }
\renewcommand{\mapsto}{ \, \tikz[baseline=-.6ex] \draw[|->,line width=.5] (0,0) -- +(.5,0); \, }
\newcommand\iso{\stackrel{\sim}{\to}}
\newcommand{\tto}{ \, \tikz[baseline=-.6ex] \draw[->>,line width=.5] (0,0) -- +(.5,0); \, }

\newcommand{\hookto}{ \, \tikz[baseline=-.6ex] \draw[right hook->,line width=.5] (0,0) -- +(.5,0); \, }
\renewcommand{\hookrightarrow}{\hookto}

\newcommand{\Ra}{2.6cm}
\newcommand{\Rb}{1.70cm}
\newcommand{\Rc}{0.85cm}

\newcommand{\verteq}{\rotatebox{90}{$\,=$}}

\newcommand{\I}{\textup{I}\xspace}
\newcommand{\II}{\textup{II}\xspace}
\newcommand{\III}{\textup{III}\xspace}
\newcommand{\IV}{\textup{IV}\xspace}

\newcommand{\A}{\mathbb{A}}
\newcommand{\F}{\mathbb{F}}
\newcommand{\p}{\mathbb{P}}

\newcommand{\Z}{\mathbf{Z}}
\newcommand{\N}{\mathbf{N}}
\renewcommand{\k}{\mathbf{k}}
\newcommand{\C}{\mathbf{C}}
\newcommand{\R}{\mathbf{R}}
\newcommand{\Q}{\mathbf{Q}}

\newcommand{\Al}{\mathcal A} 
\newcommand{\Bl}{\mathcal B} 
\newcommand{\Cl}{\mathcal C}

\newcommand{\Fl}{\mathcal F}
\newcommand{\Gl}{\mathcal G} 
\newcommand{\Hl}{\mathcal H} 
\newcommand{\Il}{\mathcal I} 
\newcommand{\Ll}{\mathcal L} 

\newcommand{\Ol}{\mathcal O}
\newcommand{\Pl}{\mathcal P}
\newcommand{\Ql}{\mathcal Q}
\newcommand{\Sl}{\mathcal S}

\DeclareMathOperator{\PGL}{PGL}
\DeclareMathOperator{\PSL}{PSL}
\DeclareMathOperator{\GL}{GL}
\DeclareMathOperator{\SL}{SL}

\renewcommand{\setminus}{\smallsetminus}
\newcommand{\eps}{{\varepsilon}}
\renewcommand{\epsilon}{\varepsilon}
\renewcommand{\phi}{\varphi}
\renewcommand{\le}{\leqslant}
\renewcommand{\ge}{\geqslant}
\renewcommand{\leq}{\leqslant}
\renewcommand{\geq}{\geqslant}
\newcommand{\id}{\textup{id}}
\newcommand{\pt}{\textup{pt}}
\newcommand{\pr}{\textup{pr}}
\newcommand{\Mat}{\textup{Mat}}
\newcommand{\Sym}{\textup{Sym}}
\newcommand{\NE}{\mathit{NE}}
\newcommand{\BVA}{\textup{BVA}}
\newcommand{\CB}{\textup{\sffamily{CB}}}
\newcommand{\MC}{\textup{\sffamily{M}}}
\renewcommand{\bigoplus}{\mathop{\oplus}}
\newcommand{\semicolon}{\, ; \,}

\DeclareMathOperator{\Ker}{Ker}
\DeclareMathOperator{\Aut}{Aut}
\DeclareMathOperator{\Bir}{Bir}
\DeclareMathOperator{\Tame}{Tame}
\DeclareMathOperator{\Dil}{Dil}
\DeclareMathOperator{\Div}{Div}
\DeclareMathOperator{\Pic}{Pic}
\DeclareMathOperator{\Proj}{Proj}
\DeclareMathOperator{\Nef}{Nef}
\DeclareMathOperator{\Ample}{Ample}
\DeclareMathOperator{\BigD}{Big}
\DeclareMathOperator{\Eff}{Eff}
\DeclareMathOperator{\Mov}{Mov}

\DeclareMathOperator{\BirMori}{BirMori}
\DeclareMathOperator{\Ex}{Ex}
\DeclareMathOperator{\codim}{codim}
\DeclareMathOperator{\Cox}{Cox}
\DeclareMathOperator{\gon}{gon}

\DeclareMathOperator{\Cr}{Cr} 
\DeclareMathOperator{\diag}{diag}
\DeclareMathOperator{\covgon}{cov.gon}
\DeclareMathOperator{\conngon}{conn.gon}
\DeclareMathOperator{\IntMov}{IntMov}
\let\div\relax
\DeclareMathOperator{\div}{div}

\newcommand{\tr}[1]{\vphantom{#1}^{t}\!#1}
\newcommand{\interior}[1]{\smash{\mathring{#1}}}
\newcommand{\parent}[1]{\textup{(}#1\textup{)}}

\theoremstyle{plain}
\newtheorem{theorem}{Theorem}[section]
\newtheorem{corollary}[theorem]{Corollary}
\newtheorem{proposition}[theorem]{Proposition}
\newtheorem{lemma}[theorem]{Lemma}

\newtheorem{maintheorem}{Theorem}

\theoremstyle{definition}
\newtheorem{definition}[theorem]{Definition}

\newtheorem{example}[theorem]{Example}
\newtheorem{remark}[theorem]{Remark}
\newtheorem{notation}[theorem]{Notation}
\newtheorem{setup}[theorem]{Set-Up}

\usepackage{stmaryrd}
\usepackage[normalem]{ulem}
\title{Quotients of higher dimensional Cremona groups}
\author{J\'er\'emy Blanc, St\'ephane Lamy \& Susanna Zimmermann}
\thanks{The first author acknowledges support by the Swiss National Science Foundation Grant ``Birational transformations of threefolds'' 200020\_178807. 
The second author was partially supported by the UMI-CRM 3457 of the CNRS in Montr\'eal, and by the Labex CIMI. 
The third author was supported by Projet PEPS 2018 "JC/JC" and is supported by the ANR Project FIBALGA ANR-18-CE40-0003-01.}
\address{Departement Mathematik und Informatik,  Universit\"at Basel, Spiegelgasse 1, 4051 Basel, Switzerland}
\email{jeremy.blanc@unibas.ch}
\address{Institut de Math\'ematiques de Toulouse UMR 5219, Universit\'e de Toulouse, UPS F-31062 Toulouse Cedex 9, France}
\email{slamy@math.univ-toulouse.fr}
\address{Univ Angers, CNRS, LAREMA, SFR MATHSTIC, F-49000 Angers, France}
\email{susanna.zimmermann@univ-angers.fr}
\date{\today}

\keywords{Cremona groups; normal subgroups; conic bundles; Sarkisov links; BAB conjecture}
\subjclass[2010]{14E07, 14E30, 20F05; 20L05, 14J45, 14E05}

\begin{document}

\begin{abstract}
We study large groups of birational transformations $\Bir(X)$, where $X$ is a variety of dimension at least $3$, defined over $\C$ or a subfield of $\C$.
Two prominent cases are when $X$ is the projective space $\p^n$, in which case $\Bir(X)$ is the Cremona group of rank~$n$, or when $X \subset \p^{n+1}$ is a smooth cubic hypersurface.
In both cases, and more generally when $X$ is birational to a conic bundle, we produce infinitely many distinct group homomorphisms from $\Bir(X)$ to $\Z/2$, showing in particular that the group $\Bir(X)$ is not perfect and thus not simple.
As a consequence we also obtain that the Cremona group of rank~$n \ge 3$ is not generated by linear and Jonqui\`eres elements.
\end{abstract}

\maketitle

\begin{small}
\tableofcontents
\end{small}

\section{Introduction}

\subsection{Higher rank Cremona groups}

The \emph{Cremona group of rank~$n$}, denoted by $\Bir_\k(\p^n)$, or simply $\Bir(\p^n)$ when the ground field $\k$ is implicit, is the group of birational transformations of the projective space.

The classical case is $n=2$, where the group is already quite complicated but is now well described, at least when $\k$ is algebraically closed. 
In this case the Noether-Castelnuovo Theorem \cite{Cas, Carraminana} asserts that $\Bir(\p^2)$ is generated by $\Aut(\p^2)$ and a single standard quadratic transformation. 
This fundamental result, together with the strong factorisation of birational maps between surfaces helps to have a good understanding of the group.

The dimension $n\ge 3$ is more difficult, as we do not have any analogue of the Noether-Castelnuovo Theorem (see \S\ref{sec:Generators} for more details) and also no strong factorisation. 
Here is an extract from the article ``Cremona group'' in the Encyclopedia of Mathematics, written by V.~Iskovskikh in 1982 (and translated in 1987) (who uses the notation $\Cr(\p^n_\k)$ for the Cremona group):

\begin{quote}
{\it 
One of the most difficult problems in birational geometry is that of describing the structure of the group $\Cr(\p^3_\k)$, which is no longer generated by the quadratic transformations. Almost all literature on Cremona transformations of three-dimensional space is devoted to concrete examples of such transformations. 
Finally, practically nothing is known about the structure of the Cremona group for spaces of dimension higher than $3$.} \hfill\cite{Encyclo}
\end{quote}

More than thirty years later, there are still very few results about the group structure of $\Bir(\p^n)$ for $n\ge 3$, even if there were exciting recent developments using a wide range of techniques.
After the pioneering work \cite{Demazure} on the algebraic subgroups of rank~$n$ in $\Bir(\p^n)$, we should mention the description of their lattices via $p$-adic methods \cite{CantatXie}, the study of the Jordan property \cite{ProShra}, and the fact that Cremona groups of distinct ranks are non-isomorphic \cite{CantatMorphisms}.

For $n=3$, there is also a classification of the connected algebraic subgroups \cite{Ume,BFT}, and partial classification of  finite subgroups \cite{Pro2011,Pro2012,Pro2014}.
There are also numerous articles devoted to the study of particular classes of examples of elements in $\Bir(\p^n)$, especially for $n$ small (we do not attempt to start a list here, as it would always be very far from exhaustive).

The question of the non-simplicity of Cremona groups of higher rank was up to now left open. 
Using modern tools such as the Minimal model programme and factorisation into Sarkisov links, we will be able in this text to give new insight on the structure of the Cremona groups $\Bir(\p^n)$ and of its quotients.

\subsection{Normal subgroups}

The question of the non-simplicity of $\Bir(\p^n)$ for each $n\ge 2$ was also mentioned in the article of V.~Iskovskikh in the Encyclopedia:
\begin{quote}
{\it It is not known to date $(1987)$ whether the Cremona group is simple.\hfill\cite{Encyclo}}\end{quote}
The question was in fact asked much earlier,
and is explicitly mentioned in a book by F.~Enriques in 1895:
\begin{quote}
{\it Tuttavia altre questioni d'indole gruppale relative al gruppo Cremona nel piano $($ed a pi\`u forte ragione in $S_n$ $n>2)$ rimangono ancora insolute; ad esempio l'importante questione se il gruppo Cremona contenga alcun sottogruppo invariante $($questione alla quale sembra probabile si debba rispondere negativamente$)$. \\ 
\cite[p.~116]{Enriques}\footnote{``However, other group-theoretic questions related to the Cremona group of the plane (and, even more so, of $\p^n$, $n>2$) remain unsolved; for example, the important question of whether the Cremona group contains any normal subgroup (a question which seems likely to be answered negatively).''}}
\end{quote}

The feeling expressed by F.~Enriques that the Cremona group should be simple was perhaps supported by the analogy with biregular automorphism groups of projective varieties, such as $\Aut(\p^n) = \PGL_{n+1}(\k)$.
In fact in the trivial case of dimension $n = 1$, we have $\Bir(\p^1)=\Aut(\p^1)=\PGL_2(\k)$, which is indeed a simple group when the ground field $\k$ is algebraically closed.
Another evidence in favour of the simplicity of the Cremona groups is that one can endow $\Bir_\C(\p^n)$ with two topologies: the Zariski or the Euclidean one (see \cite{Blanc2010,BF13}), and that in both cases all closed normal subgroups are either trivial or the whole group, as proven in  \cite{Blanc2010} for $n=2$ and generalised in \cite{BZ} to any dimension.

The non-simplicity of $\Bir(\p^2)$ as an abstract group was proven, over an algebraically closed field, by S.~Cantat and the second author \cite{CantatLamy}.
The idea of proof was to apply small cancellation theory to an action of $\Bir(\p^2)$ on a hyperbolic space.
A first instance of roughly the same idea was \cite{Danilov}, in the context of plane polynomial automorphisms (see also \cite{FurterLamy}).
The modern small cancellation machinery as developed in \cite{DGO} allowed A. Lonjou to prove the non simplicity of $\Bir(\p^2)$ over an arbitrary field, and the fact that every countable group is a subgroup of a quotient of $\Bir(\p^2)$ \cite{Lonjou}. 

Another source of normal subgroups for $\Bir(\p^2)$, of a very different nature, was discovered by the third author, when the ground field is $\R$ \cite{Zimmermann}.
In contrast with the case of an algebraically closed field where the Cremona group of rank 2 is a perfect group, she proved that the abelianisation of $\Bir_\R(\p^2)$ is an uncountable direct sum of $\Z/2$.
Here the main idea is to use an explicit presentation by generators and relations.
In fact a presentation of $\Bir(\p^2)$ over an arbitrary perfect field is available since \cite{IKT}, but because they insist in staying inside the group $\Bir(\p^2)$, they obtain very long lists. 
In contrast, if one accepts to consider birational maps between non-isomorphic varieties, the Sarkisov programme provides more tractable lists of generators.
Using this idea together with results of A.-S. Kaloghiros \cite{Kaloghiros}, the existence of abelian quotients for $\Bir(\p^2)$ was extended to the case of many non-closed perfect fields by the second and third authors \cite{LZ17}.

The present paper is a further extension in this direction, this time in arbitrary dimension, and over any ground field $\k$ which is a subfield  of $\C$.
Our first result is the following: 

\begin{maintheorem}\label{TheoremBirPnnotsimple}
For each subfield $\k\subseteq\C$ and each $n\ge 3$, there is a group homomorphism
\[
\Bir_\k(\p^n)\tto \bigoplus_{I} \Z/2
 \]
where the indexing set $I$ has the same cardinality as $\k$,
and such that the restriction to the subgroup of birational dilatations given locally by
\[\left\{(x_1,\dots,x_n)\mapsrat \left(x_1\alpha(x_2,\cdots,x_n),x_2,\dots,x_n\right)\mid \alpha\in\k(x_2,\dots,x_{n})^*\right\} \]
is surjective.
In particular, the Cremona group $\Bir_\k(\p^n)$ is not perfect and thus not simple.
\end{maintheorem}

We give below a few immediate comments, and a quick preview of the rest of the introduction where we will present several statements that generalise or complement Theorem \ref{TheoremBirPnnotsimple} in different directions. 
 
First we emphasise that this result contrasts with the situation in dimension 2 (over $\C$).
Indeed, as $\Bir_\C(\p^2)$ is generated by the simple group $\Aut(\p^2)=\PGL_3(\C)$ and one quadratic map birationally conjugated to a linear map, every non-trivial quotient of $\Bir_\C(\p^2)$ is non-abelian and uncountable. 

Another intriguing point at first sight is the indexing set $I$.
We shall be more precise later, but the reader should think of $I$ as a kind of moduli space of some irreducible varieties of dimension $n-2$.
Indeed to construct the group homomorphism we will see $\p^n$ as being birational to a $\p^1$-bundle over $\p^{n-1}$, and each factor $\Z/2$ is related to the choice of a general hypersurface in $\p^{n-1}$ of sufficiently high degree, up to some equivalence. 
Observe that in dimension $n = 2$ an irreducible hypersurface of $\p^{n-1}$ is just a point, and so cannot be of high degree, at least over~$\C$: this explains why the homomorphism of Theorem \ref{TheoremBirPnnotsimple} becomes trivial in the case of $\Bir_\C(\p^2)$.

The next natural question is to understand the kernel of the group homomorphism.
As will soon become clear, it turns out that $\Aut(\p^n)=\PGL_{n+1}(\k)$ is contained in the kernel.  
This implies that the normal subgroup generated by $\Aut(\p^n)$ and any finite subset of elements in $\Bir_\k(\p^n)$ is proper. 
Theorem \ref{TheoremTame} below will be a stronger version of this fact. 
We also point out that because of the already mentioned result from \cite{BZ}, the kernel of all our group homomorphisms is dense in $\Bir(\p^n)$, so the group homomorphisms $\Bir(\p^n)\to \Z/2$ that we construct are not continuous (when putting the non-trivial topology on $\Z/2$).

One can also ask about the possibility to get a homomorphism to a free product of $\Z/2$, instead of a direct sum. 
We will see that is is indeed possible, and  is related to the existence of many conic bundle models for $\p^n$ which are not pairwise square birational. See Theorems \ref{TheoremBirMori} and \ref{TheoremManyCB} below.

Finally, one can ask about replacing $\p^n$ by a nonrational variety.
In this direction, we will prove the following result about the group $\Bir(X)$ of birational transformations of a conic bundle $X/B$.

\begin{maintheorem} \label{TheoremBirXBnotsimple}
Let $B\subseteq \p^m$ be a smooth projective complex variety with $\dim B\ge2$, $P\to \p^m$ a decomposable $\p^2$-bundle $($projectivisation of a decomposable rank $3$ vector bundle$)$ and $X\subset P$ a smooth closed subvariety such that the projection to $\p^m$ gives a conic bundle $\eta\colon X\to B$.
Then there exists a group homomorphism
\[\Bir(X)\tto \bigoplus_\Z \Z/2,\]
the restriction of which to $\Bir(X/B)=\{\phi\in \Bir(X)\mid \eta\circ\phi=\eta\}$ is surjective.

Moreover, if there exists a subfield $\k\subseteq \C$ over which $X,B$ and $\eta$ are defined, the image of elements of $\Bir(X/B)$ defined over $\k$ is also infinite.
\end{maintheorem}

Theorem~\ref{TheoremBirXBnotsimple} applies to any product $X = \p^1 \times B$, to smooth cubic hypersurfaces $X  \subseteq \p^{n+1}$  (see Section~\ref{Sec:CubicVar} and in particular Corollary~\ref{cor:SmoothCubicBirXZ2} and Proposition~\ref{Prop:Cubicfreeproduct}), and to many other varieties of dimension $n \ge 3$ which are very far from being rational (see for instance \cite[Theorem 3]{KollarCB} and \cite[Theorem 1.1 and Corollary 1.2]{AhmOka}).
Of course it also includes the case of $X = \p^1 \times \p^{n-1}$ which is birational to $\p^n$, but observe that Theorem~\ref{TheoremBirPnnotsimple} is slightly stronger in this case, since there the set indexing the direct sum has the same cardinality as the ground field, and also because we can give an explicit subgroup, easy to describe, whose image is surjective.  

\subsection{Generators}\label{sec:Generators} 

As already mentioned, the Noether-Castelnuovo theorem provides simple generators of $\Bir(\p^2)$ when $\k$ is algebraically closed. 
Using Sarkisov links, there are also explicit (long) lists of generators of $\Bir(\p^2)$ for each field $\k$ of characteristic zero or more generally for each perfect field $\k$ \cite{Isk1991,Isk1996}. 
In dimension $n\ge 3$, we do not have a complete list of all Sarkisov links and thus are far from having an explicit list of generators for $\Bir(\p^n)$. 
The lack of an analogue to the Noether-Castelnuovo Theorem for $\Bir(\p^n)$ and the question of finding good generators was already cited in the article of the Encyclopedia above, in \cite[Question 1.6]{HMcK}, and also in the book of Enriques:

\begin{quote}
{\it Questo teorema non \`e estendibile senz'altro allo $S_n$ dove $n>2$; resta quindi insoluta la questione capitale di assegnare le pi\`u sempilici trasformazioni generatrici dell'intiero gruppo Cremona in $S_n$ per $n>2$.} \hfill\cite[p. 115]{Enriques}\footnote{``This theorem can not be easily extended to $\p^n$  where $n>2$; therefore, the main question of finding the most simple generating transformations of the entire Cremona group of $\p^n$ for $n>2$ remains open.''}
\end{quote}

A classical result, due to H. Hudson and I. Pan \cite{Hudson,Pan}, says that $\Bir(\p^n)$, for $n\ge 3$, is not generated by $\Aut(\p^n)$ and finitely many elements, or more generally by any set of elements of $\Bir(\p^n)$ of bounded degree. 
The reason is that one needs at least, for each irreducible variety $\Gamma$ of dimension $n-2$, one birational map that contracts a hypersurface birational to $\p^1\times \Gamma$. 
These contractions can be realised in $\Bir(\p^n)$ by \emph{Jonqui\`eres elements}, i.e.~elements that preserve a family of lines through a given point, which form a subgroup 
\[
\PGL_2(\k(x_2, \dots, x_n)) \rtimes \Bir(\p^{n-1}) \subseteq \Bir(\p^n).
\]

Hence, it is natural to ask whether the group $\Bir(\p^n)$ is generated by $\Aut(\p^n)$ and by Jonqui\`eres elements (a question for instance asked in \cite{PanSimis}).

We answer this question by the negative, in the following stronger form:

\begin{maintheorem}\label{TheoremTame} 
Let $\k$ be a subfield of $\C$, and $n \ge 3$. 
Let $S$ be a set of elements in the Cremona group $\Bir_\k(\p^n)$ that has cardinality smaller than the one of $\k$ $($for example $S$ finite, or $S$ countable if $\k$ is uncountable$)$, and let $G \subseteq \Bir_\k(\p^n)$ be the subgroup generated by $\Aut_\k(\p^n)$, by all Jonqui\`eres elements and by $S$.

Then, $G$ is contained in the kernel of a surjective group homomorphism 
\[\Bir_\k(\p^n) \tto \Z/2.\]
In particular $G$ is a proper subgroup of  $\Bir_\k(\p^n)$, and the same is true for the normal subgroup generated by $G$.
\end{maintheorem} 

It is interesting to make a parallel between this statement and the classical Tame Problem in the context of the affine Cremona group $\Aut(\A^n)$, or group of polynomial automorphisms.
This is one of the ``challenging problems'' on the affine spaces, described by H. Kraft in the Bourbaki seminar~\cite{KraftBourbaki}.
Recall that the tame subgroup $\Tame(\A^n) \subseteq \Aut(\A^n)$ is defined as the subgroup generated by affine automorphisms and by the subgroup of elementary automorphisms of the form $(x_1, \dots, x_n) \mapsto (ax_1 + P(x_2, \dots, x_n), x_2, \dots, x_n)$. 
This elementary subgroup is an analogue of the $\PGL_2(\k(x_{2}, \dots, x_n))$ factor in the Jonqui\`eres group, and of course the affine group is $\PGL_{n+1}(\k) \cap \Aut(\A^n)$.
The Tame Problem asks whether the inclusion $\Tame(\A^n) \subseteq \Aut(\A^n)$ is strict in dimension $n \ge 3$. 
It was solved in dimension $3$ over a field of characteristic zero in \cite{ShestakovUmirbaev}, and remains an open problem otherwise.

On the one hand, one could say that our Theorem~\ref{TheoremTame} is much stronger, since we consider the \emph{normal} subgroup generated by these elements, and we allow some extra generators.
It is not known (even if not very likely) whether one can generate $\Aut(\A^3)$ with linear automorphisms, elementary automorphisms and one single automorphism, and not even whether the normal subgroup generated by these is the whole group $\Aut(\A^3)$ (this last statement, even without the extra automorphism, seems more plausible). 

On the other hand, even in dimension 3 we should stress that Theorem~\ref{TheoremTame} does not recover a solution to the Tame Problem. 
Indeed, it seems plausible that the whole group $\Aut(\A^n)$ lies in the kernel of the group homomorphism to $\Z/2$ of Theorem~\ref{TheoremTame}. 
In fact, every element of $\Bir(\p^n)$ that admits a decomposition into Sarkisov links that contract only rational varieties (or more generally varieties birational to $\p^2 \times B$ for some variety $B$ of dimension $n-3$) is in the kernel of all our group homomorphisms (all are given by the construction of Theorem~\ref{TheoremBirMori} below), and it seems natural to expect that elements of $\Aut(\A^n)$ are of this type, but we leave this as an open question.
In fact we are not aware of any element of $\Aut(\A^3)$ which has been proved to lie outside the group generated, in $\Bir(\p^3)$, by linear and Jonqui\`eres maps: see \cite[Proposition 6.8]{BlancHeden} for the case of the Nagata automorphism, which can be generalised to any other automorphism given by a $\mathbb{G}_a$ action, as all algebraic subgroups of $\Bir(\p^3)$ isomorphic to $\mathbb{G}_a$ are conjugate \cite{BFT2}. 

\subsection{Overwiew of the strategy}

To give an idea of the way we construct group homomorphims from birational groups to $\Z/2$, first consider as a toy model the signature homomorphism on the symmetric group $S_n$.
One possible proof of the existence of the signature goes as follows.
A presentation by generators and relations of $S_n$ is
\[
S_n = \bigl\langle \tau_i = (i \, i+1) \mid \tau_i^2 = 1, (\tau_i \tau_{i+1})^3 = 1, (\tau_i \tau_j)^2 = 1 \bigr\rangle
\]
where the relations are for all $i = 1, \dots , n-1$ and all $n \ge j \ge i + 2$.
Then by sending each $\tau_i$ to the generator of $\Z/2$, one gets a group homomorphism because each relation has even length and so is sent to the trivial element.

Now we would like to apply the same strategy for a group $\Bir(Z)$ of birational transformations: use a presentation by generators and relations, send some of the generators to the generator of $\Z/2$, and check that all relations are sent to the trivial element.
The trick is that we do not apply this strategy directly to $\Bir(Z)$, but to a larger groupoid containing $\Bir(Z)$, where we are able to produce a nice presentation (as a groupoid) by generators and relations.  

To define this groupoid, first recall that by the Minimal model programme, every variety $Z$ which is covered by rational curves is birational to a Mori fibre space, and every birational map between two Mori fibre spaces is a composition of simple birational maps, called \emph{Sarkisov links} (see Definition~\ref{def:sarkisovLink}).
We are also able to give a description of the relations between Sarkisov links, in terms of \emph{elementary relations} (see Definition~\ref{def:elementary relation} and Theorem~\ref{thm:sarkisov}).
We associate to $Z$ the groupoid $\BirMori(Z)$ of all birational maps between Mori fibre spaces birational to $Z$.
The main idea is that even if we are primarily interested in describing homomorphisms from the group $\Bir(Z)$ to $\Z/2$, it turns out to be easier to first define such a homomorphism on the larger groupoid $\BirMori(Z)$, and then restrict to $\Bir(Z)$. 

\subsection{Construction of the groupoid homomorphism}

Now we describe Theorem~\ref{TheoremBirMori}, our main technical result, which is the base for all other theorems in this paper.

We concentrate on some special Sarkisov links, called \emph{Sarkisov links of conic bundles of type \II} (see Definitions~\ref{def:sarkisovLink} and~\ref{def:4TypesOfLinks}).
To each such link, we associate a marked conic bundle, which is a pair $(X/B,\Gamma)$, where $X/B$ is a conic bundle (a terminal Mori fibre space with $\dim B=\dim X-1$) and $\Gamma\subset B$ is an irreducible hypersurface (see Definition~\ref{def:markedConic} and Lemma~\ref{lem:SarkiIIConic}). 
We also define a natural equivalence relation between marked conic bundles (Definition~\ref{def:markedConic}).

For each variety $Z$, we denote by $\CB(Z)$ the set of equivalence classes of conic bundles $X/B$ with $X$ birational to $Z$, and for each class of conic bundles $C\in \CB(Z)$ we denote by $\MC(C)$ the set of equivalence classes of marked conic bundles $(X/B,\Gamma)$, where $C$ is the class of $X/B$.

The Sarkisov programme is established in every dimension \cite{HMcK} and relations among them are described in \cite{Kaloghiros}. 
Inspired by the latter, we define {\em rank~$r$ fibrations} $X/B$ (see Definition~\ref{def:rankFibration}); rank~$1$ fibrations are terminal Mori fibre spaces and rank~$2$ fibrations dominate Sarkisov links (see Lemma~\ref{lem:rank 2 and link}).
We prove that the relations among Sarkisov links are generated by {\em elementary relations} (Definition~\ref{def:elementary relation}), which we define as relations dominated by rank~$3$ fibrations (see Theorem~\ref{thm:sarkisov}).

We associate to each such Sarkisov link $\chi$ an integer $\covgon(\chi)$ that measures the degree of irrationality of the base locus of $\chi$ (see \S\ref{sec:gonality}). 
The BAB conjecture, proven in \cite{BirkarA, BirkarS}, tells us that the set of weak Fano terminal varieties of dimension $n$ form a bounded family and the degree of their images by a (universal) multiple of the anticanonical system is bounded by a (universal) integer $d$ (see Proposition~\ref{pro:BAB}). 
As a consequence, we show that any Sarkisov link $\chi$ of conic bundles of type \II appearing in an elementary relation over a base of small dimension satisfies $\covgon(\chi)\leq d$ (see Proposition~\ref{pro:boundOnGenus}). 
This and the description of the elementary relations over a base of maximal dimension and including a Sarkisov link of conic bundles of type \II (Proposition~\ref{pro:X4X3X2X1=id}) allows us to prove the following statement in \S\ref{sec:proofBirMori}. 
(Here we use the notation $\bigast$ for a free product of groups).

\begin{maintheorem}  \label{TheoremBirMori}
Let $n\ge3$. 
There is an integer $d > 1$ depending only on $n$, such that for every conic bundle $X/B$, where $X$ is a terminal variety of dimension $n$, we have a groupoid homomorphism 
\[\BirMori(X) \to \bigast_{C\in \CB(X)} \left(\bigoplus_{\MC(C)}\Z/2\right)\]
that sends each Sarkisov link of conic bundles $\chi$ of type \II with $\covgon(\chi) > \max\{ d, 8\conngon(X) \}$ onto the generator indexed by its associated marked conic bundle, and all other Sarkisov links and all automorphisms of Mori fibre spaces birational to $X$ onto zero. 

Moreover it restricts to group homomorphisms
\begin{align*}
\Bir(X) \to \bigast_{C\in \CB(X)} \left(\bigoplus_{\MC(C)}\Z/2\right)
,&& \Bir(X/B)\to \bigoplus_{\MC(X/B)} \Z/2.
\end{align*}
\end{maintheorem}

In order to deduce Theorem~\ref{TheoremBirPnnotsimple}, we study the image of the group homomorphisms from $\Bir(X)$ and $\Bir(X/B)$ provided by Theorem~\ref{TheoremBirMori}, for some conic bundle $X/B$.
In particular, we must check that these restrictions are not the trivial morphism.
We give a criterion to compute the image in \S\ref{sec:criterion}.
We apply this criterion to show that the image is very large if the generic fibre of $X/B$ is $\p^1$ (or equivalently if $X/B$ has a rational section, or is equivalent to $( \p^1\times B)/B$). 
This is done in \S\ref{sec:trivialBundle} and allows us to prove Theorem~\ref{TheoremBirPnnotsimple}. 
Then in \S\ref{sec:nontrivialBundle} we study the more delicate case where the generic fibre $X/B$ is not $\p^1$ (or equivalently if $X/B$ has no rational section), and show that for each conic bundle $X/B$, the image of $\Bir(X/B)$ by the group homomorphism of Theorem~\ref{TheoremBirMori} contains an infinite direct sum of $\Z/2$ (Proposition~\ref{pro:largeimageAnyCB}). This allows to prove  Theorem~\ref{TheoremBirXBnotsimple}. 

Finally, let us mention that
\cite{Zimmermann,LZ17,Schneider} construct homomorphisms from plane Cremona groups over certain non algebraically closed perfect fields, which we can see as two-dimensional special cases of the homomorphisms from Theorem~\ref{TheoremBirMori}.
The homomorphism in \cite{LZ17,Schneider} is in fact constructed with the same strategy as the one employed here, replacing the covering gonality with the size of Galois orbits, while \cite{Zimmermann} works with generators and relations inside $\Bir_\R(\p^2)$.

\subsection{Non-equivalent conic bundle structures}
 
Coming back to the case of $\p^n$, we study the free product structure appearing in Theorem \ref{TheoremBirMori}.
We want to prove that the indexing set $\CB(\p^n)$ is large. 
This is equivalent to the question of existence of many non-equivalent conic bundle structures on $\p^n$:
Indeed it follows from our description of relations (Proposition~\ref{pro:X4X3X2X1=id}) that two Sarkisov links of sufficiently high covering gonality on non-equivalent conic bundles cannot be part of a same elementary relation, as reflected also in Theorem~\ref{TheoremBirMori}.
Using conic bundles over $\p^2$ with discriminant an elliptic curve, we manage to produce such examples, and we get the following.

\begin{maintheorem}  \label{TheoremManyCB}
Let $n\ge3$ and let $\k\subseteq \C$ be a subfield. There is a surjective group homomorphism 
\[\Bir_\k(\p^n) \tto \bigast_{J} \Z/2,\]
where the indexing set $J$ has the same cardinality as $\k$.
In particular, every group generated by a set of involutions with cardinality smaller or equal than $\lvert \k\rvert$ is a quotient of $\Bir_\k(\p^n)$. 
Moreover, the group homomorphism that we construct admits a section, so $\Bir_\k(\p^n)$ is a semi-direct product with one factor being a free product.  
\end{maintheorem}

A first consequence is Theorem \ref{TheoremTame}. 
Other complements are given in Section~\ref{sec:Complements}.

First we get the SQ-universality of $\Bir_\k(\p^n)$, meaning that any countable group is a subgroup of a quotient of $\Bir_\k(\p^n)$.
But in fact, many natural subgroups are quotients of $\Bir_\k(\p^n)$, with no need to passing to a subgroup: this includes dihedral and symmetric groups, linear groups, and the Cremona group of rank 2 (see \S\ref{sec:SQ}). 

Another consequence of our results is that the group $\Bir_\k(\p^n)$ is not hopfian if it is generated by involutions, for each subfield $\k\subseteq \C$ and each $n\ge 3$ (Corollary~\ref{Cor:Hopfian}). This is another difference with the dimension $2$, as $\Bir_\C(\p^2)$ is hopfian and generated by involutions (see \S\ref{sec:Hopfian}).

All our results hold over any field abstractly isomorphic to a subfield of $\C$ (\S\ref{sec:field}). 
This is the case of most field of characteristic zero that are encountered in algebraic geometry: 
for instance, any field of rational functions of any algebraic variety defined over a subfield of $\C$. 

Another feature of the Cremona groups in higher dimension is that the group $\Bir_\C(\p^n)$ is a free product of uncountably many distinct subgroups, amalgamated over the intersection of the subgroups, which is the same for any two subgroups. This strong version of an amalgamated product (Theorem~\ref{thm:AmalgamatedProduct}) is again very different from $\Bir_\C(\p^2)$ (which is not a non-trivial amalgam, as already explained) and generalises to other varieties as soon as they have two non-equivalent conic bundle structures. 
Again this result can be generalised to subfields of $\C$.

Theorem~\ref{thm:AmalgamatedProduct} implies that $\Bir(\p^n)$ acts non-trivially on a tree. More generally, for each conic bundle $X/B$, we provide a natural action of $\Bir(X)$ on a graph constructed from rank~$r$ fibrations birational to $X$ (see \S\ref{sec:graph}).

\subsection*{Aknowledgements}

We thank 
Hamid Ahmadinezhad, Marcello Bernardara, Cau\-cher Birkar, Christian B\"ohnig, Hans-Christian Graf von Bothmer, Serge Cantat, Ivan Chel\-tsov, Tom Ducat, Andrea Fanelli, Enrica Floris, Jean-Philippe Furter, Marat Gizatullin, Philipp Habegger, Yang He, Anne-Sophie Kaloghiros, Vladimir Lazi\'c, Zsolt Patakfalvi, Yuri Prokhorov, Miles Reid and Christian Urech
for interesting discussions related to this project.

\section{Preliminaries}
\label{sec:preliminaries}
Unless explicitly stated otherwise, all ambient varieties are assumed to be projective, irreducible, reduced and defined over the field $\C$ of complex numbers.

This restriction on the ground field comes from the fact that this is the setting of many references that we use, such as \cite{BCHM,HMcK,Kaloghiros,KKL}.
It seems to be folklore that all the results in these papers are also valid over any algebraically closed field of characteristic zero, but we let the reader take full responsibility if he is willing to deduce that our results automatically hold over such a field.
However, in Sections \ref{sec:surjectivity} and \ref{sec:freeProduct}, see also \S\ref{sec:field}, we will show how to work over fields that can be embedded in $\C$.

General references for this section are \cite{KM, LazarsfeldI, BCHM}.

\subsection{Divisors and curves}
\label{sec:divisors}

Let $X$ be a normal variety, $\Div(X)$ the group of Cartier divisors, and $\Pic(X) = \Div(X)/\sim$ the Picard group of divisors modulo linear equivalence. 
The \emph{N\'eron-Severi} space $N^1(X) = \Div(X) \otimes \R / \equiv$ is the space of $\R$-divisors modulo numerical equivalence.
This is a finite-dimensional vector space whose dimension $\rho(X)$ is called the \emph{Picard rank} of $X$. 
We denote $N_1(X)$ the dual space of 1-cycles with real coefficients modulo numerical equivalence. 
We have a perfect pairing $N^1(X) \times N_1(X) \to \R$ induced by intersection.
If we need to work with coefficients in $\Q$ we will use notation such as $N^1(X)_\Q := \Div(X) \otimes \Q / \equiv$ or $\Pic(X)_\Q := \Pic(X) \otimes \Q$. 
We say that a Weil divisor $D$ on $X$ is \emph{$\Q$-Cartier} if $mD$ is Cartier for some integer $m >0$.
The variety $X$ is \emph{$\Q$-factorial} if all Weil divisors on $X$ are $\Q$-Cartier.
An element in $\Div(X) \otimes \Q$ is called a $\Q$-divisor.

First we recall a few classical geometric notions attached to a $\Q$-divisor $D$.
Let $m$ be a sufficiently large and divisible integer. 
$D$ is \emph{movable} if the base locus of the linear system $\lvert mD \rvert$ has codimension at least 2.
$D$ is \emph{big} if the map associated with $\lvert mD \rvert$ is birational.
Similarly, $D$ is \emph{semiample} if $\lvert mD \rvert$ is base point free, and $D$ is \emph{ample} if furthermore the associated map is an embedding.
Finally, $D$ is \emph{nef} if for any curve $C$ we have $D \cdot C \ge 0$.

Now we recall how the numerical counterparts of these notions define cones in $N^1(X)$.
The effective cone $\Eff(X) \subseteq N^1(X)$ is the cone generated by effective divisors on $X$.
Its closure $\widebar{\Eff}(X)$ is the cone of \emph{pseudo-effective} classes.
Similarly we denote $\NE(X) \subseteq N_1(X)$ the cone of effective 1-cycles, and $\widebar{\NE}(X)$ its closure. 
By Kleiman's criterion, a divisor $D$ is ample if and only if $D \cdot C >0$ for any 1-cycle $C \in \widebar{\NE}(X)$. 
It follows that the cone $\Ample(X)$ of ample classes is the interior of the closed cone $\Nef(X) \subseteq N^1(X)$ of nef classes.
Similarly, the interior of the pseudo-effective cone $\widebar{\Eff}(X)$ is the big cone $\BigD(X)$:
Indeed a class $D$ is big if and only if $D \equiv A + E$ with $A$ ample and $E$ effective.
A class is semiample if it is the pull-back of an ample class by a morphism. 
Finally the \emph{movable} cone $\widebar{\Mov}(X)$ is the closure of the cone spanned by movable divisors, and we will denote by $\IntMov(X)$ its interior.

One should keep in mind the following inclusions between all these cones:
\[\begin{tikzcd}[column sep=small, row sep=small]
\Ample(X) \ar[r,phantom,"\subseteq"] &
\text{Semiample}(X) \ar[r,phantom,"\subseteq"] & 
\Nef(X) \ar[r,phantom,"\subseteq"] \ar[d,phantom,"\verteq"] &
\widebar{\Mov}(X) \ar[r,phantom,"\subseteq"] & 
\widebar{\Eff}(X) \ar[d,phantom,"\verteq"]\\
&& \widebar{\Ample}(X) && \widebar{\BigD}(X)
\end{tikzcd}\]
We say that a 1-cycle $C \in \widebar{\NE}(X)$ is \emph{extremal} if any equality $C = C_1 + C_2$ inside $\widebar{\NE}(X)$ implies that $C, C_1, C_2$ are proportional.

\subsection{Maps} 

Let $\pi\colon X \to Y$ be a surjective morphism between normal varieties.
We shall also denote $X/Y$ such a situation.
The \emph{relative Picard group} is the quotient  $\Pic(X/Y):=\Pic(X)/\pi^*\Pic(Y)$.

We say that a curve $C \subseteq X$ is \emph{contracted} by $\pi$ if $\pi(C)$ is a point.
The subsets $\NE(X/Y) \subseteq N_1(X/Y) \subseteq N_1(X) $ are respectively the cone and the subspace generated by curves contracted by $\pi$.
The relative N\'eron-Severi space $N^1(X/Y)$ is the quotient of $N^1(X)$ by the orthogonal of $N_1(X/Y)$.
The dimension $\rho(X/Y)$ of $N^1(X/Y)$, or equivalently $N_1(X/Y)$, is the \emph{relative Picard rank} of $\pi$.
If $\pi$ has connected fibres, then $\rho(X/Y) = 0$ if and only if $\pi$ is an isomorphism, because a bijective morphism between normal varieties is an isomorphism.

We denote by $\Eff(X/Y)$, $\Nef(X/Y)$, $\Ample(X/Y)$, $\BigD(X/Y)$, $\widebar{\Mov}(X/Y)$ the images of the corresponding cones of $N^1(X)$ in the quotient $N^1(X/Y)$.
If $D \in N^1(X)$ is a class that projects to an element in $\Nef(X/Y)$, we says that $D$ is $\pi$-nef.
Equivalently, $D$ is $\pi$-nef if $D \cdot C \ge 0$ for any $C \in \NE(X/Y)$.
Similarly, we define the notion of $\pi$-ample, $\pi$-big, $\pi$-effective.
In particular a class $D$ is $\pi$-ample if $D \cdot C > 0$ for any $C \in \widebar{\NE}(X/Y)$.

Geometrically, a $\Q$-divisor $D$ is $\pi$-ample if the restriction of $D$ to each fibre is ample, and $D$ is \emph{$\pi$-big} if the restriction of $D$ to the generic fibre is big.
We have the following characterisation for this last notion:

\begin{lemma}[{\cite[Lemma 3.23]{KM}}]
\label{lem:piBig}
Let $\pi\colon X \to Y$ be a surjective morphism between normal varieties.
A $\Q$-divisor $D$ on $X$ is $\pi$-big if and only if we can write $D$ as a sum
\[
D = \pi\text{-ample} + \text{effective}.
\]
\end{lemma}

When the morphism $\pi\colon X \to Y$ is birational, the \emph{exceptional locus} $\Ex(\pi)$ is the set covered by all contracted curves.
Assume moreover that $\rho(X/Y) = 1$, and that $X$ is $\Q$-factorial.
Then we are in one of the following situations \cite[Prop 2.5]{KM}: either $\Ex(\pi)$ is a prime divisor, and we say that $\pi$ is a \emph{divisorial contraction}, or $\Ex(\pi)$  has codimension at least 2 in $X$, and we say that $\pi$ is a \emph{small contraction}. 
In this case, $Y$ is not $\Q$-factorial.

Given three normal varieties $X,Y,W$ together with surjective morphisms $X/W$, $Y/W$, we say that $\phi\colon X\rat Y$ is a \emph{rational map over $W$} if we have a commutative diagram
\[
\begin{tikzcd}[link]
X \ar[rd] \ar[rr,dashed,"\phi"] && Y\ar[ld] \\
& W
\end{tikzcd}
\]

Now let $\phi\colon X \rat Y$ be a birational map.
Any Weil divisor $D$ on $X$ is sent to a well-defined cycle $\phi(D)$ on $Y$, and by removing all components of codimension $\ge 2$ we obtain a well-defined divisor $\phi_* D$: one says that $\phi$ induces a map in codimension~1.
If $\codim \phi(D) \ge 2$ (and so $\phi_* D = 0$), we say that $\phi$ contracts the divisor $D$.
A \emph{birational contraction} is a birational map such that the inverse does not contract any divisor, or equivalently a birational map which is surjective in codimension~$1$.
A \emph{pseudo-isomorphism} is a birational map which is an isomorphism in codimension~$1$.
Birational morphisms and pseudo-isomorphisms (and compositions of those) are examples of birational contractions.

We use a dashed arrow $\rat$ to denote a rational (or birational) map, a plain arrow $\to$ for a morphism, and a dotted arrow $\ps$, or simply a dotted line $\dottedline$, to indicate a pseudo-isomorphism.

We denote by $\Bir(X)$ the group of birational transformations of $X$. 
Given a surjective morphism $\eta\colon X \to B$, we denote by $\Bir(X/B)$ the subgroup of $\Bir(X)$ consisting of all birational transformations over $B$, i.e.
\[
\Bir(X/B) := \{\phi\in \Bir(X)\mid \eta \circ \phi=\eta\} \subseteq \Bir(X).
\]

\subsection{Mori dream spaces and Cox sheaves} 
\label{sec:moriDream}

We shall use a relative version of the usual definition of Mori dream space (compare with \cite[Definition 2.2]{KKL}). Before giving the definition we recall the following notions.

Let $\pi\colon X\to Y$ be a surjective morphism, and $\Fl$ a sheaf on $X$.
The \emph{higher direct images} of $\Fl$ are the sheaves $R^i\pi_* \Fl$, $i \ge 0$, which are defined on each affine subset $U \subset Y$ as  $R^i\pi_* \Fl(U) = H^i(\pi^{-1}(U),\Fl)$.

We say that a normal variety $Y$ has \emph{rational singularities} if for some (hence any) desingularisation $\pi\colon X \to Y$, we have $R^i\pi_* \Ol_X = 0$ for all $i > 0$.

Recall also that a variety is \emph{rationally connected} if any two general points are contained in a rational curve (see \cite[IV.3]{Kollar_rational}). 

\begin{definition}
\label{def:moriDream}
Let $\eta\colon X\to B$ be a surjective morphism between normal varieties.
We say that $X/B$ is a \emph{Mori dream space} if the following conditions hold:
\begin{enumerate}[(MD1)]
\item \label{dream:Q_factorial} $X$ is $\Q$-factorial, and both $X,B$ have rational singularities.
\item \label{dream:rat_connected} A general fibre of $\eta$ is rationally connected and has rational singularities. 
\item \label{dream:nef} $\Nef(X/B)$ is the convex cone generated by finitely many semiample divisors; 
\item \label{dream:mov}
There exist finitely many pseudo-isomorphisms $f_i\colon X\ps X_i$ over $B$,  such that each $X_i$ is a $\Q$-factorial variety satisfying \ref{dream:nef}, and 
\[\widebar{\Mov}(X/B)=\bigcup f_i^*(\Nef(X_i/B)).\]
\end{enumerate}
\end{definition}

\begin{lemma}\label{lem:N1=Pic}
Let $\eta\colon X\to B$ be a surjective morphism between normal varieties, and $F$ a general fibre.
Assume that $X$ and $B$ have rational singularities, and assume:
\begin{enumerate}[$(i)$]
\item \label{N1=Pic:F rat connected}
$F$ is rationally connected and has rational singularities.
\end{enumerate}
Then the following properties hold true:
\begin{enumerate}[$(i)$,resume]
\item \label{N1=Pic:H1 F}
$H^i(F,\Ol_F) = 0$ for all $i > 0$;
\item \label{N1=Pic:R}
$\eta_* \Ol_X = \Ol_B$ and $R^i\eta_* \Ol_X = 0$ for all $i > 0$; 
\item \label{N1=Pic:H1 U}
$H^1(\eta^{-1}(U), \Ol_{\eta^{-1}(U)}) = 0$ for each affine open set $U \subset B$;
\item \label{N1=Pic:Pic XB}
$\Pic(X/B)_\Q = N^1(X/B)_\Q$.
\end{enumerate}
\end{lemma}

\begin{remark}
Condition \ref{N1=Pic:F rat connected} from Lemma \ref{lem:N1=Pic} is our condition \ref{dream:rat_connected}.
The lemma implies that we would obtain a more general definition replacing \ref{dream:rat_connected} by Condition \ref{N1=Pic:H1 U}, which is the choice of \cite{BCHM}, or by Condition \ref{N1=Pic:Pic XB}, which is a relative version of the choice made in \cite{KKL}.
However our more restrictive definition suits to our purpose and seems easier to check in practice.
\end{remark}

\begin{proof}
\ref{N1=Pic:F rat connected} $\Longrightarrow$ \ref{N1=Pic:H1 F}.
Consider a resolution of singularities $\pi\colon \hat F\rightarrow F$.
Since $F$ has rational singularities, we have $R^i\pi_*\Ol_{\hat F}=0$ for $i>0$. 
Then \cite[III, Ex.8.1]{Hartshorne} implies that $H^i(\hat F,\Ol_{\hat F})\simeq H^i(F,\pi_*\Ol_{\hat F})=H^i(F,\Ol_F)$ for all $i\geq0$. 
Finally $H^i(F,\Ol_F)=H^i(\hat{F},\Ol_{\hat{F}})=0$ by \cite[IV.3.8]{Kollar_rational}.

\ref{N1=Pic:F rat connected} $\Longrightarrow$ \ref{N1=Pic:R}.
Since $X$ has rational singularities, without loss in generality we can replace $X$ by a desingularisation and assume $X$ smooth.
Since $\eta$ has connected fibres, we get $\eta_* \Ol_X = \Ol_B$.
We just saw that  $H^i(F,\Ol_F) = 0$ for all $i > 0$, and since we assume that $B$ has rational singularities, the result follows from \cite[Theorem 7.1]{Kollar_1986}. 

\ref{N1=Pic:R}  $\Longrightarrow$ \ref{N1=Pic:H1 U}.
This is just the definition of $R^1\eta_* \Ol_X = 0$.

\ref{N1=Pic:R} $\Longrightarrow$ \ref{N1=Pic:Pic XB}.
Let $D \in \Div(X)_\Q$ a divisor which is numerically trivial against the contracted curves. We want to show that $D$ is trivial in $\Pic(X/B)_\Q$, that is, a multiple of $D$ is a pull-back.
This is exactly the content of \cite[12.1.4]{KollarMori1992}.
Observe that here again we only need the vanishing assumption for $i = 1$.
\end{proof}
 
Let $\eta\colon X \to B$ be a surjective morphism between normal varieties, and $L_1$, $\dots$, $L_r$ some $\Q$-divisors on $X$. 
We define the divisorial sheaf $R(X/B; L_1, \dots, L_r)$ to be the sheaf of graded $\Ol_B$-algebras defined on every open affine set $U \subset B$ as 
\[
R(X/B;L_1,\dots,L_r)(U) = \bigoplus_{(m_1, \dots,m_r) \in \N^r} H^0(\eta^{-1}(U)/U, m_1L_1 + \dots + m_rL_r), 
\] 
where for any $D \in \Pic(X)_\Q$ 
\begin{multline*}
H^0(\eta^{-1}(U)/U, D) = \\ 
\left\{ f \in \k(\eta^{-1}(U))^* \mid \exists L \in \Pic_\Q(U), \div(f) + D + \eta^* L \ge 0 \right\}\cup \{0\}.
\end{multline*}

If moreover $\Eff(X/B) \subseteq \sum \R_+ L_i$, which ensures that we would get the same algebras using a $\Z^r$-grading instead of $\N^r$, then we say that the sheaf is a \emph{Cox sheaf}, and we denote 
\[\Cox(X/B;L_1,\dots,L_r) := R(X/B;L_1,\dots,L_r).\]
We say that a divisorial sheaf $R(X/B;L_1,\dots,L_r)$ is finitely generated if for every affine set $U$ the $\N^r$-graded $\Ol_B(U)$-algebra $R(X/B;L_1,\dots,L_r)(U)$ is finitely generated.

As the following lemma shows, for Cox sheaves this property of finite generation is independent of the choice of the $L_i$, and therefore we shall usually omit the reference to such a choice and denote a Cox sheaf simply by $\Cox(X/B)$. 

\begin{lemma}[{\cite[I.1.2.2]{ADHL_2015}}] \label{lem:FiniteGeneration}
Let $\eta\colon X \to B$ be a surjective morphism between normal varieties, $L_1, \dots, L_r \in \Pic(X)_\Q$ such that $\Eff(X/B) \subseteq \sum \R^+ L_i$, and $\Cox(X/B;L_1,\dots,L_r)$ the associated Cox sheaf.
Let $L_1',\dots,L_s'\in \Pic(X)_\Q$. 
If $\Cox(X/B;L_1,\dots,L_r)$ is finitely generated, then the divisorial sheaf $R(X/B;L_1',\dots,$ $L_s')$ also is finitely generated.
In particular, the property of finite generation of a Cox sheaf of $X/B$ does not depend on the choice of the $L_i$. 
\end{lemma}

\begin{lemma}\label{lem:Mds}
Let $X/B$ be a surjective morphism between normal varieties, whose general fibres are rationally connected.
Assume that $X$ is $\Q$-factorial, and that $X$, $B$ and the general fibres have rational singularities.
Then $X/B$ is a Mori dream space if and only if its Cox sheaf is finitely generated.
\end{lemma}

\begin{proof}
The proof is similar to the proofs in the non-relative setting of \cite[Corollaries 4.4 and 5.7]{KKL}.
\end{proof}

\begin{example} \label{ex:Mds}
Standard examples of Mori dream spaces in the non relative case (i.e.~when $B$ is a point) are toric varieties and Fano varieties.
Both of these classes of varieties are special examples of log Fano varieties, which are Mori dream spaces by \cite[Corollary~1.3.2]{BCHM}. 
If $F$ is a log Fano variety, and $B$ is any smooth variety, then $(F \times B) / B$ is a basic example of relative Mori dream space.
\end{example}

\subsection{Minimal model programme}\label{sec:MMP}

Let $X$ be a normal $\Q$-factorial variety, and $C \in \widebar{\NE}(X)$ an extremal class.
We say that the \emph{contraction} of $C$ exists (and in that case it is unique), if there exists a surjective morphism $\pi\colon X \to Y$ with connected fibres to a normal variety $Y$,
with $\rho(X/Y) = 1$, and such that any curve contracted by $\pi$ is numerically proportional to $C$. 
If $\pi$ is a small contraction, we say that the \emph{log-flip} of $C$ exists (and again, in that case it is unique) is there exists $X \ps X'$ a pseudo-isomorphism over $Y$ which is not an isomorphism, such that $X'$ is normal $\Q$-factorial and $X' \to Y$ is a small contraction that contracts curves proportional to a class $C'$.
For each $D \in N^1(X)$, if $D'$ is the image of $D$ under the pseudo-isomorphism, we have a sign change between $D \cdot C$ and $D' \cdot C'$.
We say that $X \ps X'$ is a \emph{$D$-flip}, resp.~a \emph{$D$-flop}, resp.~a \emph{$D$-antiflip} when $D \cdot C < 0$, resp.~$D \cdot C = 0$, resp.~$D \cdot C > 0$. 

If $D$ is nef on $X$, we say that $X$ is a \emph{$D$-minimal model}.
If there exists a contraction $X \to Y$ with $\rho(X/Y) = 1$, $\dim Y < \dim X$ and $-D$ relatively ample, we say that $X/Y$ is a \emph{$D$-Mori fibre space}.

A \emph{step} in the $D$-Minimal Model Programme (or in the $D$-MMP for short) is the removal of an extremal class $C$ with $D \cdot C < 0$, either via a divisorial contraction, or via a $D$-flip.
In this paper we will ensure the existence of each step in a $D$-MMP by working in one the following contexts.
Either $D = K_X + \Delta$ will be an adjoint divisor with $\Delta$ ample and we can apply the main result of \cite{BCHM}, or we will assume that $X$ is a Mori dream space, and rely on Lemma \ref{lem:anyMMP} below (which is the reason for the name).
By \emph{running a $D$-MMP from $X$}, we mean performing a sequence of such steps, replacing each time $D$ by its image, until reaching one of the following two possible outputs: a $D$-minimal model or a $D$-Mori fibre space.  
In particular, observe that for us the output of a $D$-MMP is always of the same dimension as the starting variety, and the whole process makes sense even for $D$ not pseudo-effective (in contrast with another possible convention which would be to define the output of a $D$-MMP as $\Proj(\bigoplus_{n} H^0(X,nD))$. 

We will often work in a relative setting where all steps are maps over a base variety $B$, and we will indicate such a setting by saying that we run a $D$-MMP \emph{over $B$}.

When $D = K_X$ is the canonical divisor, we usually omit the mention of the divisor in the previous notations.
So for instance given a small contraction contracting the class of a curve $C$, we speak of the flip of $C$ only if $K_X\cdot C < 0$, of the $D$-flip of $C$ if $D\cdot C <0$, and of the log-flip of $C$ when we do not want to emphasise the sign of the intersection against any divisor. 

\begin{lemma}[{see \cite[Proposition 1.11]{HK} or \cite[Theorem 5.4]{KKL}}]
\label{lem:anyMMP}
If $X/B$ is a Mori dream space, then for any class $D \in N^1(X)$ one can run a $D$-MMP from $X$ over $B$, and there are only finitely many possible outputs for such MMP.
\end{lemma}

\subsection{Singularities}
\label{sec:singularities}

Let $X$ be a normal 
variety, and let $\pi\colon Z \to X$ be a resolution of singularities,
with exceptional divisors $E_1$, $\dots$, $E_r$.
We say that $X$ has \emph{terminal singularities}, or that $X$ is terminal, if $K_X$ is $\Q$-Cartier and in the ramification formula
\[
K_Z = \pi^* K_X + \sum a_i E_i,
\]
we have $a_i >0$ for each $i$.
Similarly we say that $X$ has \emph{Kawamata log terminal} (\emph{klt} for short) singularities, or that $X$ is klt, if $a_i > -1$ for each $i$. 
Each coefficient $a_i$, which is often called the \emph{discrepancy} of $E_i$, does not depend on a choice of resolution in the sense that it is an invariant of the geometric valuation associated to $E_i$.  
Let $\Delta$ an effective $\Q$-divisor on $X$. We call $(X,\Delta)$ a klt pair if $K_X+\Delta$ is $\Q$-Cartier and if for a (and hence any) resolution of singularities $\pi\colon Z\to X$ such that the divisor $(\pi^{-1})_* \Delta \cup\Ex(\pi)$ has normal crossing support we have
\[
K_Z=\pi^*(K_X+\Delta)+\sum a_iE_i
\]
where $\pi_*(\sum a_iE_i) +\Delta=0$ and $a_i>-1$ for all $i$.
Observe that if $(X,\Delta)$ is a klt pair and $X$ is $\Q$-factorial, then for any $\Delta \ge \Delta ' \ge 0$ the pair $(X,\Delta')$ also is klt.
In particular taking $\Delta' = 0$ we get that $X$ is klt.

\begin{lemma}
\label{lem:oppositeSigns}
Let $X, Y$ be $\Q$-factorial varieties, and $\pi\colon X \to Y$ the divisorial contraction of an extremal curve $C$, with exceptional divisor $E = \Ex(\pi)$.
If $D \in \Div(X)$ and $D' = \pi_* D$, then in the ramification formula
\[
D = \pi^* D' + aE,
\]
the numbers $a$ and $D \cdot C$ have opposite signs. 
In particular, if $X$ is terminal, then $Y$ is terminal if and only if $K_X \cdot C <0$.
\end{lemma}

\begin{proof}
We have $D \cdot C = a E \cdot C$, so the claim follows from $E \cdot C < 0$.
For this, see for instance \cite[Lemma 3.6.2\parent{3}]{BCHM}.
The last assertion follows by taking $D=K_X$ and $D'=K_Y$.
\end{proof}

If we start with a $\Q$-factorial terminal variety and we run the classical MMP (that is, relatively to the canonical divisor), then each step (divisorial contraction or flip) of the MMP keeps us in the category of $\Q$-factorial terminal varieties (for divisorial contractions, this follows from Lemma~\ref{lem:oppositeSigns}). 
Moreover, when one reaches a Mori fibre space $X/B$, the base $B$ is $\Q$-factorial as mentioned above, but might not be terminal. 
However by the following result $B$ has at worst klt singularities. 

\begin{proposition}[{\cite[Corollary 4.6]{Fujino1999}}]
\label{pro:sing_of_B}
Let $X /B$ be a Mori fibre space, where $X$ is a $\Q$-factorial klt variety.
Then $B$ also is a $\Q$-factorial klt variety. 
\end{proposition}

We will also use the following related result:

\begin{proposition}[{\cite[Theorem 1.5]{Fujino2015}}]
\label{pro:sing_of_Proj}
Let $(X,\Delta)$ be a klt pair, and consider the log canonical model
\[ Y = \Proj \Bigl( \bigoplus_m H^0(X, m(K_X + \Delta)) \Bigr) \]
where the sum is over all positive integers $m$ such that $m(K_X + \Delta)$ is Cartier.
Then there exists an effective $\Q$-divisor $\Delta_Y$ such that the pair $(Y,\Delta_Y)$ is klt. 
\end{proposition}

The following class of Mori fibre spaces will be of special importance to us.

\begin{definition} \label{def:conicBundle}
A \emph{conic bundle} is a $\Q$-factorial terminal Mori fibre space $X/B$ with $\dim B = \dim X - 1$. 
The \emph{discriminant locus} of $X/B$ is defined as the union of irreducible hypersurfaces $\Gamma\subset B$ such that the preimage of a general point of $\Gamma$ is not irreducible. 
We emphasise that the terminology of conic bundle is often used in a broader sense (for instance, for any morphism whose general fibre is isomorphic to $\p^1$, with no restriction on the singularities of $X$ or on the relative Picard rank), but for our purpose we will stick to the above more restricted definition. 

We say that two conic bundles $X/B$ and $X'/B'$ are \emph{equivalent} if there exists a commutative diagram 
\[\begin{tikzcd}[link]
X \ar[dd,swap]\ar[r,"\psi",dashed]& X'\ar[dd]  \\ \\
B  \ar[r,"\theta",dashed]& B' 
\end{tikzcd}\]
where $\psi, \theta$ are birational.
\end{definition}

The singular locus of a terminal variety has codimension at least $3$ (\cite[5.18]{KM}).
This fact is crucial in the following result.

\begin{lemma}
\label{lem:divContToCodim2}
Let $\pi\colon X \to Y$ be a divisorial contraction between $\Q$-factorial terminal varieties, with exceptional divisor $E$, and assume that $\Gamma = \pi(E)$ has codimension $2$ in $Y$. Then, the following hold:
\begin{enumerate}
\item\label{ExistenceU}
There is an open subset $U\subseteq Y$ such that $U\cap \Gamma$, $U$ and $\pi^{-1}(U)$ are non-empty and contained in the smooth locus of $\Gamma$, $Y$ and $X$ respectively. 
 \item\label{BlowUponU}
For each choice of $U$ as in \ref{ExistenceU}, $\pi|_{\pi^{-1}(U)}\colon \pi^{-1}(U) \to U$ is the blow-up  of $U\cap \Gamma$ $($with reduced structure$)$. In particular, for each $p\in U$, the fibre $f=\pi^{-1}(\{p\})$ is a smooth rational curve such that
$K_X \cdot f = E \cdot f = -1$.
\end{enumerate}
\end{lemma}

\begin{proof}
Assertion~\ref{ExistenceU} follows from the fact that  $X$ and $Y$ are smooth in codimension~$2$.

Let $U$ be as in \ref{ExistenceU}, let $p\in \Gamma \cap U$ and take a general smooth surface $S \subseteq U$ containing~$p$.
Up to shrinking $U$ we can assume that $p$ is the only intersection point of $S$ and $\Gamma$.
The strict transform $\tilde{S}$ of $S$ is again a smooth surface.
Let $C_1, \dots, C_m$ be the irreducible curves contracted by the birational morphism $\tilde S \to S$, which is the composition of $m$ blow-ups.
We now show  that $m = 1$.
The condition $\rho(X/Y)=1$ implies that all $C_i$ are numerically equivalent in $X$, so for each $i,j$ we have
\[(C_i^2)_{\tilde S} = C_i \cdot E = C_j \cdot E = (C_j^2)_{\tilde S}.\]
Since at least one of the self-intersection $(C_i^2)_{\tilde S}$ must be equal to $-1$, and the exceptional locus of $\tilde S \to S$ is connected, we conclude that $m = 1$.
So $\tilde{S}\to S$ is the blow-up of $p$, hence $\pi^{-1}(U) \to U$ is the blow-up of $U\cap \Gamma$, which gives~\ref{BlowUponU}.
\end{proof}

\begin{lemma} \label{lem:fibresAreTerminal}
Let $\eta\colon X \to B$ be a morphism between normal varieties with $X$ terminal $($resp.~klt$)$.
Then for a general point $p \in B$, the fibre $\eta^{-1}(p)$ also is terminal $($resp.~klt$)$, so in particular it has rational singularities.
\end{lemma}

\begin{proof}
The fact that $\eta^{-1}(p)$ is terminal (resp.~klt) follows from \cite[7.7]{Kollar_SantaCruz} by taking successive hyperplane sections on $B$ locally defining $p$.
As already mentioned klt singularities are rational, see \cite[5.22]{KM}.
\end{proof}

\begin{lemma}\label{lem:fibresRatConnected}\phantomsection~
\begin{enumerate}
\item \label{connected:-KBigNef}
Let $(X,\Delta)$ be a klt pair, and $\pi\colon X \to Y$ be a morphism with connected fibres such that $-(K_X+\Delta)$ is $\pi$-big and $\pi$-nef. 
Then for every $p \in Y$ the fibre $\pi^{-1}(p)$ is covered by rational curves, and for a general $p \in Y$ the fibre $\pi^{-1}(p)$ is rationally connected with klt singularities.  
\item  \label{connected:birational}
Let $(Y, \Delta_Y)$ be a klt pair,
 and $\pi\colon X \to Y$ a birational morphism.
Then every fibre of $\pi$ is covered by rational curves.
\item  \label{connected:flips}
Let $\phi\colon X \ps X'$ be a sequence of log-flips between $\Q$-factorial klt varieties, and $\Gamma \subset X$ a codimension $2$ subvariety contained in the base locus of $\phi$.
Then $\Gamma$ is covered by rational curves. 
\end{enumerate}
\end{lemma}

\begin{proof}
\ref{connected:-KBigNef} and \ref{connected:birational} follow from \cite[Corollary 1.3(1) and Corollary 1.5(1)]{HMcK2007}.
Then \ref{connected:flips} is a straightforward consequence of \ref{connected:-KBigNef} applied in the case of a small contraction.
\end{proof}

\begin{lemma}\label{lem:B'_over_B_is_Mds}
Let $X \to Y$ be a morphism that factorises as $X \to W$ and $W \to Y$, where $W$ is a $\Q$-factorial klt variety.
If $X/Y$ is a Mori dream space then $W/Y$ also is a Mori dream space.
\end{lemma}

\begin{proof}
The general fibres of $W/Y$ are rationally connected because they are images of the rationally connected fibres of $X/Y$, and they have rational singularities by Lemma \ref{lem:fibresAreTerminal}. 
For any affine open subset $U \subset Y$, the algebra $\Cox(W/Y)(U)$ embeds by pull-back as a subalgebra of $\Cox(X/Y)(U)$, hence is finitely generated by Lemma \ref{lem:FiniteGeneration}.
We conclude by Lemma \ref{lem:Mds}.
\end{proof}

\subsection{Two-rays game}
\label{sec:2rays}

A reference for the notion of two-rays game is \cite[\S2.2]{Corti_explicit}.
We use a slightly different setting in the discussion below. 
Namely, first we ensure that all moves do exist by putting a Mori dream space assumption, and secondly we do not put strong restrictions on singularities (this will come later in Definition~\ref{def:rankFibration}).

Let $Y \to X$ be a surjective morphism between normal varieties, with $\rho(Y/X) = 2$.
Assume also that there exists a morphism $X/B$ such that $Y/B$ is a Mori dream space.
In particular, by Lemma \ref{lem:anyMMP} for any divisor $D$ on $Y$ one can run a $D$-MMP over $B$, hence a fortiori over $X$.
Then $\NE(Y/X)$ is a closed 2-dimensional cone, generated by two extremal classes represented by curves $C_1, C_2$.
Let $D = -A$ where $A$ is an ample divisor on $Y$, so that a $D$-minimal model does not exist.   
Then by Lemma~\ref{lem:anyMMP} for each $i = 1,2$ we can run a $D$-MMP from $Y$ over $X$, which starts by the divisorial contraction or log-flip of the class $C_i$, and produce a commutative diagram that we call the \emph{two-rays game} associated to $Y/X$ (and which does not depend on the choice of $D$):
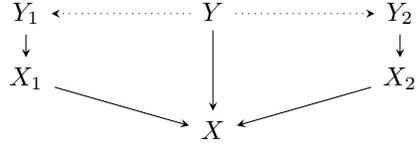
\[
\begin{tikzcd}[link]
Y_1 \ar[dd]  & Y \ar[l,dotted] \ar[r,dotted] \ar[ddd] & Y_2 \ar[dd] \\ \\
X_1 \ar[dr] && X_2 \ar[dl] \\
& X &
\end{tikzcd}
\]
Here $Y \ps Y_i$ is a (possibly empty) sequence of $D$-flips, and $Y_i \to X_i$ is either a divisorial contraction or a $D$-Mori fibre space.

Now we give a few direct consequences of the two-rays game construction.

\begin{lemma}
\label{lem:2raysOverFlip}
Let $Y_1/B$ be a Mori dream space, $Y_1 \to X_1$ a morphism over $B$ with $\rho(Y_1/X_1) = 1$, and $X_1 \ps X_2$ a sequence of relative log-flips over $B$. 
Then there exists a sequence of log-flips $Y_1 \ps Y_2$ over $B$ such that the induced map $Y_2 \to X_2$ is a morphism, of relative Picard rank~$1$ by construction. 
Moreover if $Y_1/X_1$ is a divisorial contraction $($resp.~a Mori fibre space$)$, then $Y_2/X_2$ also is.
\end{lemma}

\begin{proof}
By induction, it is sufficient to consider the case where $X_1 \ps X_2$ is a single log-flip over a non $\Q$-factorial variety $X$ dominating $B$, given by a diagram
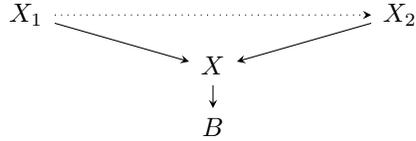
\[
\begin{tikzcd}[link]
X_1 \ar[rd] \ar[rr,dotted] && X_2 \ar[ld] \\
& X \ar[dd] \\ \\
& B
\end{tikzcd}
\]
In this situation, the two-rays game $Y_1/X$ gives a diagram
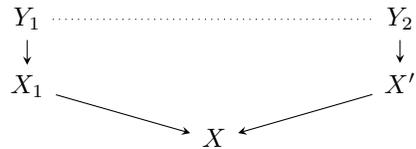
\[
\begin{tikzcd}[link]
Y_1 \ar[rr, dotted, -] \ar[dd] && Y_2 \ar[dd] \\ \\
X_1 \ar[dr] && X' \ar[dl] \\
& X &
\end{tikzcd}
\]
where $Y_1 \ps Y_2$ is a sequence of log-flips and $Y_2 \to X'$ is a morphism of relative Picard rank~$1$, with $X'$ a $\Q$-factorial variety.
If $Y_1/X_1$ is a divisorial contraction, then $Y_2/X'$ must be birational hence also is a divisorial contraction. 
On the other hand if $Y_1/X_1$ is a Mori fibre space, then $Y_2/X'$ cannot be birational, otherwise $X'/X$ would be a $D$-Mori fibre space for some divisor $D$; impossible since $X'$ is $\Q$-factorial but not $X$.  
By uniqueness of the log-flip associated to the small contraction $X_1 \to X$, we conclude in both cases that $X' = X_2$.
\end{proof}

\begin{lemma}
\label{lem:corti2.7}
Let $\phi\colon Y \ps Y'$ be a pseudo-isomorphism over $X$, where $X, Y, Y'$ are $\Q$-factorial varieties, and assume we are in one of the following situations:
\begin{enumerate}
\item \label{corti2.7:1}
$Y/X$ and $Y'/X$ are Mori fibre spaces;
\item \label{corti2.7:2}
$Y/X$ and $Y'/X$ are divisorial contractions.
\end{enumerate}
Then $\phi$ is an isomorphism.
\end{lemma}

\begin{proof}
Assertion \ref{corti2.7:1} is \cite[Proposition 3.5]{Corti1995} (the proof given there extends \emph{verbatim} in the higher dimensional case).
We now give a proof of \ref{corti2.7:2}, which is very similar.
Let $E$ and $E'$ be the exceptional divisors of $\pi\colon Y \to X$ and $\pi'\colon Y' \to X$, respectively.
Observe that $\phi_* E = E'$.
Pick $A$ a general ample divisor on $X$ and $0 < \eps \ll 1$, and consider $H = \pi^*A - \eps E$, $H' = {\pi'}^*A-\eps E'$.
Both $H$ and $H'$ are ample, and we have $H' = \phi_* H$, so by \cite[Proposition 2.7]{Corti1995} we conclude that $Y  \ps Y'$ is an isomorphism.
\end{proof}

\begin{lemma}
\label{lem:2rays2Divisors}
Let $T \to Y$ and $Y\to X$ be two divisorial contractions between $\Q$-factorial varieties, with respective exceptional divisors $E$ and $F$.
Assume that there exists a morphism $X \to B$ such that $T/B$ is a Mori dream space.
Then there exist two others $\Q$-factorial varieties $T'$ and $Y'$, with a pseudo-isomorphism $T \ps T'$ and birational contractions $T' \to Y' \to X$, with respective exceptional divisors the strict transforms of $F$ and $E$, such that the following diagram commutes:
\[
\begin{tikzcd}[link]
T \ar[rr, dotted, -] \ar[dd,"E",swap] && T' \ar[dd,"F"] \\ \\
Y \ar[dr,"F",swap] && Y' \ar[dl,"E"] \\
& X &
\end{tikzcd}
\]
\end{lemma}

\begin{proof}
The diagram comes from the two-rays game associated to $T/X$.
The only thing to prove is that the divisors are not contracted in the same order on the two sides of the two-rays game.
Assume that both $\pi\colon Y\to X$ and $\pi'\colon Y' \to X$ contract the strict transforms of the same divisor $F$. 
Then $T\to Y$ and $T'\to Y'$ both contract a same divisor $E$ and $T\ps T'$ descends to a pseudo-isomorphism $Y\ps Y'$.
By Lemma~\ref{lem:corti2.7}\ref{corti2.7:2} the pseudo-isomorphism  $Y  \ps Y'$ is an isomorphism.
Then applying  again Lemma~\ref{lem:corti2.7}\ref{corti2.7:2} to the two divisorial contractions from $T, T'$ to $Y \simeq Y'$, with same exceptional divisor $E$, we obtain that $T \ps T'$ also is an isomorphism.
The morphisms $T/Y$ and $T/Y'$ are then divisorial contractions of the same extremal ray, contradicting the assumption that the diagram was produced by a two-rays game.
\end{proof}

\subsection{Gonality and covering gonality}
\label{sec:gonality}
Recal that the \emph{gonality} $\gon(C)$ of a (possibly singular) curve $C$ is defined to be the least degree of the field extension associated to a dominant rational map $C\rat \p^1$.

Note that $\gon(C) = 1$ if and only if $C$ is rational. 
Moreover, for each smooth curve $C\subset \p^2$ of degree $>1$ we have $\gon(C)=\deg(C)-1$. 
Indeed, the inequality $\gon(C)\le \deg(C)-1$ is given by the projection from a general point of $C$ and the other inequality is given by a result of Noether (see for instance \cite{BDELU}).

The following definitions are taken from \cite{BDELU} (with a slight change, see Remark~\ref{rem:CovGonConnGon}).

\begin{definition}
For each variety $X$ we define the \emph{covering gonality} of $X$ to be
\[ 
\covgon(X) = \min \left  \{ c > 0 \ 
\left | 
\parbox{2.7in}{
\begin{center} 
There is a dense open subset $U\subseteq X$ such that each point $x\in U$ is contained in an irreducible curve $C \subseteq X$ with  $\gon(C) \le c$. 
\end{center}
}
\right \}\right..
\]
Similarly we define the \emph{connecting gonality} of $X$ to be 
\[ 
\conngon(X) = \min \left  \{ c > 0 \ 
\left | 
\parbox{2.7in}{
\begin{center} 
There is a dense open subset $U\subseteq X$ such that any two points $x,y\in U$ are contained in an irreducible curve $C \subseteq X$ with  $\gon(C) \le c$.
\end{center}
}
\right \}\right..
\]
\end{definition}

\begin{remark}\label{rem:CovGonConnGon}\item
\begin{enumerate}
\item
Our definitions of the covering and connecting gonality slightly differ from those of  \cite{BDELU}, as we ask $\gon(C)\le c$ where they ask $\gon(C)=c$. 
Lemma~\ref{lem:Covgonfamily} shows that the covering gonality is the same for both definitions.  A similar argument should also work for the connecting gonality, but we do not need it here, as we will not use any result of  \cite{BDELU} involving the connecting gonality.
\item
 The covering gonality and connecting gonality are integers which are invariant under birational maps.
\item
For each variety $X$, we have
\[\covgon(X) \le \conngon(X).\]
Moreover, if $\dim(X)=1$, then $\covgon(X)=\conngon(X)=\gon(X)$.
\item
If $\covgon(X)=1$ one says that $X$ is \emph{uniruled}. 
This corresponds to asking that the union of all rational curves on $X$ contains an open subset of $X$.
Similarly, $X$ is said to be \emph{rationally connected} if $\conngon(X)=1$. 
As already mentioned in \S\ref{sec:moriDream}, this corresponds to asking that a rational curve passes through two general points.
\item \label{gonality:P1xB}
Each rationally connected variety is uniruled. However, the converse does not hold in general. Indeed, for each variety $B$, we have  $\covgon(B\times \p^n)=1$ for each $n\ge 1$, but $\conngon(B\times \p^n)=\conngon(B)$ as the following lemma shows:
Lemma~\ref{lem:CovgonRatMorphism}\ref{covconngonXY} applied to $B\times \p^n/B$ gives $\conngon(B\times \p^n)\ge \conngon(B)$, and the other inequality is given by taking sections in $B\times \p^n$ of curves in $B$.
\end{enumerate}
\end{remark}

We recall the following classical facts:
\begin{lemma}\label{lem:CovgonRatMorphism}
Let $X,Y$ be varieties 
and $\phi\colon X\to Y$ a surjective morphism. 
\begin{enumerate}
\item\label{covXY}
If $X$ and $Y$ have dimension $1$, then $\gon(X)\ge \gon(Y)$.
\item\label{covconngonXY}
We have $\conngon(X)\ge\conngon(Y)$ $($but not $\covgon(X)\ge \covgon(Y)$ in general,
see Remark~$\ref{rem:CovGonConnGon}\ref{gonality:P1xB})$.
\item\label{covgonfinitemorphism}
 If $\dim X=\dim Y$, denote by $\deg(\phi)$ the degree of the associated field extension $\C(Y)\subseteq \C(X)$. Then
\[\covgon(X)\le  \covgon(Y)\cdot \deg(\phi).\]
\item\label{covgondeg}
If $X\subseteq\p^n$ is a closed subvariety, then $\covgon(X)\le \deg(X)$.
\end{enumerate}
\end{lemma}

\begin{proof}
\ref{covXY}. 
See for instance \cite[Proposition A.1(vii)]{Poonen}. 

\ref{covconngonXY}.
We take two general points $y_1,y_2\in Y$, choose then two general points $x_1,x_2\in X$ with $\phi(x_i)=y_i$ for $i=1,2$, and take an irreducible curve $C\subset X$ of gonality $\le \conngon(X)$ which contains $x_1$ and $x_2$. The image is an irreducible curve $\phi(C)$ of gonality $\le \conngon(X)$ (by \ref{covXY}), containing $y_1$ and $y_2$.

\ref{covgonfinitemorphism}.
By definition of $\covgon(Y)$, the union of irreducible curves $C$ of $Y$ with $\gon(C)\le\covgon(Y)$ covers a dense open subset of $Y$. 
Taking the preimages of general such curves, we obtain a covering of a dense open subset of $X$ by irreducible curves $D$ of $X$ with $\gon(D)\le \covgon(Y)\cdot \deg(\phi)$. 

\ref{covgondeg}.
If $X\subseteq\p^n$ is a closed subvariety, we apply \ref{covgonfinitemorphism} to the projection onto a general linear subspace $Y\subseteq\p^n$ of dimension $\dim(Y)=\dim(X)$.
\end{proof}

\begin{lemma}\label{lem:Covgonfamily}
Let $X$ be a variety with $\covgon(X)=c$. 
There is a smooth projective morphism $\mathcal{C}\to T$ over a quasi-projective irreducible base variety $T$, with irreducible fibres of dimension one and of gonality $c$, together with a dominant morphism $\mathcal{C}\to X$ such that a general fibre of $\mathcal{C}/T$ is birational to its image in $X$. 
In particular, there is a dense open subset $U$ of $X$ such that through every point $p\in U$ there is an irreducible curve $C\subseteq X$ with $\gon(C)=c$.
\end{lemma}

\begin{proof}
The proof is analogue to the one of \cite[Lemma 2.1]{GouKou}. 
We consider the Hilbert Scheme $\Hl$ of all one-dimensional subschemes of $X$, which is not of finite type, but has countably many components. 
One of the irreducible components contains enough curves of gonality $\le \covgon(X)$ to get a dominant map to $X$. 
We then look at the set of gonality $i$ for each $i$ and obtain algebraic varieties parametrising these, as in \cite[Lemma 2.1]{GouKou}. 
Having finitely many constructible subsets in the image, at least one integer $i\le \covgon(X)$ gives a dominant map to $X$ parametrising curves of gonality $i$. 
By definition of $\covgon(X)$, this integer $i$ has to be equal to $\covgon(X)$.
\end{proof}

The following result gives a bound from below that complements the easy bound from above from Lemma~\ref{lem:CovgonRatMorphism}.

\begin{theorem}[{\cite[Theorem A]{BDELU}}]
\label{thm:gonality}
Let $X\subset \p^{n+1}$ be an irreducible hypersurface of degree $d\ge n+2$ with canonical singularities. Then, $ \covgon(X)\ge d-n$.
\end{theorem}

We now recall the following definition of \cite{BDELU}, which is a birational version of the classical $p$-very ampleness criterion, which asks that  every subscheme of length $p+1$ imposes independent  conditions on the sections of a line bundle.

\begin{definition} \label{BVA}
Let $X$ be variety  and let $p\ge 0$ be an integer.

A line bundle $L$ on $X$ \textit{satisfies property $\BVA_p$} if there exists a proper Zariski-closed subset $Z = Z(L) \subsetneqq X$ depending on $L$ such that the restriction map
$H^{0}(X,L) \to H^{0}(X,{L \otimes \Ol_{\xi}})$ is surjective for every finite subscheme $\xi \subset X$ of length $p+1$ whose support is disjoint from $Z$. 

The line bundle is moreover $p$-\textit{very ample} if one can choose $Z$ to be empty.
\end{definition}

The property $\BVA_0$ corresponds to asking that $L$ is effective, and $\BVA_1$ is usually called ``birationally very ample''. This explains the notation.
This notion is related to the covering gonality via the following result, which essentially follows from the fact that if the canonical divisor  $K_C$ of a smooth irreducible curve $C$ satisfies $\BVA_p$, then $\gon(C)\ge p+2$ (see  \cite[Lemma 1.3]{BDELU}):

\begin{theorem}[{\cite[Theorem 1.10]{BDELU}}]
\label{thm:BVAK}
Let $X$ be a variety, and $p\ge 0$ an integer.
If $K_X$ satisfies $\BVA_p$, then $\covgon(X)\ge p+2$.
\end{theorem}

We will use the following observations of \cite{BDELU} to check the hypothesis of Theorem~\ref{thm:BVAK}:

\begin{lemma}\label{lem:BVApEasy}
Let $X$ be a variety, $L$ a line bundle on $X$ and $p\ge 0$ an integer.
\begin{enumerate}
\item\label{BVAeff}
If $L$ satisfies $\BVA_p$ and $E$ is an effective divisor on $X$, then  $\Ol_X(L+E)$ satisfies $\BVA_p$.
\item\label{BVAmor}
Suppose that $f\colon Y\to X$ is a morphism which is birational onto its image, that $L$ satisfies $\BVA_p$ and that the closed set $Z\subseteq X$ from Definition~$\ref{BVA}$ does not contain the image of $f$. 
Then,  $f^*L $ satisfies $\BVA_p$.
\item\label{BVAPn}
For each $d\ge 0$, $\Ol_{\p^n}(p)$ is $p$-very ample, i.e.~satisfies  $\BVA_p$ with an empty closed set $Z\subseteq \p^n$.
\end{enumerate}
\end{lemma}

\begin{proof}
The three assertions follow from the definition of $\BVA_p$, as mentioned in \cite[Example 1.2]{BDELU}.
\end{proof}

\section{Rank \textit{r} fibrations and Sarkisov links}\label{sec:SarkisovLinks}

In this section we introduce the notion of rank $r$ fibration, recovering the notion of Sarkisov link for $r = 2$.
Then we focus on rank $r$ fibrations and Sarkisov links with general fibre a curve.

\subsection{Rank \textit{r} fibrations}

The notion of rank~$r$ fibration is a key concept in this paper. 
Essentially these are (relative) Mori dream spaces with strong constraints on singularities.
The cases of $r = 1,2,3$ will allow us to recover respectively the notion of terminal Mori fibre spaces, of Sarkisov links, and of elementary relations between those.
The precise definition is as follows.

\begin{definition} \label{def:rankFibration}
Let $r\ge 1$ be an integer.
A morphism $\eta\colon X\to B$ is a \emph{rank~$r$ fibration} if the following conditions hold:
\begin{enumerate}[(RF1)]
\item \label{fib:dream} 
$X/B$ is a Mori dream space (see Definition~\ref{def:moriDream});
\item \label{fib:eta} 
$\dim X > \dim B \ge 0$ and $\rho(X/B) = r$;
\item \label{fib:singX} 
$X$ is $\Q$-factorial and terminal, and for any divisor $D$ on $X$, the output of any $D$-MMP from $X$ over $B$ is  still $\Q$-factorial and terminal (recall that such an output has the same dimension as $X$ by definition, see \S\ref{sec:MMP});
\item \label{fib:singB}
There exists an effective $\Q$-divisor $\Delta_B$ such that $(B, \Delta_B)$ is klt.
\item \label{fib:big} 
The anticanonical divisor $-K_X$ is $\eta$-big (see Lemma~\ref{lem:piBig}). 
\end{enumerate}

We say that a rank~$r$ fibration $X/B$ \emph{factorises through} a rank~$r'$ fibration $X'/B'$, or that \emph{$X'/B'$ is dominated by $X/B$}, if the fibrations $X/B$ and $X'/B'$ fit in a commutative diagram 
\[
\begin{tikzcd}[link]
X \ar[rrr] \ar[dr,dashed] &&& B \\
& X' \ar[r] & B' \ar[ur]
\end{tikzcd}
\]
where $X \rat X'$ is a birational contraction, and $B' \to B$ is a morphism with connected fibres.
This implies $r \ge r'$.
\end{definition}

The notion of rank $r$ fibration bears some resemblance with the notion of fibration of Fano type in \cite{BirkarA}.
Note however that our condition \ref{fib:singX} imposing strong restriction on singularities does not seem to appear previously in the literature.

\begin{example}\label{ex:rank_r} \phantomsection~
\begin{enumerate}
\item \label{ex:rank_r:2}
If $X$ is a $\Q$-factorial terminal Fano variety of rank~$r$, then $X/\pt$ is a rank~$r$ fibration.
Indeed as already mentioned in Example \ref{ex:Mds}, $X$ is a Mori dream space, and moreover for any divisor $D$ the output of a $D$-MMP is $\Q$-factorial and terminal. 
Both assertions follow from the fact that we can pick a small rational number $\eps >0$ such that $-K_X + \eps D$ is ample, and then writing $\eps D = K_X + (-K_X + \eps D)$ we see that a $D$-MMP is also a ($K_X+$ ample)-MMP. 
\item \label{ex:rank_r:3}
Let $p_1, p_2$ be two distinct points on a fibre $f$ of $\p^1 \times \p^1/\p^1$, and consider $S \to \p^1 \times \p^1$ the blow-up of $p_1$ and $p_2$.
Then $S$ is a weak del Pezzo toric surface of Picard rank~$4$, hence in particular $S/\pt$ is a Mori dream space.
However $S/\pt$ is \emph{not} a rank~$4$ fibration, because when contracting the strict transform of $f$ one gets a  singular point (hence non terminal as we work here with surfaces), which is forbidden by condition~\ref{fib:singX} of Definition~\ref{def:rankFibration}.
\end{enumerate}
\end{example}

Other basic examples are terminal Mori fibre spaces:

\begin{lemma} \label{lem:rank1=Mfs}
Let $\eta\colon X \to B$ a surjective morphism between normal varieties.
Then $X/B$ is a rank~$1$ fibration if and only if $X/B$ is a terminal Mori fibre space.
\end{lemma}

\begin{proof}
Observe that if $\rho(X/B) = 1$, the notions of $\eta$-ample and $\eta$-big are equivalent.
So the implication 
\[
X/B \text{ is a rank $1$ fibration} \implies  X/B \text{ is a terminal Mori fibre space}
\] 
is immediate from the definitions, and we need to check the converse.

Assume that $X/B$ is a terminal Mori fibre space.
Then $\dim X > \dim B$ and $\rho(X/B) = 1$, which is \ref{fib:eta}, by Proposition~\ref{pro:sing_of_B} the base $B$ is klt, which gives
\ref{fib:singB}, and $-K_X$ is $\eta$-ample, which gives \ref{fib:big}.

We now prove that $X/B$ is a Mori dream space, which is \ref{fib:dream}.
Condition \ref{dream:Q_factorial} holds by assumption.
By Lemma~\ref{lem:fibresRatConnected}\ref{connected:-KBigNef} the general fibre of $X/B$ is rationally connected with rational singularities, which gives \ref{dream:rat_connected}.
Moreover since $\rho(X/B) = 1$, we have $\Ample(X/B) = \Nef(X/B) = \widebar{\Mov}(X/B)$ equal to a single ray, and so conditions \ref{dream:nef} and \ref{dream:mov} are immediate.

Finally we prove \ref{fib:singX}.
By assumption $X$ is terminal and $\Q$-factorial.
For any divisor $D$, either $D$ is $\eta$-nef and $X/B$ is a $D$-minimal model, or $-D$ is $\eta$-ample and $X/B$ is a $D$-Mori fibre space.
So $X$ is the only possible output for a $D$-MMP, which proves the claim.  
\end{proof}

\begin{lemma}\label{lem:StableUnderMMP}
Let $X/B$ be a rank~$r$ fibration.
\begin{enumerate}
\item \label{stab_under_MMP:1}
If $X'$ is obtained from $X$ by performing a log-flip $($resp.~a divisorial contrac\-tion$)$ over $B$, then $X'/B$ is a rank~$r$ fibration $($resp.~a rank~$(r-1)$-fibration$)$.
\item \label{stab_under_MMP:2}
Assume that $X/B$ factorizes through a rank~$s$ fibration $X'/B'$ such that the birational map $X \to X'$ is a morphism.
Let $t = \rho(X/B')$.
Then $X/B'$ is a rank~$t$ fibration.
\end{enumerate}
\end{lemma}

\begin{proof} 
\ref{stab_under_MMP:1}.
Let $\pi\colon X\to X'$ be a divisorial contraction over $B$, with exceptional divisor $E$ (the case of a log-flip, which is similar and easier, is left to the reader). 

\ref{fib:dream}. 
The general fibre of $X'/B$ remains rationally connected, and is terminal by Lemma~\ref{lem:fibresAreTerminal}, so it remains to show that a Cox sheaf of $X'/B$ is finitely generated (Lemma~\ref{lem:Mds}). 

Let $L_1, \dots, L_p \in \Pic(X)_\Q$ and $L_1',\dots,L_q' \in \Pic_\Q(X')$ such that $\Eff(X/B) \subseteq \sum \R_+ L_i$ and $\Eff(X'/B) \subseteq \sum \R_+ L_i'$.
For each open set $U \subseteq B$,  by pulling-back we get an injective morphism of algebras 
\[\Cox(X'/B;L_1',\dots,L_q')(U) \hookto \Cox(X/B; E,\pi^*L_1',\dots,\pi^*L_q',L_1,\dots, L_p)(U).\] 
Since $X/B$ is a rank~$r$ fibration, its Cox sheaf is finitely generated by Lemma~\ref{lem:Mds}, and so $\Cox(X'/B;L_1',\dots,L_q')$ also is finitely generated by  Lemma~\ref{lem:FiniteGeneration}.

\ref{fib:eta}.
By definition of a divisorial contraction we have $\dim X' = \dim X > \dim B$, and $\rho(X') = \rho(X) -1$, so $\rho(X'/B) = r - 1$.

\ref{fib:singX}.
The output of any MMP from $X'$ also is the output of a MMP from $X$, and so is $\Q$-factorial and terminal by assumption. 

\ref{fib:singB} holds by assumption.

\ref{fib:big}. 
Follows from the fact that the image of a big divisor by a birational morphism is still big. 

\ref{stab_under_MMP:2}.
The conditions of \ref{fib:eta} and \ref{fib:singB} hold by assumption. 
\ref{fib:singX} follows because any MMP over $B'$ also is a MMP over $B$.
For \ref{fib:big} we observe that a curve contracted by $X/B'$ also is contracted by $X/B$, so a divisor relatively ample for $X/B$ also is relatively ample for $X/B'$. Then we can restrict a decomposition $-K_X = \eta\text{-ample} + \text{effective}$ for $X/B$ to get a similar decomposition for $X/B'$.

Finally we show \ref{fib:dream}. 
Let $L_1, \dots, L_r$ be $\Q$-divisors on $X$ such that $\Eff(X/B) \subseteq \sum \R_+ L_i$, which implies $\Eff(X/B') \subseteq \sum \R_+ L_i$.
Let $\phi\colon B' \to B$ the morphism given by assumption. 
Then for each affine open set $U' \subset B'$, we have
\[
\Cox(X/B'; L_1, \dots, L_r)(U') = \Cox(X/B; L_1, \dots, L_r)(\phi(U')),
\]
and the latter is finitely generated by assumption.
A general fibre of $X/B'$ is rationally connected because it is birational to a fibre of $X'/B'$, and it has rational singularities by Lemma~\ref{lem:fibresAreTerminal}.
We conclude by Lemma \ref{lem:Mds}.
\end{proof}

\begin{lemma} 
\label{lem:weakFano}
Any rank~$r$ fibration $X/B$ is pseudo-isomorphic, via a sequence of antiflips over $B$, to another rank~$r$ fibration $Y/B$ such that $-K_Y$ is relatively nef and big over $B$.
\end{lemma}

\begin{proof}
We run a $(-K)$-MMP from $X$ over $B$ (recall that by Lemma~\ref{lem:anyMMP}, one can run a $D$-MMP for an arbitrary divisor $D$).
It is not possible to have a divisorial contraction, because by Lemma~\ref{lem:oppositeSigns} the resulting singularity would not be terminal, in contradiction with assumption \ref{fib:singX} in the definition of rank~$r$ fibration.
If there exists an extremal class that gives a small contraction, we anti-flip it.
After finitely many such steps, either $-K$ is relatively nef, or there exists a fibration such that $K$ is relatively ample. 
But this last situation contradicts the assumption \ref{fib:big} that the anti-canonical divisor is big over $B$.
So finally $-K$ is also relatively nef over $B$, as expected.
\end{proof}

\begin{corollary}
\label{cor:weakFanoFiber}
Let $\eta\colon Y \to B$ be a rank~$r$ fibration, $p \in B$ a general point, and $Y_p = \eta^{-1}(p)$ the fibre over $p$.
\begin{enumerate}
\item \label{cor:weakFanoFiber1}
If $-K_Y$ is relatively nef and big over $B$, then the curves $C \subset Y_p$ such that $K_{Y_p} \cdot C = 0$ cover a subset of codimension at least $2$ in $Y_p$.
\item \label{cor:weakFanoFiber2}
Without assumption on $-K_Y$, the fibre $Y_p$ is pseudo-isomorphic to a weak Fano terminal variety, and the curves $C\subset Y_p$ that satisfy $ K_{Y_p}\cdot C \ge 0$ cover a subset of codimension at least $2$ in $Y_p$.
\end{enumerate}
\end{corollary}

\begin{proof}
\begin{enumerate}
\item[\ref{cor:weakFanoFiber1}]
By Lemma~\ref{lem:fibresAreTerminal} there is a dense open subset $U\subseteq B$ such that for each $p \in U$ the fibre $Y_p$ is terminal.
As $K_{Y_p}=K_Y|_{Y_p}$ is big and nef for each $p\in U$, \cite[Theorem 1.1]{Kollar93} gives an integer $m$ such that $-mK_{Y_p}$ is base-point free for each $p\in U$. In particular, the rational $\phi :=  \lvert -mK_Y \rvert \times \eta \colon Y \rat \p^N \times B$ is a morphism on $Y_U := \eta^{-1}(U)$, and $\phi$ induces a birational contraction over $U$ from $Y_U$ onto its image~$X_U \subset \p^N \times U$.
Let $\Gamma \subset Y_U$ be the subset covered by curves contracted by $Y_U/U$ that are trivial against the canonical divisor.
Write $K_{Y_U} = \phi^* K_{X_U} + \sum a_i E_i$, where the $E_i$ run over all the divisors contained in $\Gamma$.
Each $a_i$ is positive because $X$ is terminal by definition of a rank~$r$ fibration, but since $K_{Y}$ is $\phi$-nef the Negativity Lemma also says that $a_i <0$ for all $i$.
In conclusion $\Gamma$ does not contain any divisor.
So $\Gamma$ has codimension at least 2 in $Y_U$, hence $\Gamma_p = \Gamma \cap Y_p$ has codimension at least 2 in $Y_p$ for a general $p$.

\item[\ref{cor:weakFanoFiber2}]
By Lemma \ref{lem:weakFano}, the rank $r$ fibration $Y/B$ is pseudo-isomorphic, via a sequence of antiflips over $B$, to another rank~$r$ fibration $Y'/B$ such that $-K_{Y'}$ is relatively nef and big over $B$.
For a general $p\in B$, the fibre $Y_p\subset Y$ is pseudo-isomorphic to the fibre $Y'_p\subset Y'$. Denote by $F\subseteq Y_p$ (respectively $F'\subseteq Y'_p$) the closure of the union of the curves $C\subset Y_p$ that satisfy $C\cdot K_{Y_p}\ge 0$ (respectively  $C\subset Y'_p$ that satisfy $C\cdot K_{Y'_p}\ge 0$). We want to prove that the codimension of $F$ in $Y_p$ is at least~$2$.

By \ref{cor:weakFanoFiber1}, the set $F'$ has codimension at least 2 in $Y'_p$.
As $-K_Y$ is relatively big, the divisor $-K_{Y_p}$ is big. Hence, for some large $m>0$, the base-locus $\Bl$ of $-mK_{Y_p}$ has codimension $\ge 2$.
It remains to see that each curve $C\subset Y_p$ such that $C\cdot K_{Y_p}\ge 0$ is contained either in $\Bl$ or is the strict transform of a curve $C'\subset F'$.
If $C\cdot K_{Y_p}>0$, then $C\cdot (-mK_{Y_p})<0$, so $C$ is contained in $\Bl$. If $C\cdot K_{Y_p} =  0$ and  $C$ is not contained in $\Bl$, then $C$ is disjoint from a general member of the linear system $\lvert -mK_Y \rvert$, and so is not affected by the sequence of antiflips.
Hence, the strict transform of $C$ is a curve $C'\subset Y'_p$ that is also disjoint from a general member of $\lvert -mK_{Y'_p} \rvert$ whence  $C'\cdot K_{Y'_p}= 0$.
\qedhere
\end{enumerate}
\end{proof}

\subsection{Sarkisov links}

The notion of rank~$2$ fibration recovers the notion of Sarkisov link: 

\begin{lemma} \label{lem:rank 2 and link}
Let $Y/B$ be a rank~$2$ fibration.
Then $Y/B$ factorises through exactly two rank~$1$ fibrations $X_1/B_1, X_2/B_2$, which both fit into a diagram
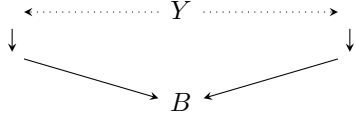
\[
\begin{tikzcd}[link]
\ar[dd]  & Y \ar[l,dotted] \ar[r,dotted] & \ar[dd] \\ \\
 \ar[dr] && \ar[dl] \\
& B &
\end{tikzcd}
\]
where the top dotted arrows are sequences of log-flips, and the other four arrows are morphisms of relative Picard rank~$1$.
\end{lemma}

\begin{proof}
The diagram comes from the two-rays game associated to $Y/B$, as explained in \S\ref{sec:2rays}.
Morever, since $\dim Y > \dim B$, on each side of the diagram exactly one of the two descending arrows corresponds to a morphisms $X_i \to B_i$ with $\dim Y = \dim X_i > \dim B_i$.
If $B_i = B$ then $X_i/B_i$ is a rank~$1$ fibration by Lemma~\ref{lem:StableUnderMMP}\ref{stab_under_MMP:1}.
If $\rho(B_i/B) = 1$, we can use Lemma~\ref{lem:StableUnderMMP}\ref{stab_under_MMP:2}, or alternatively use the following simpler argument.
Since $-K_{X_i}$ is relatively big over $B$ we have $-K_{X_i}\cdot C > 0$ for a general contracted curve of $X_i/B_i$ (write $-K_{X_i} = A + E$ with $A$ relatively ample and $E$ effective, and take $C$ not contained in $E$.) 
So $-K_{X_i}$ is relatively ample over $B_i$, hence $X_i/B_i$ is a terminal Mori fibre space, or equivalently a rank~$1$ fibration (Lemma~\ref{lem:rank1=Mfs}).
\end{proof}

\begin{definition} 
\label{def:sarkisovLink}
In the situation of Lemma~\ref{lem:rank 2 and link}, we say that the birational map $\chi\colon X_1 \rat X_2$ is a \emph{Sarkisov link}.
The diagram is called a \emph{Sarkisov diagram}.
Observe that a rank 2 fibration uniquely defines a Sarkisov diagram, but such a diagram does not have a canonical ``left side'' or ``right side''. In other words,  when $\chi$ is not an involution, the rank 2 fibration only defines the unordered pair $\{ \chi, \chi^{-1} \}$ of a Sarkisov link and its inverse.
Nevertheless we will commit the slight abuse of speaking of \emph{the} Sarkisov link associated to a rank 2 fibration.

If a rank~$r$ fibration factorises through $Y/B$, we equivalently say that it \emph{factorises through} the Sarkisov link associated to $Y/B$.

We say that the Sarkisov link associated with a rank 2 fibration $Y/B$ is a \emph{Sarkisov link of conic bundles} if $\dim B = \dim X - 1$. 
Observe that in this situation both $X_1/B_1$ and $X_2/B_2$ are indeed conic bundles in the sense of Definition~\ref{def:conicBundle}. 
\end{definition}

\begin{definition} \label{def:4TypesOfLinks}
In the diagram of Lemma~\ref{lem:rank 2 and link}, there are two possibilities for the sequence of two morphisms on each side of the diagram: either the first arrow is already a Mori fibre space, or it is divisorial and in this case the second arrow is a Mori fibre space.
This gives $4$ possibilities, which correspond to the usual definition of \emph{Sarkisov links of type \I, \II, \III and \IV}, as illustrated on Figure~\ref{fig:SarkisovTypes}.
\end{definition}

\begin{figure}[ht]
\[
{
\def\arraystretch{2.2}
\begin{array}{cc}
\begin{tikzcd}[ampersand replacement=\&,column sep=1.3cm,row sep=0.16cm]
\ar[dd,"div",swap]  \ar[rr,dotted,-] \&\& X_2 \ar[dd,"fib"] \\ \\
X_1 \ar[uurr,"\chi",dashed,swap] \ar[dr,"fib",swap] \&  \& B_2 \ar[dl] \\
\& B_1 = B \&
\end{tikzcd}
&
\begin{tikzcd}[ampersand replacement=\&,column sep=.8cm,row sep=0.16cm]
\phantom{X}\ar[dd,"div",swap]  \ar[rr,dotted,-] \&\& \ar[dd,"div"] \\ \\
X_1 \ar[rr,"\chi",dashed,swap] \ar[dr,"fib",swap] \&  \& X_2 \ar[dl,"fib"] \\
\& B_1 = B = B_2 \&
\end{tikzcd}
\\
\I & \II 
\\
\begin{tikzcd}[ampersand replacement=\&,column sep=1.3cm,row sep=0.16cm]
X_1 \ar[ddrr,"\chi",dashed,swap] \ar[dd,"fib",swap]  \ar[rr,dotted,-] \&\& \ar[dd,"div"] \\ \\
B_1 \ar[dr] \& \& X_2 \ar[dl,"fib"] \\
\& B = B_2 \&
\end{tikzcd}
&
\begin{tikzcd}[ampersand replacement=\&,column sep=1.7cm,row sep=0.16cm]
X_1 \ar[rr,"\chi",dotted,swap] \ar[dd,"fib",swap]  \&\& X_2 \ar[dd,"fib"] \\ \\
B_1 \ar[dr] \& \& B_2 \ar[dl] \\
\& B \&
\end{tikzcd}
\\
\III & \IV 
\end{array}
}
\]
\caption{The four types of Sarkisov links.}
\label{fig:SarkisovTypes}
\end{figure}

\begin{remark}
\label{rem:top_row}
The definition of a Sarkisov link in the literature is usually not very precise about the pseudo-isomorphism involved in the top row of the diagram.
An exception is \cite[Definition 3.1.4\parent{b}]{CPR}, but even there they do not make clear that there is at most one flop, and that all varieties admit morphisms to a common~$B$.
Observe that our definition is a priori more restrictive, notably because we assume the anticanonical divisor of a rank 2 fibration to be relatively big.
However one could check that the definition is equivalent to the usual one.

It follows from our definition that there are strong constraints about the sequence of antiflips, flops and flips (that is, about the sign of the intersection of the exceptional curves against the canonical divisor).
Precisely, the top row of a Sarkisov diagram has the following form:
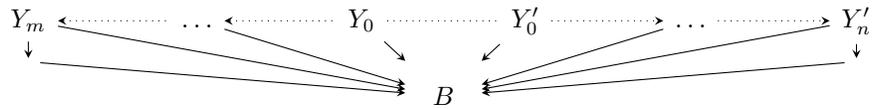
\[
\begin{tikzcd}[column sep=1.5cm,row sep=0.2cm]
Y_m \ar[d] \ar[r,dotted,<-] & \dots \ar[r,dotted,<-] & Y_0 \ar[r,dotted,-] & Y'_0 \ar[r,dotted,->] & \dots \ar[r,dotted,->] & Y'_n \ar[d] \\
\text{}  &&&&& \text{} \\
&& \ar[r,phantom,"B",""{name=B, inner sep=5mm}] &
\ar[from=1-1, to=B] \ar[from=1-2, to=B] \ar[from=1-3, to=B] 
\ar[from=1-4, to=B] \ar[from=1-5, to=B] \ar[from=1-6, to=B]
\ar[from=2-1, to=B] \ar[from=2-6, to=B]
\end{tikzcd}
\]
where $Y_0 \ps Y'_0$ is a flop over $B$ (or an isomorphism), $m,n \ge 0$, and each $Y_i \ps Y_{i+1}$, $Y'_i \ps Y'_{i+1}$ is a flip over $B$.
This follows from the fact that for $Y = Y_i$ or $Y_i'$, a general contracted curve $C$ of the fibration $Y/B$ satisfies $K_Y \cdot C < 0$, hence at least one of the two extremal rays of the cone $\NE(Y/B)$ is strictly negative against $K_Y$.

Observe also that both $Y_0/B$ and $Y_0'/B$ are relatively weak Fano (or Fano if the flop is an isomorphism) over $B$, as predicted by Lemma~\ref{lem:weakFano}. 
All other $Y_i/B$ and $Y'_i/B$ are not weak Fano over $B$, but still each is a rank 2 fibration that uniquely defines the Sarkisov diagram.
\end{remark}

\begin{example}
\label{ex:simple_links}
We give some simple examples of Sarkisov links of each type in dimension 3.
Here all varieties are smooth, and the pseudo-isomorphisms in the top rows of the Sarkisov diagrams are isomorphisms.
For more complicated (and typical) examples, see \S\ref{sec:examples}.
Observe that \ref{simple_link:1} and \ref{simple_link:2} are examples of Sarkisov links of conic bundles, while \ref{simple_link:3} and \ref{simple_link:4} are not.
\begin{enumerate}
\item \label{simple_link:1}
Let $X_1/B_1 = \p^1 \times \p^2 / \p^2$, and let $X_2 \to X_1$ be the blow-up of one fibre.
Then $X_2 = \p^1 \times \F_1$ is a Mori fibre space over the Hirzebruch surface $B_2 = \F_1$. The map $\chi\colon X_1/B_1 \rat X_2/B_2$ is a link of type \I, or equivalently $\chi^{-1}\colon X_2/B_2 \to X_1/B_1$ is a link of type \III. 
\item \label{simple_link:2}
Take again $X_1/B_1 = \p^1 \times \p^2 / \p^2$, let $L \subset \p^2$ be a line, and $\Gamma = \{0\} \times L \subset X_1$.
Let $Y\to X_1$ be the blow-up of $\Gamma$, and denote by $D$ the strict transform on $Y$ of $\p^1 \times L \subset X_1$.
Then there is a divisorial contraction $Y \to X_2$ that contracts $D$ to a curve, and $X_2/\p^2$ is still a $\p^1$-bundle (but not a trivial product).
The map $\chi\colon X_1/\p^2 \rat X_2/\p^2$ is a link of type \II.
\item \label{simple_link:3}
A general cubo-cubic map in $\Bir(\p^3)$ provides an example of link of type \II with $X_1, X_2$ equal to $\p^3$ and $B_1 = B_2 = \pt$ a point.
Indeed the resolution of such a map consists in blowing-up a smooth curve of genus 3 and degree 6 in $X_1$, and then contracting a divisor onto a curve of the same kind in $X_2$. 
This is the only example of a link of type \II from $\p^3$ to $\p^3$ starting with the blow-up of a smooth curve where the pseudo-isomorphism is in fact an isomorphism: see \cite{Katz}. 
\item \label{simple_link:4}
Finally, take $X_1 = X_2 = \p^1 \times \p^2$, $B_1 = \p^1$, $B_2 = \p^2$, and let $X_1/B_1$ and $X_2/B_2$ be respectively the first and second projection.
Then the identity map $\id\colon X_1/B_1 \to X_2/B_2$ is a link of type \IV.
\end{enumerate}
\end{example}

\begin{lemma}
\label{lem:distinctDivisors}
Consider a Sarkisov link of type \II:
\[
\begin{tikzcd}[link]
Y_1 \ar[rr,dotted,->,"\phi"] \ar[dd,"\pi_1",swap]&& Y_2 \ar[dd,"\pi_2"]\\ \\
X_1  \ar[rr,"\chi",dashed] \ar[rd,swap]&& X_2 \ar[ld]\\
&B 
\end{tikzcd}
\]
and denote $E_1, E_2$ the respective exceptional divisors of $\pi_1, \pi_2$.
Then $\phi_* E_1 \neq E_2$.
\end{lemma}

\begin{proof}
Assume that $\phi_* E_1 = E_2$.
Then $\chi\colon X_1 \ps X_2$ is a pseudo-isomorphism, hence an isomorphism by Lemma~\ref{lem:corti2.7}\ref{corti2.7:1}.
Then Lemma~\ref{lem:corti2.7}\ref{corti2.7:2} implies that the pseudo-isomorphism $\phi\colon Y_1 \ps Y_2$ also is an isomorphism.
The morphisms $Y_1/X_1$ and $Y_1/X_2$ are then divisorial contractions of the same extremal ray,  contradicting the assumption that the diagram was the result of the two-rays game from $Y_1/B$.
\end{proof}

\begin{lemma} \label{lem:B'toB}
Let $X/B$ be a rank~$2$ fibration that factorises through a rank~$1$ fibration $\sigma\colon X \to B'$, with $\dim X - 1 = \dim B' > \dim B$ $($i.e.~through a conic bundle $X/B')$.
Then $\eta\colon B' \to B$ is a klt Mori fibre space, and in particular for each $p \in B$, the fibre $\eta^{-1}(p)$ is covered by rational curves.
\end{lemma}

\begin{proof}
Note that $B'$ is $\Q$-factorial and klt (Proposition~\ref{pro:sing_of_B}).
We need to show that $-K_{B'}$ is $\eta$-ample, and then the fibre $\eta^{-1}(p)$ is covered by rational curves for each $p\in B$ 
by Lemma~\ref{lem:fibresRatConnected}\ref{connected:-KBigNef}, applied with $\Delta = 0$.

By assumption $\rho(B'/B) = 1$, so we only need to show that there exists a contracted curve $C \subseteq B'$ such that $-K_{B'} \cdot C > 0$.
Since $\dim B' > \dim B$, the contracted curves cover $B'$, so we can choose $C$ sufficiently general in a fibre $\eta^{-1}(q)$ of a general point $q\in B$ such that the following holds:

\begin{enumerate}[$(i)$]
\item \label{B'toB:1}
$C$ is not contained in the discriminant locus $\Delta' \subset B'$ of the conic bundle $\sigma\colon X \to B'$;

\item \label{B'toB:2}
The surface $\sigma^{-1}(C)$ does not contain any of the curves $C' \subseteq X$ contracted by $\eta \circ \sigma$ with $-K_X \cdot C' \le 0$.

\item \label{B'toB:3}
The fibre $F=(\eta\circ \sigma)^{-1}(q)$ of $\eta \circ \sigma \colon X \to B$ containing the surface $\sigma^{-1}(C)$ is general, so that $(-K_X)|_F$ is big 
\end{enumerate} 
More precisely, for \ref{B'toB:1} is suffices to choose $\eta^{-1}(q)$ not contained in the hypersurface $\Delta' \subset B'$.
We can ensure \ref{B'toB:2} because by Corollary~\ref{cor:weakFanoFiber} such curves cover at most a codimension 2 subset of $F$.
Finally for \ref{B'toB:3} recall first that since $X/B$ is a rank~$2$ fibration, $-K_X$ is relatively big by \ref{fib:big}.
Moreover the intersection $(-K_X)|_F \cdot \sigma^{-1}(C)$ is a non-trivial effective $1$-cycle.
Indeed, since $(-K_X)|_F$ is big, we can take a large integer $m>0$ and find that $(-mK_X)|_F$ induces a rational morphism contracting no curve on the complement of a divisor of $F$.
It suffices then to choose $C$ such that $\sigma^{-1}(C)$ is not contained in this divisor.

As in \cite[Corollary 4.6]{MoriMukai}, we have $-4K_{B'} \equiv \sigma_*(-K_X)^2 + \Delta$. Intersecting with $C$, we obtain
\begin{align*}
-4K_{B'}\cdot C &=  \sigma_*(-K_X)^2 \cdot C + \Delta\cdot C \\
& \ge (-K_X)|_F \cdot (-K_X)|_F \cdot \sigma^* C  \\
& > 0 \text{ by our choice of $C$}. \qedhere
\end{align*}
\end{proof}

\subsection{Rank \textit{r} fibrations with general fibre a curve}
\label{sec:conicFibrations}

Let $\eta\colon  T \to B$ be a rank~$r$ fibration, with $\dim B = \dim T -1$.
If $\Gamma \subset B$ is an irreducible hypersurface, we define $\eta^\sharp(\Gamma) \subseteq T$ to be the Zariski closure of all fibres of dimension $1$ over $\Gamma$.
The reason for introducing this notion is twofold:
first $B$ might not be $\Q$-factorial, so we cannot consider the pull-back of $\Gamma$ as a $\Q$-Cartier divisor, and second the preimage $\eta^{-1}(\Gamma)$ might contain superfluous components (see Example~\ref{ex:type_I_and_II_div}).  

Now we distinguish two classes of special divisors in $T$, and we shall show in Proposition~\ref{pro:rankrFibrations} below that they account for the relative rank of $T/B$. 
Let $D \subset T$ be a prime divisor.
If $\eta(D)$ has codimension at least 2 in $B$, we say that $D$ is a \emph{divisor of type} \I.
If $\eta(D)$ is a divisor in $B$, and the inclusion $D \subsetneq \eta^{\sharp} (\eta (D))$ is strict, we say that $D$ is a \emph{divisor of type} \II.

\begin{remark}
The similarity between the terminology for Sarkisov links and for special divisors of type \I or \II is intentional.
See Lemma~\ref{lem:LinksCB}\ref{classification2} below. 
\end{remark}

\begin{example}
\label{ex:type_I_and_II_div}
We give an example illustrating the definitions above, which also shows that the inclusion $\eta^\sharp (\Gamma) \subseteq \eta^{-1}(\Gamma)$ might be strict.
For $B$ an arbitrary smooth variety, consider $Y = \p^1 \times B$ with $Y/B$ the second projection.
Let $\Gamma \subset B$ be any irreducible smooth divisor, $D = \p^1 \times \Gamma$ the pull-back of $\Gamma$ in $Y$, $\Gamma' = \{t\} \times \Gamma \subset D$ a section and $p \in D \setminus \Gamma'$ a point.
Let $T \to Y$ be the blow-up of $\Gamma'$ and $p$, with respective exceptional divisors $D'$ and $E$, and denote again $D$ the strict transform of $\p^1 \times \Gamma$ in $T$. 
Then one can check that the induced morphism $\eta\colon T \to B$ is a rank 3 fibration (see Example~\ref{ex:with_flip} for the case $B = \p^2$), $E$ is a divisor of type I, $D \cup D'$ is a pair of divisors of type II, and
\[
\eta^\sharp(\Gamma) = D \cup D' \subsetneq D \cup D' \cup E = \eta^{-1}(\Gamma).
\]
\end{example}

\begin{proposition} \label{pro:rankrFibrations}
Let $\eta\colon  T \to B$ be a rank~$r$ fibration, with $\dim B = \dim T -1$.
\begin{enumerate}
\item \label{rankr:0}
For any rank~$r'$ fibration $T'/B'$ such that $T/B$ factorises through $T'/B'$, any divisor contracted by the birational contraction $T \rat T'$ is a divisor of type \I or \II for $T/B$.
\item \label{rankr:1}
Divisors of type \II always come in pairs: for each divisor $D_1$ of type \II, there exists another divisor $D_2$ of type \II such that
\[
D_1 \cup D_2 =  \eta^{\sharp} (\eta (D_1)) =  \eta^{\sharp} (\eta (D_2)).
\]
\item \label{rankr:2}
If $D_1 \cup D_2$ is a pair of divisors of type \II, and $p$ a general point of $\eta(D_1) = \eta(D_2)$, then $\eta^{-1}(p) = f_1 \cup f_2$ with $f_i \subseteq D_i$, $i =1,2$, some smooth rational curves satisfying
\begin{align*}
K_T \cdot f_i = -1, && D_i\cdot f_i = -1,  && D_1\cdot f_2 = D_2\cdot f_1 = 1.
\end{align*}
\item \label{rankr:3}
Let $D \subset T$ be a divisor of type \I or \II. 
Then there exists a birational contraction over $B$
\[
T \rat X \to B
\]
that contracts $D$ and such that $\rho(X) = \rho(T) -1$. 
\item \label{rankr:4}
Assume $B$ is $\Q$-factorial.
Let $d_1$ $($resp.~$d_2)$ be the number of divisors of type~\I $($resp.~the number of pairs of divisors of type $\II)$.
Then
\[
r = 1 + d_1 + d_2.
\] 
\end{enumerate}
\end{proposition}

\begin{proof}
\ref{rankr:0}.
Assume $D$ is a prime divisor contracted by $T\rat T'$, which is neither of type \I nor of type \II for $T/B$.
So $\eta(D) \subset B$ is a divisor, and $D=\eta^{\sharp}(\eta(D))$.
By running a $D$-MMP over $B$ we produce a sequence of log-flips $T \ps T_1$, and then a divisorial contraction $\pi\colon T_1 \to T_2$ contracting $D$.
Since a log-flip does not change the type of special divisors, without loss of generality we can assume $T = T_1$.
Since $\eta(D) \subset B$ is a divisor, $\pi(D)$ has codimension 2 in $T_2$.
By Lemma~\ref{lem:divContToCodim2}, a general fibre $f$ of $\pi$ is an irreducible curve, and since $D=\eta^{\sharp}(\eta(D))$, we have $f = \eta^{-1}(p)$ for some $p \in \eta(D)$.
So $f$ is proportional to a general fibre of $\eta$, in contradiction with the fact that the extremal contraction of $f$ is divisorial.

\ref{rankr:1} and \ref{rankr:2}.
Let $D_1$ be a divisor of type \II, and let $D_2, \dots, D_s$ be the other divisors of type \II such that
\(
\eta^{\sharp} (\eta (D_1)) = D_1 \cup \dots \cup D_s.
\)
By definition of $\eta^{\sharp}$, for each~$i$ the general fibres of $D_i\to \eta(D_i)$ are curves.
Hence, $\Gamma = \eta(D_i)$ is a hypersurface in $B$, which does not depend on $i$.
Let $p \in \Gamma$ be a general point, and write $f := \eta^{-1} (p) = f_1 + \dots + f_s$ with $f_i$ a curve in $D_i$.
We have $D_i \cdot \eta^{-1} (p) = 0$ for each $i$, $D_i \cdot f_j > 0$ for at least one $j$ by connectedness of $f$, which gives
\(
D_i \cdot f_i < 0.
\)

Then by running a $D_i$-MMP from $T$ over $B$, we obtain a sequence of log-flips that does not affect $f$, and then a divisorial contraction of $D_i$ to a center of codimension~$2$.
By Lemma~\ref{lem:divContToCodim2}, this implies that $f_i$ is smooth with $K_T \cdot f_i = D_i \cdot f_i = -1$.
But $K_T \cdot f = -2$, so we conclude that $s = 2$ as expected.
The equality $D_1 \cdot f_2 = D_2 \cdot f_1 =1$ follows immediately from $D_i \cdot f = 0$.

To prove \ref{rankr:3}, we show that the divisor $D$ is covered by curves $\ell$ such that $D \cdot \ell <0$, and then we get the expected birational contraction by running a $D$-MMP.
When $D$ has type \II we showed in \ref{rankr:2} that $D$ is covered by such curves.
Now let $D$ be a divisor of type \I, $p$ a general point in $\eta(D)$, and let $d \ge 0$ be the dimension of $\eta(D)$.
By definition of a divisor of type \I we have $n -3 \ge d$, where $n = \dim T$.
Now consider a surface $S \subset T$ obtained as
\[
S = \bigcap_{i = 1}^{n-2-d} H_i \cap \bigcap_{j = 1}^{d} \eta^* H'_j,
\]
where the $H_i$ are general hyperplane sections of $T$, and the $H'_j$ general hyperplane sections of $B$ through $p$.
By construction, $\ell := S \cap D$ is an irreducible curve contracted to $p$ by $\eta$.
Moreover $\eta(S)$ is a surface; indeed each $H_i$ is transverse to the general fibres of $\eta$, which are curves, and $n-2-d\ge 1$.
Since a curve contracted by a morphism between two surfaces has negative self-intersection, we obtain $D \cdot \ell = (\ell \cdot \ell)_S < 0$ as expected.

To prove \ref{rankr:4}, first observe that the contraction of a divisor of type \I does not affect the other special divisors, and the contraction of a divisor of type \II only affects the other divisor in the pair, which is not special anymore.
So by applying several times \ref{rankr:3}, we may assume $d_1 = d_2 = 0$, and we want to show $r = 1$, or equivalently, that $T/B$ is a terminal Mori fibre space.
Now we run a MMP from $T$ over $B$.
A flip does not change $d_1$ nor $d_2$, so we can assume that we have a divisorial contraction or a Mori fibre space.
By \ref{rankr:0}, a divisorial contraction would contradict our assumption $d_1 = d_2 = 0$.
On the other hand, if $T \to B'$ is a Mori fibre space,
then both $B'$ and $B$ are $(n-1)$-dimensional varieties, and $B'$ is $\Q$-factorial klt by Proposition~\ref{pro:sing_of_B}.
If the birational morphism $B' \to B$ is not an isomorphism, it must contracts at least one divisor $D$ because $B$ is $\Q$-factorial by assumption.
By Lemma \ref{lem:B'_over_B_is_Mds} $B'/B$ is a Mori dream space, so we can run a $D$-MMP from $B'$ over $B$.
After a sequence of $D$-flips this produces a divisorial contraction, hence a divisor of type \I on $T$ by pulling-back, and again a contradiction.
In conclusion, $B' \iso B$ is an isomorphism and $T/B$ is a Mori fibre space, as expected.
\end{proof}

\begin{lemma} \label{lem:gonalityOfTypeII}
Let $\eta\colon T \to B$ be a rank~$r$ fibration with $\dim B = \dim T -1$. 
Assume that $D$ is a divisor of type \II for $T/B$, with $\covgon(\eta(D))> 1$ $($i.e.~$\eta(D)$ is not uniruled$)$.
Then for any rank~$r'$ fibration $T'/B'$ that factorises through $T/B$, with $\dim B'=\dim T'-1=\dim B$, the strict transform of $D$ is a divisor of type \II for $T'/B'$.
\end{lemma}

\begin{proof}
Recall that $T'\rat T$ is a birational contraction and $\pi\colon B\to B'$ is a morphism with connected fibres between klt pairs 
(Definition~\ref{def:rankFibration}\ref{fib:singB}), which in our situation is birational as $\dim(B)=\dim(B')$. 
We write $D=D_1$ and by Proposition~\ref{pro:rankrFibrations}\ref{rankr:1} we have a pair $D_1\cup D_2$ of divisors of type \II for $T/B$, where $\Gamma=\eta(D_1)=\eta(D_2)$ is a divisor of $B$ and $D_1 \cup D_2 =  \eta^{\sharp} (\Gamma)$.

We first observe that the image of $\Gamma$ in $B'$ is again a divisor $\Gamma'\subset B'$. 
Indeed otherwise, the divisor $\Gamma \subset B$ is one of the divisors contracted by the birational morphism $\pi\colon B \to B'$.
By Lemma~\ref{lem:fibresRatConnected}\ref{connected:birational}, this implies that $\Gamma$ is covered by rational curves, in contradiction with our assumption $\covgon(\Gamma) > 1$.

Writing $\eta'\colon T'\to B'$ the rank~$r'$ fibration, one observe that the strict transforms  $\tilde{D}_1$ and $\tilde{D}_2$ of $D_1$ and $D_2$ are such that $\tilde{D}_1\cup \tilde{D}_2\subseteq \eta^{\sharp}(\Gamma')$. Hence, $\tilde{D}_1$ and $\tilde{D}_2$ are divisors of type \II for $T'/B'$.
\end{proof}

\begin{lemma}
\label{lem:factorThroughRank1}
Let $T/B$ be a rank~$r$ fibration with $\dim B = \dim T -1$ and $B$ $\Q$-factorial.
Assume that for each divisor $D$ of type \II for $T/B$, we have 
\[\covgon(\eta(D)) > 1.\]
Then $T/B$ factorises through a rank~$1$ fibration $T'/B'$ such that $T \ps T'$ is a pseudo-isomorphism if and only if $T/B$ does not admit any divisor of type \II.  

If this holds, then $\dim B' = \dim T - 1$, $B' \to B$ is a birational morphism and $\rho(B'/B) = r-1$.  
\end{lemma}

\begin{proof}
If $T/B$ factorises through a rank~$r'$ fibration $T'/B'$ such that $T \ps T'$ is a pseudo-isomorphism, first observe that $\rho(B'/B) = r - r'$, and $B'\to B$ is birational, since $\dim(B)=\dim(B')$, which follows from  
\[\dim(T)=\dim(T')>\dim(B')\ge\dim(B)=\dim(T)-1.\]

If $D_1\cup D_2$ is a pair of divisors of type \II for $T/B$, then their strict transforms $\tilde{D}_1,\tilde{D}_2$ have the same image in $B'$, which is a divisor because $B'\to B$ is birational. 
So if $T/B$ admits at least one divisor of type \II, then by Proposition~\ref{pro:rankrFibrations}\ref{rankr:2} some fibres of $T'/B'$ have the form $f_1 + f_2$ with $f_1, f_2$ non proportional.
In particular $r' = \rho(T'/B') \ge 2$ and so $T'/B'$ is not a Mori fibre space.

To prove the converse, we assume that $T/B$ does not admit any divisor of type \II, and we proceed by induction on the number $d_1$ of divisors of type \I.
If $d_1 = 0$ then by Proposition~\ref{pro:rankrFibrations}\ref{rankr:4},  $T/B$ is already a rank~$1$ fibration, so we just take $T'/B' = T/B$.
Now if $d_1 > 0$, by Proposition~\ref{pro:rankrFibrations}\ref{rankr:3} there exists a birational contraction over $B$, $T \rat X_1 \to B$, which contracts one divisor $D$ of type \I. 
Since the contraction is obtained by running a $D$-MMP, in fact it factorises as $T \ps T_1 \to X_1$, where $T \ps T_1$ is a sequence of $D$-flips and $T_1 \to X_1$ is a divisorial contraction.
Then by induction hypothesis $X_1/B$ factorises through a rank~$1$ fibration $X_2/B_2$ with $X_1 \ps X_2$ a pseudo-isomorphism (here we use Lemma~\ref{lem:gonalityOfTypeII}, which shows that $X_1/B$ does not admit any divisor of type \II).
By Lemma~\ref{lem:2raysOverFlip}, there exist a pseudo-isomorphism $T_1 \ps T_2$ and a divisorial contraction $T_2 \to X_2$ that makes the diagram on Figure~\ref{fig:contractsTypeI} commute.
Finally we play the two-rays game $T_2/B_2$.
Since $T_2/B_2$ admits one divisor of type \I and no divisor of type \II (by our assumption on the covering gonality and by Lemma~\ref{lem:gonalityOfTypeII}), the other side of the two-rays game must be a Mori fibre space, which gives the expected rank~$1$ fibration $T'/B'$.
\end{proof}

\begin{figure}[ht]
\[
\begin{tikzcd}[link]
T \ar[ddddddr] \ar[r,dotted] & T_1 \ar[r,dotted] \ar[dd] & T_2 \ar[r,dotted] \ar[dd] & T ' \ar[dd] \\ \\
  & X_1 \ar[dddd] \ar[r,dotted] & X_2 \ar[dd] & B' \ar[ldd] \\ \\
  &     & B_2 \ar[ddl] \\ \\
  & B
\end{tikzcd}
\]
\caption{} \label{fig:contractsTypeI}
\end{figure}

\subsection{Sarkisov links of conic bundles}
\label{sec:SarkisovLinksCB}

In this subsection, by applying Proposition~\ref{pro:rankrFibrations} to the case $r=2$, we classify Sarkisov links of conic bundles. 

\begin{lemma}
\label{lem:LinksCB}
Let $Y/B$ be a rank~$2$ fibration with $\dim B = \dim Y -1$, and $\chi$ the associated Sarkisov link, well-defined up to taking inverse.
\begin{enumerate}
\item \label{classification1}
$\chi$ has type \IV if and only if $B$ is not $\Q$-factorial.
\item \label{classification2}
If $B$ is $\Q$-factorial, let $d_1$ $($resp.~$d_2)$ be the number of special divisors of type \I $($resp.~of type \textup{II}$)$ for $Y/B$. Then
\begin{itemize}
\item $\chi$ has type \I or \III if and only if $(d_1,d_2)=(1,0)$.
\item $\chi$ has type \II if and only if $(d_1,d_2)=(0,1)$.
\end{itemize}
\end{enumerate}
\end{lemma}

\begin{proof}
\ref{classification1}.
If $B$ is not $\Q$-factorial, then it follows directly that $\chi$ has type \IV, from the fact that the base of a terminal Mori fibre space if always $\Q$-factorial (Proposition~\ref{pro:sing_of_B}), and by inspection of the diagrams in Figure~\ref{fig:SarkisovTypes}.
Conversely, assuming that $\chi\colon X_1/B_1 \rat X_2/B_2$ is a link of type \IV, we show that $B$ is not $\Q$-factorial. As $\dim B=\dim Y-1$, the morphisms $B_1/B$, $B_2/B$ are birational.
If $B$ is $\Q$-factorial, then $B_1/B$ and $B_2/B$ are birational contractions with respective exceptional divisors $E_1$ and $E_2$.
If the birational map $B_1 \rat B_2$ sends $E_1$ onto $E_2$, then the map is a pseudo-isomorphism, hence an isomorphism by Lemma~\ref{lem:corti2.7}\ref{corti2.7:2}, and then $X_1 \ps X_2$
also is an isomorphism by Lemma~\ref{lem:corti2.7}\ref{corti2.7:1}, a contradiction.
Otherwise, the pull-backs of $E_1$, $E_2$ together with the choice of any ample divisor give three independent classes in $N^1(Y/B)$, in contradiction with $\rho(Y/B) = 2$.  

To prove \ref{classification2}, first we observe that Proposition~\ref{pro:rankrFibrations}\ref{rankr:4} gives $d_1 + d_2 = 1$, hence the two possibilities $(d_1,d_2) = (1,0)$ or $(0,1)$.
Recall also from Proposition~\ref{pro:rankrFibrations}\ref{rankr:0} that any divisor contracted by a  birational contraction from $Y$ over $B$ must be of type \I or \II.  
If the  link $\chi$ is of type \II, then  Lemma~\ref{lem:distinctDivisors} gives two birational contractions from $Y$ contracting distinct prime divisors, and this is possible only in the case $(d_1,d_2) = (0,1)$ where there is a pair of divisors of type \II available. Conversely, if $(d_1,d_2)=(0,1)$, we have two distinct prime divisors, that we can contract via two distinct birational contractions (Proposition~\ref{pro:rankrFibrations}\ref{rankr:3}).  These are the two starting moves of a $2$-ray game which provides a link of type \II.
\end{proof}

\begin{corollary}\label{cor:SarkiIConic}
Let $\chi$ be a Sarkisov link of conic bundles of type \I:
\[
\begin{tikzcd}[link]
Y_1 \ar[dd,"\pi_1",swap]  \ar[rr,dotted,-] && X_2 \ar[dd,"\eta_2"] \\ \\
X_1 \ar[uurr,"\chi",dashed,swap] \ar[dr,"\eta_1",swap] &  & B_2 \ar[dl] \\
& B_1 &
\end{tikzcd}
\]
Let $E_1$ be the exceptional divisor of the divisorial contraction $\pi_1$.
Then $\eta_1 \circ \pi_1(E_1)$ has codimension at least $2$ in $B_1$.
\end{corollary}

\begin{proof}
Follows from the fact that $E_1$ is a divisor of type \I for $Y_1/B_1$.
\end{proof}

\begin{remark}\label{rem:TypeIVhigher}
There are examples of link of type \IV as in Lemma~\ref{lem:LinksCB}\ref{classification1} only when $\dim B \ge 3$, hence $\dim Y \ge 4$.
See the discussion on the two subtypes of type IV links in \cite[p.~391 after Theorem 1.5]{HMcK}.
For instance, take $B_1$ and $B_2$ that differ by a log-flip, and $B$ the non $\Q$-factorial target of the associated small contractions. 
Then the birational map from $(\p^1\times B_1)/B_1$ to $(\p^1\times B_2)/B_2$ induced by the log-flip is a link of type \IV.
\end{remark}

Now we focus on the case of Sarkisov links of conic bundles of type \II.
First we introduce the following definition.

\begin{definition}
\label{def:markedConic}
A \emph{marked conic bundle} is a triple $(X/B,\Gamma)$, where $X/B$ is a conic bundle in the sense of Definition \ref{def:conicBundle}, and $\Gamma \subset B$ is an irreducible hypersurface, not contained in the discriminant locus of $X/B$ $($i.e.~the fibre of a general point of $\Gamma$ is isomorphic to $\p^1)$. 
The \emph{marking} of the marked conic bundle is defined to be $\Gamma$.

We say that two marked conic bundles $(X/B,\Gamma)$, and $(X'/B',\Gamma')$ are  \emph{equivalent} if there exists a commutative diagram
\[\begin{tikzcd}[link]
X \ar[dd,swap]\ar[r,"\psi",dashed]& X'\ar[dd]  \\ \\
B  \ar[r,"\theta",dashed]& B' 
\end{tikzcd}\]
where $\psi, \theta$ are birational and such that the restriction of $\theta$ induces a birational map $\Gamma\rat \Gamma'$ between the markings.
In particular, if $(X/B,\Gamma)$, and $(X'/B',\Gamma')$ are equivalent, then the conic bundles $X/B$ and $X'/B'$ are equivalent in the sense of Definition \ref{def:conicBundle}.

For each variety $Z$, we denote by $\CB(Z)$ the set of equivalence classes of conic bundles $X/B$ with $X$ birational to $Z$  and denote, for each class of conic bundles $C\in \CB(Z)$ by $\MC(C)$ the set of equivalence classes of marked conic bundles $(X/B,\Gamma)$ where $C$ is the class of $X/B$.
\end{definition}

The next lemma explains how a Sarkisov link of conic bundles of type \II gives rise to an equivalence class of marked conic bundles.

\begin{lemma}\label{lem:SarkiIIConic}
Let $\chi$ be a Sarkisov link of conic bundles of type \II between varieties of dimension $n \ge 2$.
Recall that $\chi$ fits in a commutative diagram of the form
\[
\begin{tikzcd}[link]
Y_1 \ar[rr,dotted,"\phi"] \ar[dd,"\pi_1",swap]&& Y_2 \ar[dd,"\pi_2"]\\ \\
X_1  \ar[rr,"\chi",dashed] \ar[rd,"\eta_1",swap]&& X_2 \ar[ld,"\eta_2"]\\
&B 
\end{tikzcd}
\]
where $X_1, X_2, Y_1, Y_2$ are $\Q$-factorial terminal varieties of dimension $n$, $B$ is a $\Q$-factorial klt variety of dimension $n-1$, $\phi$ is a sequence of log-flips over $B$, and each $\pi_i$ is a divisorial contraction with exceptional divisor $E_i \subset Y_i$ and centre $\Gamma_i = \pi_i(E_i) \subset X_i$.

Then there exists an irreducible hypersurface $\Gamma \subset B$ $($of dimension $n-2)$ such that 
\begin{enumerate}
\item \label{SarkiII:1}
for $i=1,2$, the centre $\Gamma_i = \pi_i(E_i)$ has codimension $2$ in $X_i$, and the restriction $\eta_i|_{\Gamma_i}\colon \Gamma_i \to \Gamma$ is birational.
In particular, for each $i$ we have $\eta_i \circ \pi_i (E_i) = \Gamma$, and the marked conic bundles $(X_1/B,\Gamma)$ and $(X_2/B,\Gamma)$ are equivalent.
\item \label{SarkiII:4}
Let $Y$ be equal to $Y_1$, $Y_2$, or any one of the intermediate varieties in the sequence of log-flips $\phi$. 
Then $E_1\cup E_2$ is a pair of divisors of type \II for $Y/B$.
\item \label{SarkiII:2}
$\Gamma$ is not contained in the discriminant locus of $\eta_1$, or equivalently of $\eta_2$, which means that a general fibre of $\eta_i\colon \eta_i^{-1}(\Gamma)\to \Gamma$ is isomorphic to $\p^1$.
\item \label{SarkiII:3}
At a general point $x \in \Gamma_i$, the fibre of $X_i/B$ through $x$ is transverse to $\Gamma_i$. 
\end{enumerate}
\end{lemma}

\begin{proof}
\ref{SarkiII:1} and \ref{SarkiII:4}.
By Lemma~\ref{lem:LinksCB}, $Y_1/B$ admits no divisor of type \I, and exactly one pair of divisors of type \II.
By Lemma~\ref{lem:distinctDivisors} we have $\phi_* E_1 \neq E_2$, so the birational contractions $Y_1 \rat X_1$ and $Y_1 \rat X_2$ contract distinct divisors.
It follows from Proposition~\ref{pro:rankrFibrations} that the pair of divisors of type \II is $E_1 \cup E_2$. 
So by definition $E_1$ and $E_2$ projects to the same hypersurface $\Gamma \subset B$. 
By Proposition~\ref{pro:rankrFibrations}\ref{rankr:2} both finite maps $\Gamma_i\to \Gamma$ are birational, otherwise the fibre in $Y_i$ over a general point of $\Gamma$ would have more than two components.

\ref{SarkiII:2} and \ref{SarkiII:3} follow from Proposition~\ref{pro:rankrFibrations}\ref{rankr:2}.
Indeed if $\Gamma$ was in the discriminant locus of $\eta_1$ then the preimage in $Y_1$ of a general point $p \in B$ would have 3 irreducible components, instead of 2. 
Moreover writing $f_1 \cup f_2$ the fibre through $x$, with $f_i \subseteq E_i$, the fact that the fibre is transverse to $\Gamma_i$ is equivalent to $f_1 \cdot E_2 = f_2 \cdot E_1 = 1$.
\end{proof}

\begin{definition}\label{def:marked_cb_link}
By Lemma~\ref{lem:SarkiIIConic}\ref{SarkiII:1}, to each Sarkisov link of conic bundles of type \II $\chi\colon X_1\rat X_2$, 
we can associate the equivalence class of the marked conic bundle $(X_1/B,\Gamma)$ given in this lemma. 
We define the \emph{marking} of $\chi$ to be $\Gamma\subset B$.
We say that two Sarkisov links of conic bundles of type \II are \emph{equivalent} if their corresponding marked conic bundles are equivalent.
\end{definition}

We also extend the notion of covering gonality (see \S\ref{sec:gonality}) to Sarkisov links of conic bundles of type \II.

\begin{definition}\label{def:Covgonchi}
Let $\chi$ be a Sarkisov of conic bundles of type \II between varieties of dimension $n\ge 3$.
We define $\covgon(\chi)$ to be $\covgon(\Gamma)$, where $\Gamma$ is the marking of $\chi$.
\end{definition}

\begin{remark}
If two Sarkisov links of conic bundles of type \II are equivalent, then their markings are birational  to each other.
In particular the number $\covgon(\chi)$ only depends on the equivalence class of $\chi$.

The above definition makes sense if the varieties $X_i$ have dimension $\ge 2$, but it is not a very good invariant if the dimension is $2$, as the centre is always a point, and there is only one class of marked conic bundles, given by a point in the base of a Hirzebruch surface. 
However, the analogue definition over $\Q$ or over a finite field, instead of over $\C$, is interesting even for surfaces.
\end{remark}

\section{Relations between Sarkisov links}\label{sec:RelSarkisov}

The fact that one can give a definition of Sarkisov links in terms of relative Mori dream spaces of Picard rank 2 as in the previous section was independently observed in \cite[\S 2]{AZ} and \cite[\S 2.3]{LZ17}.
Our next aim is to extend this observation to associate some relations between Sarkisov links to each rank 3 fibration.
First we define elementary relations, and then we relate this notion to the work of A.-S. Kaloghiros about relations in the Sarkisov programme.  

\subsection{Elementary relations}

\begin{definition}\label{def:Tequivalent}
Let $X/B$ and $X'/B'$ be two rank~$r$ fibrations, and $T \rat X$, $T \rat X'$ two birational maps from the same variety $T$.
We say that $X/B$ and $X'/B'$ are \emph{$T$-equivalent} (the birational maps being implicit) if there exist  a pseudo-isomorphism $X \ps X'$ and an isomorphism $B \iso B'$ such that all these maps fit in a commutative diagram:
\[
\begin{tikzcd}[link]
 & T \ar[dl, dashed] \ar[dr, dashed] \\
X \ar[dd] \ar[rr,dotted] &  & X' \ar[dd] \\ \\
B  \ar[rr,"\sim"] & & B'  \\ 
\end{tikzcd}
\]
\end{definition}

One should think of the maps $T \rat X$ and $T \rat X'$ as providing a marking with respect to a prefered model variety $T$.
See \S\ref{sec:graph} for an illustration of this point of view.
In particular, we do not assume $T \rat X$ and $T \rat X'$ to be birational contractions, even if it happens to be the case in the proof of the following lemma. 

\begin{lemma}
\label{lem:2 rank 2}
Let $X_3/B_3$ be a rank~$3$ fibration that factorises through a rank~$1$ fibration $X_1/B_1$.
Then up to $X_3$-equivalence there exist exactly two rank~$2$ fibrations that factorise through $X_1/B_1$, and that are dominated by $X_3/B_3$.
\end{lemma}

\begin{proof}
We distinguish three cases according to $\rho(B_1/B_3)$.

If $\rho(B_1/B_3) = 2$, then $B_1$ -- being the base of a klt Mori fibre space -- is $\Q$-factorial klt (Proposition~\ref{pro:sing_of_B}), and $B_1/B_3$ is a Mori dream space by Lemma~\ref{lem:B'_over_B_is_Mds}.
The associated two-rays game yields exactly two non-isomorphic $B_2, B_2'$ with $\rho(B_2/B_3) = \rho(B_2'/B_3) = 1$.
Then Lemma~\ref{lem:2raysOverFlip} provides sequences of log-flips over $B_3$, $X_1 \ps X_2$ and $X_1 \ps X_2'$, such that $X_2/B_2$, $X_2'/B_2'$ are the expected rank 2 fibrations. 

If $\rho(B_1/B_3) = 1$, then the base $B_2$ of any of the expected rank 2 fibrations must be equal to $B_1$ or $B_3$, because by assumption we have morphisms $B_1 \to B_2 \to B_3$.
By Lemma~\ref{lem:StableUnderMMP}\ref{stab_under_MMP:1} $X_1/B_3$ is the first expected rank 2 fibration, and up to equivalence it is the only one with base $B_3$, because any rank 2 fibration $X_2/B_3$ satisfies $\rho(X_2) = \rho(X_1)$, so the birational contraction $X_2 \rat X_1$ is a pseudo-isomorphism. 
Let $D$ be the pull-back on $X_3$ of an ample divisor on $X_1$. The birational contraction $X_3 \rat X_1$ is a $D$-MMP over $B_3$, and as $\rho(X_3)-\rho(X_1)=1$, it decomposes as a sequence of $D$-flips $X_3 \ps X_3'$, a divisorial contraction $X_3' \to X_1'$, and a sequence of $D$-flips $X_1' \ps X_1$. 
Then Lemma~\ref{lem:2raysOverFlip} provides a sequence of log-flips over~$B_3$, $X_3' \ps X_2$, such that $X_2\to X_1$ is a divisorial contraction, and by Lemma~\ref{lem:StableUnderMMP} $X_2/B_1$ is the second expected  rank 2 fibrations.
Any other rank 2 fibration $X_2'/B_1$ satisfying the lemma is equivalent to $X_2/B_1$, because as before the condition on Picard numbers forces $X_2 \rat X_2'$ to be a pseudo-isomorphism. 

If $\rho(B_1/B_3) = 0$, then $\rho(X_3) - \rho(X_1) = 2$, and $B_1 = B_3$ must be the base of any of the expected rank 2 fibrations.
By applying several times Lemma~\ref{lem:2raysOverFlip} we construct a sequence of log-flips over $B_3$, $X_3 \ps X_3'$, such that $X_3' \to X_1$ is a morphism.
The associated two-rays game yields exactly two divisorial contractions $X_2 \to X_1$ and $X_2' \to X_1$. 
Moreover $X_2$ and $X_2'$ are not pseudo-isomorphic by Lemma~\ref{lem:2rays2Divisors}, and are uniquely determined up to equivalence by Lemma~\ref{lem:corti2.7}\ref{corti2.7:2}.
Then $X_2/B_1$ and $X_2'/B_1$ are the expected rank 2 fibrations.
\end{proof}

\begin{proposition} \label{pro:from T3}
Let $T/B$ be a rank~$3$ fibration.
Then there are only finitely many Sarkisov links $\chi_i$ dominated by $T/B$, and they fit in a relation
\[
\chi_t \circ \dots \circ \chi_1 = \id.
\]
\end{proposition}

\begin{proof}
Since $T/B$ is a Mori dream space, by Lemma~\ref{lem:anyMMP} there are only finitely many rank~$1$ or 2 fibrations dominated by $T/B$.
We construct a bicolored graph $\Gamma$ as follows. 
Vertices are rank~$1$ or 2 fibrations dominated by $T/B$, up to $T$-equivalence, and we put an edge between $X_2/B_2$ and $X_1/B_1$ if $X_2/B_2$ is a rank 2 fibration that factorises through the rank~$1$ fibration $X_1/B_1$.
By construction, two vertices of rank~$1$ of $\Gamma$ are at distance 2 if and only if there is a Sarkisov link between them.
Then by Lemmas~\ref{lem:rank 2 and link} and \ref{lem:2 rank 2} we obtain that $\Gamma$ is a circular graph, giving the expected relation.
\end{proof}

\begin{definition}\label{def:elementary relation}
In the situation of Proposition~\ref{pro:from T3}, we say that 
\[
\chi_t \circ \dots \circ \chi_1 = \id
\]
is an \emph{elementary relation} between Sarkisov links, coming from the rank 3 fibration $T/B$.
Observe that the elementary relation is uniquely defined by $T/B$, up to taking the inverse, cyclic permutations and insertion of isomorphisms. 
\end{definition}

\subsection{Geography of ample models}
\label{sec:geography}

In this section, we recall some preliminary material from \cite{BCHM,HMcK,KKL}.
The aim is to explain the construction of a polyhedral complex attached with the choice of some ample divisors on a smooth variety, and to state some properties (Proposition~\ref{pro:coneC} and Lemma \ref{lem:faces_factory}) that we will use in the next section to understand relations between Sarkisov links.

\begin{definition}[{\cite[Definition 3.6.5]{BCHM}}] \label{def:amplesemiample}  
Let $Z$ be a terminal $\Q$-factorial variety, $D$ be a $\R$-divisor on $Z$ and $\phi\colon Z\rat Y$ a dominant rational map to a normal variety $Y$. 
We take a resolution 
\[
\begin{tikzcd}[link]
& W \ar[dl,"p",swap] \ar[dr,"q"]\\
Z \ar[rr, dashed,"\phi"]& &Y
\end{tikzcd}
\]
where $W$ is smooth, $p$ is a birational morphism and $q$ a morphism with connected fibres.
We say that \emph{$\phi$ is an ample model of $D$} if there exists an ample divisor $H$ on $Y$ such that $p^*D$ is linearly equivalent 
to $q^*H+E$  where $E\ge 0$, and if for each effective $\R$-divisor $R$ linearly equivalent to $p^*D$ we have $R\ge E$.

If $\phi$ is a birational contraction, we say that $\phi$ is \emph{a semiample model of $D$} if $H=\phi_*D$ is semiample (hence in particular $\R$-Cartier) and if $p^*D=q^*H+E$ where $E\ge 0$ is $q$-exceptional. 
\end{definition}

We recall some properties related to these notions.
The first lemma gives some direct consequences of the definition of a semiample model, we leave the proof to the reader (hint: use the Negativity Lemma).

\begin{lemma} \label{lem:direct_properties}
 Let $\phi \colon Z \rat Y$ be a birational contraction between $\Q$-factorial varieties. 
\begin{enumerate}
\item \label{comp:model} 
For any $D_Y\in \Nef(Y)$, $\phi$ is a semiample model of $\phi^* D_Y$.

\item \label{comp:combination}
If $\{D_i\}$ is a finite collection of classes in $N^1(Z)$ such that $\phi$ is a semiample model of each, then $\phi$ is a semiample model for any convex combination of the $D_i$.

\item  \label{comp:composition}
If $\phi'\colon Y \rat Y'$ is a birational contraction to a $\Q$-factorial varietiy, and $\phi' \circ \phi$ is the ample model of a divisor $D$ on $Z$, then $\phi'$ is the ample model of $\phi_* D$.
\end{enumerate}
\end{lemma}

\begin{lemma}[{\cite[Lemma 3.6.6]{BCHM}}] \label{lem:semiampleample} 
Let $Z$ be a terminal $\Q$-factorial variety and $D$ a $\R$-divisor on $Z$.
\begin{enumerate}
\item \label{uniqueAmple}
If $\phi_i\colon Z\rat Y_i$, $i=1,2$, are two ample models of $D$, there exists an isomorphism $\theta\colon Y_1\iso Y_2$ such that $\phi_2=\theta\circ \phi_1$.
\item\label{SemiampleFactorisation}
If a birational map $\psi\colon Z\rat X$ is a semiample model of $D$, the ample model $\phi\colon Z\rat Y$ exists and $\phi=\theta\circ \psi$ for some morphism $\theta\colon X\to Y$. 
Moreover, $\psi_* D =\theta^* H$, where $H$ is the ample divisor $H=\phi_* D$.
\item\label{birsemiampleample}
A birational map $\phi\colon Z\rat Y$ is the ample model of $D$ if and only if it is a semiample model of $D$ and $\phi_* D$ is ample.
\end{enumerate}
\end{lemma}

Note that composing with an isomorphism of the target does not change the notion of ample or semiample model, so it is natural to say that two ample or semiample models $\phi_1\colon Z\rat Y_1$, $\phi_2\colon Z\rat Y_2$ are equivalent if there is an isomorphism $\theta\colon Y_1\iso Y_2$ such that $\phi_2=\theta\circ \phi_1$. 
Then Lemma \ref{lem:semiampleample}\ref{uniqueAmple} says that up to equivalence, if an ample model exists then it is unique.
This justifies that we can speak of \emph{the} ample model of a divisor $D$.

\begin{definition}
We say that two divisors $D$ and $D'$ are \emph{Mori equivalent} if they have the same ample model.
\end{definition}

\begin{remark}
\label{rem:Proj}
For a $\Q$-divisor, the ample model of $D$, if it factorises through a semiample model, is the rational map $\phi_D$ associated with the linear system $\lvert mD\rvert $ for $m \gg 0$, whose image is $Z_D= \Proj (\bigoplus_{m} H^0(Z,mD))$, where the sum is over all positive integers $m$ such that $mD$ is Cartier (see \cite[Remark 2.4\textup{$(ii)$}]{KKL}). 
It does exist if the ring $\bigoplus_{m} H^0(Z,mD)$ is finitely generated, which is for instance true if $D=K_Z+A$ for some ample $\Q$-divisor $A$ (follows from \cite[Corollary 1.1.2]{BCHM}).
\end{remark}

\begin{setup} \label{setup:coneC1}
Let $Z$ be a smooth variety with $K_Z$ not pseudo-effective and let $A_1,\ldots,A_s$ be ample $\Q$-divisors that generate the $\R$-vector space $N^1(Z)$.
Assume that there exist ample effective $\Q$-divisors $A, A_1', \dots, A_s'$ such that for each $i$, $A_i = A + A_i'$.
Define
\begin{multline*}
\Cl= \Bigl\{D\in \Div(Z)_\R \Bigm| D=a_0  K_Z + \sum_{i=1}^s a_i A_i, \\ a_0,\ldots,a_s\ge 0 \text{ and }D \text{ is pseudo-effective}\Bigr\}.
\end{multline*}
Then every element of $\Cl$ has an ample model, and the Mori equivalence classes give a partition \[\Cl=\coprod_{i\in I} \Al_i.\]
For each $i \in I$ we denote by $\phi_i\colon Z\rat Z_i$ the common ample model of all $D \in \Al_i$. 
\end{setup}

Let $V_\Q$ be a $\Q$-vector space, and $V_\R = V_\Q \otimes \R$ the associated real vector space.
Recall that a \textit{rational polytope} in $V_\R$ is the convex hull of finitely many points lying in $V_\Q$. 
In particular, it is convex and compact.

\begin{proposition} \label{pro:polyhedral}
Assume Set-Up~$\ref{setup:coneC1}$.
Then the index set $I$ is finite, the set $\Cl$ is a cone over a rational polytope, and each $\Al_i$ is a finite union of relative interiors of cones over rational polytopes.   
\end{proposition}

\begin{proof}
This follows from \cite[Theorem 3.3]{HMcK}. 
Indeed, we can apply their result with (in their notation) the affine subspace $V \subset \Div(Z)_\R$ generated by $A_1',\ldots,A_s'$ and $-A$.
Observe that they normalise their divisors by $a_0 = 1$, and they put a log-canonical condition on the log-pairs, so they work with an affine section of a subset of our cone $\Cl$.
Precisely, by choosing representatives for the $A_i$ with simple normal crossing support and very small positive coefficients, we can obtain all divisors of the form
\[ D = a_0 (K_Z + \sum_{i=1}^s \frac{a_i}{a_0} A_i) \]
for $a_0$ greater than a given constant $\eps_0 > 0$.
When $\eps_0$ is sufficiently small the missing divisors are all ample, so our cone $\Cl$ minus a small portion of the chamber containing the ample divisors correspond to their cone.
In a moment we will work up to numerical equivalence, and this awkward issue will disappear: see Set-Up~\ref{setup:coneC} and Remark~\ref{rem:klt}.
\end{proof}

We say that a Mori chamber has \emph{maximal dimension} if it spans $N^1(Z)$.

Recall that a \emph{fan} is a collection of rational strongly convex polyhedral cones, such that each face (of any dimension) of a cone is also part of the collection, and such that the intersection of two cones is a face of each.

\begin{lemma}
\label{lem:fan_structure}
The closures of the chambers of maximal dimension yield a fan structure on $\Cl$, which is the same as the fan structure considered in \cite[Theorems 3.2 \& 4.2]{KKL}.
\end{lemma}

\begin{proof}
The fan structure in \cite{KKL}, which slightly generalises \cite[Theorem 4.1]{Ein_al}, is constructed as follows.
One considers the coarsest polyhedral decomposition $\Cl = \bigcup \Cl_i$ such that for any geometric valuation $\Gamma$, the asymptotic order function $o_\Gamma$ is linear in restriction to each $\Cl_i$. 
Moreover this decomposition is a fan by convexity of the $o_\Gamma$.
Then one writes $\Cl = \coprod \Al_i'$ as the disjoint union of the relative interiors of the faces of this fan. 
The Mori chamber $\Al_j$ of ample divisors corresponds to one of the $\Al_i'$, since its closure $\overline{\Al_j} = \Nef(Z) \cap \Cl$ is characterised as the set of divisors in $\Cl$ on which all $o_\Gamma$ vanish.
Then the result follows from \cite[Lemma 2.11]{KKL} combined with the following fact, which can be extracted from the proof of \cite[Corollary 4.4]{KKL} or \cite[Theorem 3.3(4)]{HMcK}:
for each Mori chamber $\Al_i$ of maximal dimension, associated to a birational contraction $\phi_i\colon Z \rat Z_i$, the closure of $\Al_i$ is the intersection of $\Cl$ with the closed convex cone generated by  $\phi_i^* \Nef(Z_i)$ and by the exceptional divisors of $\phi_i$.
\end{proof}

\begin{notation} \label{notation}
We will usually denote $\Fl$ a face of the fan $\Cl$ given by Lemma \ref{lem:fan_structure}, and $\mathring \Fl$ its relative interior.
We emphasise that \cite{HMcK} and \cite{KKL} use the notation $\Al_i$ in a non compatible way, so the reader should be aware of the following convention when checking these references.
\begin{itemize}
\item 
First, our $\Al_i$ are the same as in \cite{HMcK}, and are the Mori chambers defined above.
\item 
We denote $\widebar{\Al_i}$ the closure of $\Al_i$ in the ambiant real vector space, these are the $\Cl_i$ of \cite{HMcK}.
\item 
The relative interior $\mathring \Fl$ of faces are the $\Al_i$ of \cite{KKL} \parent{and also the $\Al_i'$ in the proof of Lemma~\ref{lem:fan_structure}}.
\item 
Our faces $\Fl$ are the $\widebar{\Al_i}$ in \cite{KKL}.
\end{itemize}
\end{notation}

\begin{proposition}\label{pro:coneC}
Assume Set-Up $\ref{setup:coneC1}$. 
Then, the following holds $($each $i,j$ is always assumed to be in $I$ in the next statements$)$:
 \begin{enumerate}
\item\label{coneC:MaxDim}
For each $i$, the following are equivalent:
\begin{enumerate}
\item
The image of $\Al_i$ in $N_1(Z)$ has non-empty interior;
\item
$\phi_i$ is birational and $Z_i$ is $\Q$-factorial;
\item
$\phi_i$ is a birational contraction that is the output of a $(K_Z+\Delta)$-MMP for some $K_Z+\Delta\in \Cl$;
\end{enumerate}

\item\label{coneC:phi_iBirational}
If $\phi_j$ is birational, then $\widebar{\Al_j}$ is a cone over a rational polytope, and we have
 \[\widebar{\Al_j}=\{ D\in \Cl \mid \phi_j \text{ is a semiample model of }D\}.\]
 
\item\label{coneC:phi_ji}
If $i,j$ are such that $\widebar{\Al_j}\cap \Al_i\not=\emptyset$, there exists a morphism $\phi_{ji}\colon Z_j\to Z_i$ with connected fibres such that $\phi_i=\phi_{ji}\circ \phi_j$. 
If moreover $\phi_j$ is birational, we have
\[
\widebar{\Al_j}\cap \widebar{\Al_i}=\{ D\in \widebar{\Al_j}\mid \phi_{j*}D\cdot C=0 \text{ for each }C\in N_1(Z_j/Z_i)\}.
\]

\item\label{coneC:singularities}
For each $i$, the variety $Z_i$ is normal, and there exists an effective $\Q$-divisor $\Delta_i$ such that $(Z_i, \Delta_i)$ is klt.
In particular, it has rational singularities.

\item\label{coneC:numericalonly}
For each numerically equivalent divisors $D,D'\in \Cl$ and each $i$, we have 
\begin{align*}
D\in \Al_i\iff D'\in \Al_i &&
\text{and} &&
D\in \widebar{\Al_i}\iff D'\in \widebar{\Al_i}.
\end{align*}
\end{enumerate}
\end{proposition}

\begin{proof}
\ref{coneC:MaxDim}.
\cite[Theorem 3.3\parent{3}]{HMcK}.

\ref{coneC:phi_iBirational}.
\cite[Theorem 4.2\parent{1}\&\parent{4}]{KKL}.

\ref{coneC:phi_ji}.
The first claim is \cite[Theorem 3.3\parent{2}]{HMcK} or \cite[Theorem 4.2\parent{3}]{KKL} 
The second claim follows from \ref{coneC:phi_iBirational} and the Negativity Lemma.

\ref{coneC:singularities}.
The variety $Z_i$ is normal by definition of an ample model. 
If $\Al_i$ satisfies the equivalent conditions of \ref{coneC:MaxDim}, then $Z_i$ is terminal as the output of a $(K_Z+\Delta)$-MMP.
Otherwise, there exists a chamber $\Al_j$ satisfying the conditions of \ref{coneC:MaxDim}, such that $\widebar{\Al_j} \cap \Al_i \neq\emptyset$. 
So $Z_j$ is $\Q$-factorial and terminal, by \ref{coneC:phi_ji} there is a contraction $\phi_{ji}\colon Z_j\to Z_i$, and so the claim follows from Proposition~\ref{pro:sing_of_Proj} and Remark~\ref{rem:Proj}. 

\ref{coneC:numericalonly}.
\cite[Lemma 3.11]{KKL} is the big case.
In the non-big case, we pick $\Al_j$ of maximal dimension such that $\widebar{\Al_j}\cap \Al_i\not=\emptyset$.
By \ref{coneC:phi_ji} there exists $\phi_{ji}\colon Z_j \to Z_i$, so we reduce to the big case by pulling back divisors to $Z_j$.
\end{proof}

\begin{setup} \label{setup:coneC}
Let $Z$ be a smooth variety with $K_Z$ not pseudo-effective and let $A_1,\ldots,A_s$ be ample $\Q$-divisors that generate the $\R$-vector space $N^1(Z)$.
We still denote 
\begin{multline*}
\Cl= \Bigl\{D\in N^1(Z) \Bigm| D=a_0  K_Z + \sum_{i=1}^s a_i A_i, \\ a_0,\ldots,a_s\ge 0 \text{ and }D \text{ is pseudo-effective}\Bigr\}.
\end{multline*}
This is the image under the natural map $\Div(Z)_\R \to N^1(Z)$ of the cone from Proposition \ref{pro:coneC}, for some choice of ample effective $\Q$-divisors $A, A_1', \dots, A_s'$ such that for each $i$, $A_i \equiv A + A_i'$.
By Proposition \ref{pro:coneC}\ref{coneC:numericalonly}, the decomposition $\Cl = \coprod_{i\in I} \Al_i$ (hence also its image in $N^1(Z)$) does not depend on such a choice of effective representatives.
So from now on we will work directly in the finite dimensional $\R$-vector space $N^1(Z)$, and use the notation $\Cl, \Al_i$ in this context only.
\end{setup}

\begin{remark} \label{rem:klt}
One advantage of working up to numerical equivalence is that we can always assume that the pairs $(Z,\Delta)$ in Set-Up \ref{setup:coneC} are klt with arbitrary small discrepancies, where $\Delta = \frac{1}{a_0}  \sum_{i=1}^s a_i A_i$.
Indeed, by expressing each $A_i$ as $A_i \equiv \frac1N \sum_{j = 1}^N H_{i,j}$ for some large integer $N$ and some general members $H_{i,j} \in \lvert A_i\rvert$, we can ensure that the union of the supports of the $H_{i,j}$ is a simple normal crossing divisor and that all coefficients appearing in the convex combination $\Delta$ are positive and very small.
\end{remark}

Assuming Set-Up \ref{setup:coneC}, we introduce some terminology. 
Recall that we say that a chamber $\Al_i$ has \emph{maximal dimension} if it has non-empty interior in $N^1(Z)$, which corresponds to the equivalent assertions of Proposition~\ref{pro:coneC}\ref{coneC:MaxDim}.
We say that a chamber $\Al_i$ is \emph{big} if all divisors (or equivalently, one divisor) in $\Al_i$ are big.
By the \emph{codimension} of a face in $\Cl$ we always mean the codimension in $N^1(Z)$ of the smallest vector subspace containing it. 
We will usually denote $\Fl^r$ a face of codimension $r$ in $\Cl$, and $\interior{\Fl}^r$ its relative interior.

We denote by $\partial^+\Cl$  the set of non-big divisors in $\Cl$.
As $\partial^+\Cl$ is the intersection of $\Cl$ with the boundary of the pseudo-effective cone, the set $\partial^+\Cl$ is a closed subset of the boundary of $\Cl$. 
We have $\Al_i\subseteq \partial^+\Cl$ if $\dim Z_i<Z$ and $\Al_i\subseteq \Cl\setminus \partial^+\Cl$ if $\dim Z_i=Z$.

By definition, the cone $\Cl \subset N^1(Z)$ is equal to the intersection of two convex closed cones, namely $\Cl = \Cl' \cap \widebar{\Eff}(Z)$ with $\Cl'$ the convex cone generated by $K_Z$ and the $A_i$.
We will say that a face $\Fl \subseteq \Cl$ is \emph{inner} if it meets the interior of $\Cl'$.
In particular, $\Fl$ is inner if for any $D' \in \interior{\Fl}$, there exists a neighborhood $V$ of $D'$ in $N^1(Z)$ such that $\widebar{\Eff}(Z) \cap V = \Cl \cap V$.
Equivalently, a face is inner if it meets either the interior of $\Cl$ or the relative interior of $\partial^+ \Cl$.

\begin{remark}\label{rem:InnerFace}
If $\Fl$ is an inner face, then for any $D \in \widebar{\Eff}(Z)$ and any $D'\in \interior{\Fl}$, we have $D'+\eps D \in \Cl$ for sufficiently small $\eps \ge 0$.
Indeed with the notation above one can choose $V \subset \Cl'$ a neighborhood of $D'$ such that $\widebar{\Eff}(Z) \cap V = \Cl \cap V$.
Then it suffices to choose $\eps$ such that $D'+\eps D \in V$.
Since $D, D'$ are both pseudo-effective, the segment $[D,D']$ also is contained in the convex cone $\widebar{\Eff}(Z)$, and the claims follows.
\end{remark}

\begin{lemma} \label{lem:faces_factory}~
\begin{enumerate}
\item\label{faces:inner_is_inter}
Any inner face $\Fl^r \subseteq \Cl$ is of the form $\Fl^r= \Fl_{ji} := \widebar{\Al_j}\cap\widebar{\Al_i}$ for some chamber $\Al_j$ of maximal dimension, and some chamber $\Al_i$ containing $\interior{\Fl}^r$.

\item\label{faces:EinFji} 
If $\Fl_{ji}$ is such an inner face, then the vector space
\[V_{ji}:=\{ D\in N^1(Z)\mid \phi_{j*}D\cdot C=0 \text{ for each }C\in N_1(Z_j/Z_i)\}\] 
is spanned by $\Ex(\phi_j)$ and $\phi_i^*\Nef(Z_i)$, has codimension $\rho(Z_j/Z_i)$ in $N^1(Z)$, and $V_{ji}$ is also the vector space spanned by $\Fl_{ji}$.
\end{enumerate}
\end{lemma}

\begin{proof}
\ref{faces:inner_is_inter} follows from \cite[Theorem 4.2\parent{2}]{KKL}.

\ref{faces:EinFji} can be extracted from the proof of \cite[Theorem 3.3\parent{4}]{HMcK}. (In particular the $k$ in their statement is the number of prime divisors in $\Ex(\phi_j)$).
\end{proof}

\begin{notation} \label{not:Fij} 
Lemma \ref{lem:faces_factory}\ref{faces:inner_is_inter} provides the following indexing system for faces. 
Any inner face can be written $\Fl_{ji} := \widebar{\Al_j} \cap \widebar{\Al_i}$, for some chamber $\Al_j$ of maximal dimension and some chamber $\Al_i$ such that $\interior{\Fl}_{ji} \subseteq \Al_i$.
The index $i$ is uniquely defined by this last property, but there might be several possible choices for the index $j$. 
For instance, if we have a log-flip from $Z_j$ to $Z_k$, over a non $\Q$-factorial $Z_i$, we have $\Fl_{ji} = \Fl_{ki}$. 
\end{notation}

\begin{example}\label{ex:P2blownup}
We illustrate the definition of Mori chambers and faces on the simple example of the blow-up $Z \to \p^2$ at two distinct points $p_1$ and $p_2$. Using the notation above, there are eight Mori chambers $\Al_0,\ldots,\Al_7$, corresponding to morphisms $\phi_i\colon Z\to Z_i$, $i=0,\ldots,7$ to the varieties $Z_0=Z$, $Z_1=Z_2=\F_1$, $Z_3=\F_0$, $Z_4=\p^2$, $Z_5=Z_6=\p^1$ and $Z_7=\pt$ in the commutative diagram on Figure~\ref{fig:faces} ($\phi_0$ being the identity). 
The two morphisms $\phi_{14},\phi_{24}\colon \F_1\to \p^2$ are the blow-ups of $p_1,p_2\in \p^2$ respectively, and $\phi_1,\phi_2\colon Z\to \F_1$ are the blow-ups of the images of $p_1$ and $p_2$. The morphisms $\phi_{15},\phi_{26}\colon \F_1\to \p^1$ correspond to the $\p^1$-bundle of $\F_1$ and $\phi_3=\phi_5\times \phi_6\colon Z\to\F_0= \p^1\times \p^1$. 

We give the detail of the relation between these Mori chambers and the faces of the cone $\Cl$ in Figure~\ref{fig:faces}. We denote by $E_1, E_2\subset Z$ the curves contracted onto $p_1,p_2\in \p^2$ respectively, by $L$ the strict transform of the line through $p_1$ and $p_2$, and by $H=L+E_1+E_1$ the pull-back of a general line.
The cone $\Eff(Z)$ is the closed convex cone generated by $E_1, E_2$ and $L$, which are the only $(-1)$-curves on $Z$, while the cone $\Nef(Z)$ is the closed convex cone generated by $H, H - E_1$ and $H-E_2$.
The anti-canonical divisor $-K_Z = 3H - E_1 -E_2 = 3L +2E_1 +2E_2$ is ample.
In the figure we represent an affine section of the cone, and all divisors must be understood up to rescaling by an adequate homothety: for instance this is really $-\frac{1}{7}K_Z$ that is in the same affine section as $E_1, E_2$ and $L$, but for simplicity we write $-K_Z$.
Since $-K_Z$ is ample, one can choose the $A_i$ in Set-Up \ref{setup:coneC1} such that $-K_Z$ is contained in the cone generated by the $A_i$, and then $\Cl=\Eff(Z)$.

The faces $\Fl^0_i=\widebar{\Al_i}$, $i=0,\ldots,4$ are the faces of maximal dimension, the faces $\Fl_{ji}$ (written $\Fl^r_{ji}$ where $r$ is the codimension as usual) are as above $\Fl_{ji}=\widebar{\Al_j} \cap \widebar{\Al_i}$. Every face of $\Cl=\Eff(Z)$ is inner.
We can notice that the ample chamber $\Al_0$ is the only open one and that $\Al_7$ is the only closed one.
Moreover, as a hint that the behaviour of non maximal Mori chambers can be quite erratic, observe that  $\Al_7=\widebar{\Al_7}$ is not connected, and that neither $\widebar{\Al_5}$ nor $\widebar{\Al_6}$  is a single face.

This example will be continued in Example \ref{ex:facesAndFib} below.
\end{example}

\begin{figure}[ht]
\begin{align*}
\begin{tikzcd}[ampersand replacement=\&,scale=.5]
\&|[alias=Z]| Z_0=Z \ar[to=F12,"\phi_2"] \ar[to=F11,"\phi_1",swap] \ar[to=F0,"\phi_3"] \\
|[alias=F11]| Z_1=\F_{1}
\& |[alias=F0]| Z_3=\F_0 \ar[to=P12,"\phi_{36}",pos=.2,swap] \ar[to=P11,"\phi_{35}",pos=.2] 
\&|[alias=F12]| Z_2=\F_{1} \ar[to=P12,"\phi_{26}"] \\ 
|[alias=P11]| Z_5=\p^1 \ar[to=pt,"\phi_{57}",swap] \ar[from=F11, crossing over,"\phi_{15}",swap] 
\& |[alias=P2]| Z_4=\p^2 \ar[from=F11,crossing over,"\phi_{14}",pos=.2] \ar[from=F12,crossing over,"\phi_{24}",pos=.2,swap] \ar[to=pt,"\phi_{47}"] 
\& |[alias=P12]| Z_6=\p^1 \ar[to=pt,crossing over,"\phi_{67}"] \\
\& |[alias=pt]| Z_7=\pt
\end{tikzcd}
&&
\begin{array}{rcl}
\Al_0 &=& \interior{\Fl}^0_0. \\
\Al_1 &=& \interior{\Fl}^0_1 \cup \interior{\Fl}^1_{01}. \\
\Al_2&=& \interior{\Fl}^0_2 \cup \interior{\Fl}^1_{02}. \\
\Al_3&=& \interior{\Fl}^0_3 \cup \interior{\Fl}^1_{03}.\\
\Al_4&=& \interior{\Fl}^0_4 \cup \interior{\Fl}^1_{14} \cup \interior{\Fl}^1_{24} \cup \interior{\Fl}^2_{04}\\ 
\Al_5&=& \interior{\Fl}^1_{15} \cup \interior{\Fl}^1_{35} \cup \interior{\Fl}^2_{05}. \\
\Al_6&=& \interior{\Fl}^1_{26} \cup \interior{\Fl}^1_{36} \cup \interior{\Fl}^2_{06}. \\
\Al_7&=& \interior{\Fl}^1_{47} \cup \interior{\Fl}^2_{17}  \cup \interior{\Fl}^2_{27}\cup \interior{\Fl}^2_{37}.\\ 
\end{array}
\end{align*}

\begin{align*}
\begin{tikzpicture}[scale=4.6,font=\footnotesize,baseline=(-K)]
\coordinate (E1) at (0:0) {};
\coordinate (E2) at (0:1) {};
\coordinate (P) at (60:1) {};
\coordinate (H) at (barycentric cs:E1=1,E2=1,P=1) {};
\coordinate (P1) at (barycentric cs:E1=0,E2=1,P=1) {};
\coordinate (P2) at (barycentric cs:E1=1,E2=0,P=1) {};
\coordinate (-K) at (barycentric cs:E1=2,E2=2,P=3) {};
\draw (E1) to (P2)
      (P2) to (P)
      (P) to (P1)
      (P1) to (E2)
      (E2) to (E1);
\draw (E1) to (H)
      (E2) to (H)
      (H) to (P2)
      (P1) to (H)
      (P1) to (P2);
\node at (H) {$\bullet$};
\node at (P1) {$\bullet$};
\node at (P2) {$\bullet$};
\node at (E1) {$\bullet$};
\node at (E2) {$\bullet$};
\node at (P) {$\bullet$};
\node at (-K) {$\bullet$};
\node[right,xshift=-8,yshift=3] at (-K) {$-K_Z$};
\node[left] at (E1) {$E_2$};
\node[left,yshift=4] at (P2) {$L+E_2$};
\node[left,yshift=-4,xshift=-2] at (P2) {$=H-E_1$};
\node[above] at (P) {$L$};
\node[right,yshift=4] at (P1) {$L+E_1$};
\node[right,yshift=-4,xshift=2] at (P1) {$=H-E_2$};
\node[right] at (E2) {$E_1$};
\node[below] at (H) {$H$};
\end{tikzpicture}
&&
\begin{tikzpicture}[scale=4.6,font=\footnotesize,baseline=(-K)]
\coordinate (E1) at (0:0) {};
\coordinate (E2) at (0:1) {};
\coordinate (P) at (60:1) {};
\coordinate (H) at (barycentric cs:E1=1,E2=1,P=1) {};
\coordinate (P1) at (barycentric cs:E1=0,E2=1,P=1) {};
\coordinate (P2) at (barycentric cs:E1=1,E2=0,P=1) {};
 \coordinate (-K) at (barycentric cs:E1=2,E2=2,P=3) {};
\draw (E1) to ["\scriptsize $\Fl^1_{15}$",xshift=1] (P2)
      (P2) to ["\scriptsize $\Fl^1_{35}$",xshift=1] (P)
      (P) to ["\scriptsize $\Fl^1_{36}$",xshift=-1] (P1)
      (P1) to ["\scriptsize $\Fl^1_{26}$",xshift=-1] (E2)
      (E2) to ["\scriptsize $\Fl^1_{47}$",yshift =0.3mm] (E1);
\draw (E1) to [swap,"\scriptsize $\Fl^1_{14}$",yshift =1mm,xshift=-0.5](H)
      (E2) to ["\scriptsize $\Fl^1_{24}$",yshift =1mm,xshift=0.5](H)
      (H) to ["\scriptsize $\Fl^1_{01}$",yshift =1mm,xshift=0.5] (P2)
      (P1) to ["\scriptsize $\Fl^1_{02}$",yshift =1mm,xshift=-0.5] (H)
      (P1) to ["\scriptsize $\Fl^1_{03}$",swap,yshift=-.5] (P2);
\node at (H) {$\bullet$};
\node at (P1) {$\bullet$};
\node at (P2) {$\bullet$};
\node at (E1) {$\bullet$};
\node at (E2) {$\bullet$};
\node at (P) {$\bullet$};
\node[left] at (E1) {\scriptsize $\Fl^2_{17}$};
\node[left] at (P2) {\scriptsize $\Fl^2_{05}$};
\node[above] at (P) {\scriptsize $\Fl^2_{37}$};
\node[right] at (P1) {\scriptsize $\Fl^2_{06}$};
\node[right] at (E2) {\scriptsize $\Fl^2_{27}$};
\node[below,yshift=1] at (H) {\scriptsize $\Fl^2_{04}$};
\node[below,yshift=-18] at (H) {\scriptsize $\Fl^0_4\!=\!\widebar{\Al_4}$};
\node[below left,xshift=-10,yshift=3] at (H) {\scriptsize  $\Fl^0_1\!=\!\widebar{\Al_1}$};
\node[above,yshift=20] at (-K) {\scriptsize $\Fl^0_3\!=\!\widebar{\Al_3}$};
\node[below right,xshift=10,yshift=3] at (H) {\scriptsize $\Fl^0_2\!=\!\widebar{\Al_2}$};
\node[above,yshift=5] at (H) {\scriptsize $\Fl^0_0\!=\!\widebar{\Al_0}$};
\end{tikzpicture}
\end{align*}
\caption{Ample models and faces in Example~\ref{ex:P2blownup}.} \label{fig:faces}
\end{figure}

As a warm-up before the next section, we let the reader check that Proposition~\ref{pro:coneC} implies the following facts about codimension $1$ faces of $\Cl$.

\begin{remark} \label{rem:codim1}
Let $\Fl^1$ be an inner codimension $1$ face of the cone $\Cl \subseteq N^1(Z)$ from Set-Up~$\ref{setup:coneC}$, and $\Al_i$ the Mori chamber containing $\interior{\Fl}^1$ given by Lemma~\ref{lem:faces_factory}\ref{faces:inner_is_inter}.
Then $\Fl^1$ is contained in the closure of exactly one or two chambers of maximal dimension, depending whether $\Fl^1$ is in the boundary of $\Cl$ or not. 
\begin{enumerate}
\item Assume first that $\Fl^1 \subset \widebar{\Al_j}$ for a unique chamber $\Al_j$ of maximal dimension, so $\Fl^1$ is in the boundary of $\Cl$.
Moreover since $\Fl^1$ is inner we have $\Fl^1 \subseteq \partial^+ \Cl$, so $\dim Z_i < \dim Z_j$.
The associated map $\phi_{ji} \colon Z_j \to Z_i$ satisfies $\rho(Z_j/Z_i) = 1$. 
Moreover $-K_{Z_j}$ is relatively ample, so that $Z_j/Z_i$ is a terminal Mori fibre space (\cite[Lemma 3.2]{Kaloghiros}, see also Proposition~\ref{pro:boundary} below for a generalisation).

\item Now consider the case where $\Fl^1 = \widebar{\Al_j} \cap \widebar{\Al_k}$ for some distinct chambers $\Al_j, \Al_k$ of maximal dimension. 
We distinguish two subcases.
\begin{enumerate}[wide]
\item If $\Al_i$ is of maximal dimension, up to renumbering we can assume $\Al_i = \Al_k$, so that $\widebar{\Al_j} \cap \Al_i \supseteq \interior{\Fl}^1$.
In this situation both $Z_j$ and $Z_i$ are $\Q$-factorial and terminal, so the morphism $\phi_{ji} \colon Z_j \to Z_i$ with relative Picard rank $1$ given by Proposition~\ref{pro:coneC} is a divisorial contraction.

\item Finally if $\Al_i$ is not of maximal dimension, both birational morphisms $\phi_{ji}$ and $\phi_{ki}$ given by Proposition~\ref{pro:coneC} have relative Picard $1$ and target variety $Z_i$ which is not $\Q$-factorial, so $\phi_{ji}$ and $\phi_{ki}$ are small contractions. 
By uniqueness of the log-flip, the induced birational map $Z_j \ps Z_k$ must be the associated log-flip.
\end{enumerate}
\end{enumerate}
\end{remark}

\begin{remark} \label{rem:MMPscaling}
Let $\Delta \in \Cl$ be an ample divisor.
Then the successive chambers of maximal dimension that are cut by the segment $[\Delta, K_Z]$ can be interpreted as successive steps in a $K_Z$-MMP from $Z$.
In \cite[Remark 3.10.10]{BCHM} this is called a $K_Z$-MMP \emph{with scaling of $\Delta$}.
Moreover by perturbing $\Delta$ we can assume that the segment is transverse to the polyhedral decomposition.
Then as mentioned in Remark \ref{rem:codim1}, each intermediate face of codimension 1 that the segment meets corresponds either to a flip or to a divisorial contraction, and the last codimension 1 face in the boundary of the pseudo-effective cone corresponds to a Mori fibre space structure on the output of the MMP. 
\end{remark}

\subsection{Generation and relations in the Sarkisov programme}
\label{sec:kaloghiros}

The goal of this section is to prove Theorem~\ref{thm:sarkisov}, which will allow us to define the group homomorphisms of the main theorems. 
The main technical intermediate step is Proposition \ref{pro:boundary}, which explains the relation between our notion of rank~$r$ fibration and the combinatorics of the non-big boundary of the cone $\Cl$ as given in \cite{Kaloghiros}.

The following lemma can be extracted from \cite[Lemma 4.1]{HMcK} and \cite[Proposition 3.1\parent{ii}]{Kaloghiros}.
 
\begin{proposition}\label{pro:ResolutionZ}
Let $t\ge 2$ be an integer. For $i=1,\ldots,t$, let  $\eta_i\colon X_i\to B_i$ be a terminal Mori fibre space and let $\theta_i\colon X_i\rat X_{i+1}$ be a birational map $($here $\theta_t$ goes from $X_t$ to $X_{t+1}:=X_1)$. We assume moreover that $\theta_t\circ \cdots\circ  \theta_1=\id_{X_1}$. 
 
There exists a smooth variety $Z$, together with birational morphisms $\pi_i\colon Z\to X_i$, $i=1,\ldots,t$, and ample $\Q$-divisors $A_1,\ldots,A_{m}$ on $Z$ such that the following hold:
\begin{enumerate}
\item\label{CommonResZgen}
The divisors $A_1,\ldots,A_m$ generate the $\R$-vector space $N^1(Z)$.
\item\label{CommonResZisthere}
For $i=1,\ldots,t$, the birational morphism $\pi_i$ and the morphism $\eta_i\circ\pi_i$ are ample models of an element of  
\[
\Cl= \Bigl\{a_0  K_Z + \sum_{i=1}^m a_i A_i \Bigm| a_0,\ldots,a_m\ge 0\Bigr\}\cap \widebar{\Eff}(Z).
\] 
 \item\label{CommonResZcommut}
 For $i=1,\ldots,t$ we have $\theta_i\circ \pi_i=\pi_{i+1}$ $($with $\pi_{t+1}:=\pi_1)$. We then have a commutative diagram
 as in Figure~$\ref{fig:CommonZ}$.
 \end{enumerate} 
\end{proposition}

\begin{figure}[ht]
\[
\begin{tikzpicture}[scale=1.1,font=\small]
\node (B) at (0:0cm) {$Z$};
\foreach \x in {1,...,5}{
\node (X\x) at (\x*45+45:1.8*\Rc) {$X_{\x}$};
\draw[->] (B) to [auto,"$\pi_\x$",outer sep=-2pt,pos=.6] (X\x);
}
\node (Xt) at (360+45:1.8*\Rc) {$X_{t}$};
\draw[->] (B) to [auto,"$\pi_t$",outer sep=-2pt,pos=.6] (Xt);
\node (Xtm) at (315+45:1.8*\Rc) {$X_{t-1}$};
\draw[->] (B) to [auto,"$\pi_{t-1}$",outer sep=-2pt,pos=.6] (Xtm);
\draw[dashed,->] (X1) to [bend right=10,pos=.2,swap,"$\theta_1$"] (X2);
\draw[dashed,->] (X2) to [bend right=10,pos=.5,swap,"$\theta_2$"] (X3);
\draw[dashed,->] (X3) to [bend right=10,pos=.4,swap,"$\theta_3$"] (X4);
\draw[dashed,->] (X4) to [bend right=10,pos=.8,swap,"$\theta_4$"] (X5);
\node at (-45:1.6*\Rc) {\reflectbox{$\ddots$}};
\draw[dashed,->] (Xtm) to [bend right=10,auto,swap,"$\theta_{t-1}$",pos=.3] (Xt);
\draw[dashed,->] (Xt) to [bend right=10,auto,swap,"$\theta_t$"] (X1);
\end{tikzpicture}
\]
\caption{The commutative diagram in Proposition~\ref{pro:ResolutionZ}}
\label{fig:CommonZ}
\end{figure}

In the following discussion (and until Corollary \ref{cor:face_codim}) we work with the setting given by Proposition~\ref{pro:ResolutionZ}, that is, the commutative diagram of Figure~\ref{fig:CommonZ} and an associated choice of cone $\Cl \subset N^1(Z)$. 
Also recall that $\partial^+ \Cl \subset \Cl$ is the subset of non-big divisors.

\begin{lemma}
\label{lem:disc_or_sphere}
$\partial^+\Cl$ is the cone over a polyhedral complex homeomorphic to a disc or a sphere of dimension $\rho(Z)-2$. 
\end{lemma}

\begin{proof}
Consider the auxiliary cone $\Cl'$ of classes of the form
\[
\sum a_i A_i \text{ where } a_i \ge 0 \text{ for all } i.
\]
In other words, $\Cl'$ is the cone over the convex hull of the $A_i$, and in particular $\Cl'$ is a closed subcone of the ample cone of $Z$.
Let $\partial^+ \Cl'$ be the points in the boundary of $\Cl'$ that are visible from the point $K_Z$. 
Formally:
\[
\partial^+ \Cl' = \left\lbrace D \in \Cl' \mid [D, K_Z] \cap \Cl' = \{D\} \right\rbrace.
\]
By an elementary convexity argument (using the fact that a closed convex set with non empty interior is homeomorphic to a ball),
this cone $\partial^+\Cl'$ is homeomorphic to the cone over a sphere or a disc of dimension $\rho(Z)-2$, the first case occurring precisely if $-K_Z$ is in the interior of $\Cl'$.
Then we have a continuous map 
\[
\begin{tikzcd}[map]
\pi \colon \partial^+\Cl' &\to& \partial^+\Cl \\
D &\mapsto& \pi(D)
\end{tikzcd}
\]
that sends $D$ to the intersection of the segment $[D,K_Z]$ with $\partial^+\Cl$. 
The intersection exists because $K_Z \not\in \widebar{\Eff}(Z)$, while $D \in \Cl$, and the intersection is unique by convexity of $\Cl$. 
The injectivity of $\pi$ follows directly from the definition of $\partial^+\Cl'$, and $\pi$ is also surjective, because by definition the cone $\Cl$ in contained in the cone over the convex hull of $K_Z$ and the $A_i$, which is the same as the cone over the convex hull of $K_Z$ and $\Cl'$.
In conclusion $\pi$ is a homeomorphism, as expected.
\end{proof}

Recall that the codimension of a face is taken relatively to the ambient space $N^1(Z)$, so in particular if $\Fl^k \subseteq \partial^+\Cl$ we have $k \ge 1$. 

By Remark \ref{rem:codim1}, a face $\Fl^1$ of codimension 1 in $\partial^+\Cl$ corresponds to a Mori fibre space, or equivalently a rank 1 fibration (Lemma \ref{lem:rank1=Mfs}).
More generally, we now prove that inner codimension $r$ faces in $\partial^+\Cl$ correspond to rank~$r$ fibrations. 

\begin{proposition} \label{pro:boundary} 
Let $\Fl^r \subseteq \partial^+ \Cl$ be an inner codimension $r$ face.
By Lemma~$\ref{lem:faces_factory}\ref{faces:inner_is_inter}$, we can write $\Fl^r = \widebar{\Al_j}\cap\widebar{\Al_i}$ with $\Al_j$ a chamber of maximal dimension and $\Al_i \subseteq \partial^+ \Cl$ the Mori chamber containing the interior of $\Fl^r$.
Then 
\begin{enumerate}
\item \label{boundary:rank} 
The associated morphism $\phi_{ji}\colon Z_j \to Z_i$ is a rank~$r$ fibration.
\item \label{boundary:through}
If $\Fl^s \subseteq \partial^+ \Cl$ is an inner codimension $s$ face and $\Fl^r\subseteq\Fl^s$, then the rank $r$ fibration associated to $\Fl^r$ from \ref{boundary:rank} factorises through the rank $s$ fibration associated to $\Fl^s$.
\end{enumerate}
\end{proposition}

\begin{proof}
\ref{boundary:rank} We check the assertions of Definition~\ref{def:rankFibration}:

\ref{fib:eta}.
By Lemma~\ref{lem:faces_factory}\ref{faces:EinFji}, $\phi_{ji}\colon Z_j \to Z_i$ is a morphism with relative Picard rank equal to $r$, and $\dim Z_i < \dim Z_j$ because $\Al_i\subseteq  \partial^+ \Cl$.

\ref{fib:singB}. This is Proposition~\ref{pro:coneC}\ref{coneC:singularities}. 

\ref{fib:big}.
To show that $-K_{Z_j}$ is $\phi_{ji}$-big, we take $D\in \widebar{\Al_j}\cap\Al_i$. 
By Proposition~\ref{pro:coneC}\ref{coneC:phi_iBirational}, we have $D=K_Z+\Delta$ for some ample divisor $\Delta$, and $\phi_{j*}D \in \Nef(Z_j)$ is $\phi_{ji}$-trivial. 
By Lemma~\ref{lem:semiampleample}\ref{birsemiampleample} $\phi_j$ is a semiample model of any element of $\Al_j$.
So $\phi_j$ is a birational contraction and $\phi_{j*} K_Z=K_{Z_j}$, which we rewrite as $-K_{Z_j}=\phi_{j*}\Delta-\phi_{j*}D$. 
Since $\Delta$ is ample and $\phi_j$ is birational, the divisor $\phi_{j*} \Delta$ is big , which means we can write it as a sum of an ample and an effective divisor.
So $-K_{Z_j}$ is the sum of a $\phi_{ji}$-ample and an effective divisor and hence is $\phi_{ji}$-big by Lemma~\ref{lem:piBig}. 

\ref{fib:dream}.
We prove that $Z_j/Z_i$ is a Mori dream space:

\ref{dream:Q_factorial} and \ref{dream:rat_connected}. 
By Proposition~\ref{pro:coneC}, $Z_j$ is $\Q$-factorial terminal, $Z_i$ has rational singularities and $\dim Z_j > \dim Z_i$.
A general fibre of $\phi_{ji}$ has rational singularities by Lemma \ref{lem:fibresAreTerminal}. 
By Remark \ref{rem:klt} we can assume that $(Z,\Delta)$, and also $(Z_j,\phi_{j*}\Delta)$, are klt pairs.
By Proposition~\ref{pro:coneC}\ref{coneC:phi_ji} the divisor $K_{Z_j}+\phi_{j*} \Delta=\phi_{j*} D$ is $\phi_{ji}$-trivial. 
We have just seen that $-K_{Z_j}$ is $\phi_{ji}$-big. 
Then it follows from Lemma \ref{lem:fibresRatConnected}\ref{connected:-KBigNef} that a general fibre of $\phi_{ji}$ is rationally connected.

\ref{dream:nef}.
We show that the nef cone $\Nef(Z_j/Z_i)$ is generated by finitely many semiample divisors. 

We take $D_j\in \Nef(Z_j)$ and set  $D = \phi_j^* D_j \in N^1(Z)$. 
We choose $D' \in \interior{\Fl}^r\subseteq\Al_i\cap \widebar{\Al_j}$. 
By Remark~\ref{rem:InnerFace}, for  $t \gg 0$ we have $D + t D' \in \Cl$.
By Lemma~\ref{lem:direct_properties}\ref{comp:model}\&\ref{comp:combination}, we have $D + tD' \in\widebar{\Al_j}$. 
Since $\phi_{j*}D'$ is $\phi_{ji}$-trivial by Proposition~\ref{pro:coneC}\ref{coneC:phi_ji}, we get that 
$\phi_{j*} (D+ t D') = D_j + t\phi_{j*} D'$ is equivalent to $D_j$ in  $\Nef(Z_j/Z_i)$. 
Hence, any class in $\Nef(Z_j/Z_i)$ can be represented by a divisor in $\phi_{j*} \widebar{\Al_j}$. 
We conclude that $\Nef(Z_j/Z_i)$ is generated by finitely many divisors of the form $\phi_{j*}(K_Z+\Delta)$, where $K_Z + \Delta$ runs over the vertices of a polytope generating the cone $\widebar{\Al_j}$, and the $\phi_{j*}(K_Z+\Delta)$ are semiample by Proposition~\ref{pro:coneC}\ref{coneC:phi_iBirational}.

\ref{dream:mov}. 
Let $D_j \in \IntMov(Z_j)$, in particular $D_j$ is big.
Set $D = \phi_j^* D_j$ and pick $D'\in\interior{\Fl}^r\subseteq\Al_i\cap \widebar{\Al_j}$.
By Remark~\ref{rem:InnerFace}, for  $t \gg 0$ we have $\hat D := D + t D' \in \Cl$.

Replacing $D$ by an arbitrary close class in $\Cl$ we can assume that $\hat D = D + tD' \in \Al_k$ where $\Al_k$ is of maximal dimension. We also replace $D_j$ by $\phi_{j*}D$, which is a small perturbation of the initial class hence still in $\IntMov(Z_j)$. We keep the same notation for simplicity. (Observe that after perturbation we lose the property  $D = \phi_j^* D_j$, but we will not need it).
By finiteness of the chamber decomposition, $\Al_k$ does not depend on the choice of the large real $t$, which also implies $D'\in \widebar{\Al_k}\cap\Al_i$.
So we have $\Fl_{ji} \subseteq \Fl_{ki}$, hence a similar inclusion for the vector subspaces spanned by these faces.
By Lemma~\ref{lem:faces_factory}\ref{faces:EinFji} this implies that all divisors contracted by $\phi_j$ are also contracted by $\phi_k$, hence $f_k:= \phi_k \circ \phi_j^{-1}\colon Z_j \rat Z_k$ is a birational contraction.

As above $D_j$ and $\hat D_j := D_j + t\phi_{j*} D'$ represent the same class in $N^1(Z_j/Z_i)$.
Moreover by Lemma~\ref{lem:semiampleample}\ref{SemiampleFactorisation} we have $\phi_{j*} D' = \phi_{ji}^* D_i$, and the pull-back of an ample divisor being movable we have $\phi_{j*} D' \in \widebar{\Mov}(Z_j)$.
So we have $\hat D_j \in \IntMov(Z_j)$, and $\phi_{j*} \hat D = \hat D_j$ with $\hat D \in \Al_k$.

By Lemma \ref{lem:direct_properties}\ref{comp:composition}, the birational contraction $f_k\colon Z_j \rat Z_k$ is the ample model of $\hat D_j$.
Since $\hat D_j \in \IntMov(Z_j)$, its ample model $f_k$ is a pseudo-isomorphism.
Finally $\hat D_j \in f_k^*(\Ample(Z_k/Z_i))$ where $Z_k$ is $\Q$-factorial, and by taking closures we obtain 
\[\widebar{\Mov}(Z_j/Z_i) \subseteq \bigcup f_l^*(\Nef(Z_l/Z_i))\]
for some finite collection of pseudo-isomorphisms $f_l\colon Z_j \ps Z_l$ over $Z_i$ to $\Q$-factorial varieties.

For the other inclusion, we note that for any pseudo-isomorphism $f_l\colon Z_j\ps Z_l$ over $Z_i$, we have $f_l^*\Ample(Z_l/Z_i)\subset \Mov(Z_l/Z_i)$ and the claim follows by taking closures.

\ref{fib:singX}.
Let $D_j\in N^1(Z_j)$ be a divisor.
We now show that the output of any $D_j$-MMP from $Z_j$ over $Z_i$ can be obtained by running a $K_Z$-MMP from $Z$.  
Let $D'\in\Fl^r\subseteq\widebar{\Al_j}$. 
Then by Proposition~\ref{pro:coneC}\ref{coneC:phi_iBirational}, $\phi_j$ is a semiample model of $D'$, $\phi_i$ is its ample model, and by Lemma~\ref{lem:semiampleample}\ref{SemiampleFactorisation} $\phi_{j*} D'=\phi_{ji}^* H_i$ for some ample divisor $H_i$ on $Z_i$.
To run a $D_j$-MMP from $Z_j$ over $Z_i$, we pick $H_j\in\Ample(Z_j)$ and consider all pseudo-effective convex combinations $D_t:=\varepsilon(tD_j+(1-t)H_j)+\phi_{ij}^* H_i$ for some $1\gg\varepsilon>0$. 
The set of the $\phi_j^* D_t$ is a segment in a small neighborhood of $D'$ inside $\Cl$. 
Therefore, any intermediate variety in this $D_j$-MMP over $Z_i$ can be obtained by running a $K_Z$-MMP from $Z$. 
In particular the output of this MMP has the form $\Proj H^0(Z_j, D_{t_0}) = \Proj H^0(Z, \phi_j^* D_{t_0})$ for some $t_0 \in (0,1)$,
and by Proposition \ref{pro:coneC}\ref{coneC:MaxDim}, this is a $\Q$-factorial and terminal variety, as expected. 

\ref{boundary:through} (Analogous to \cite[Proposition 3.10(2)]{LZ17}): 
Let $\Al_i,\Al_k\subseteq\partial^+\Cl$ be the chambers containing the interior of $\Fl^r$, $\Fl^s$ respectively. By Lemma~\ref{lem:faces_factory}\ref{faces:inner_is_inter} there exist maximal chambers $\Al_j$ and $\Al_l$ such that $\Fl^r=\widebar{\Al_j}\cap\widebar{\Al_i}$ and $\Fl^s=\widebar{\Al_l}\cap\widebar{\Al_k}$.  
Since moreover $\Fl^r \subseteq \Fl^s$ implies that $\widebar{\Al_l} \cap \Al_i \neq \emptyset$, by Proposition~\ref{pro:coneC}\ref{coneC:phi_ji} we have a commutative diagram induced by the maps from $Z$:
\[
\begin{tikzcd}[link]
Z_j \ar[r,dashed] \ar[dd] & Z_l \ar[dd] \ar[ddl] \\ \\
Z_i & Z_k \ar[l]
\end{tikzcd}
\]
We want to prove that the birational map $Z_j \rat Z_l$ is a birational contraction.

Let $D\in\interior{\Fl}^r \subseteq \Al_i$. 
There exists an ample class $\Delta\in\Cl$ and $t_1 > 0$ such that $D = (1-t_1)\Delta+t_1 K_Z$. 
For $t_1 > t_0 > 0$ sufficiently close to $t_0$, any chamber of maximal dimension $\Al_{j_0}$ such that $ (1-t_0)\Delta+t_0 K_Z \in \Al_{j_0}$ satisfies $\Fl^r \subset \widebar{\Al_{j_0}}$. 
Now there exists a small perturbation $\Delta'$ of $\Delta$ such that the segment $[\Delta',K_Z]$ meets successively a chamber $\Al_{j_0}$ and then the chamber $\Al_l$. Indeed,  $t_1>t_0$ and the ordering is preserved under a small perturbation.
Up to replacing $j$ by this $j_0$, by Remark \ref{rem:MMPscaling} this segment corresponds to a $K_Z$-MMP with scaling of $\Delta'$, and provides the expected birational contraction from $Z_j$ to $Z_l$. 
\end{proof}

\begin{figure}[ht]
\begin{align*}
\begin{tikzpicture}[scale=4.5,font=\footnotesize,baseline=(-K)]
\coordinate (E1) at (0:0) {};
\coordinate (E2) at (0:1) {};
\coordinate (P) at (60:1) {};
\coordinate (H) at (barycentric cs:E1=1,E2=1,P=1) {};
\coordinate (P1) at (barycentric cs:E1=0,E2=1,P=1) {};
\coordinate (P2) at (barycentric cs:E1=1,E2=0,P=1) {};
\coordinate (-K) at (barycentric cs:E1=2,E2=2,P=3) {};
\draw (E1) to ["\scriptsize $\Fl^1_{15}$",xshift=1] (P2)
      (P2) to ["\scriptsize $\Fl^1_{35}$",xshift=1] (P)
      (P) to ["\scriptsize $\Fl^1_{36}$",xshift=-1] (P1)
      (P1) to ["\scriptsize $\Fl^1_{26}$",xshift=-1] (E2)
      (E2) to ["\scriptsize $\Fl^1_{47}$",yshift =0.3mm] (E1);
\draw (E1) to [swap,"\scriptsize $\Fl^1_{14}$",yshift =1mm,xshift=-0.5](H)
      (E2) to ["\scriptsize $\Fl^1_{24}$",yshift =1mm,xshift=0.5](H)
      (H) to ["\scriptsize $\Fl^1_{01}$",yshift =1mm,xshift=0.5] (P2)
      (P1) to ["\scriptsize $\Fl^1_{02}$",yshift =1mm,xshift=-0.5] (H)
      (P1) to ["\scriptsize $\Fl^1_{03}$",swap,yshift=-.5] (P2);
\node at (H) {$\bullet$};
\node at (P1) {$\bullet$};
\node at (P2) {$\bullet$};
\node at (E1) {$\bullet$};
\node at (E2) {$\bullet$};
\node at (P) {$\bullet$};
\node[left] at (E1) {\scriptsize $\Fl^2_{17}$};
\node[left] at (P2) {\scriptsize $\Fl^2_{05}$};
\node[above] at (P) {\scriptsize $\Fl^2_{37}$};
\node[right] at (P1) {\scriptsize $\Fl^2_{06}$};
\node[right] at (E2) {\scriptsize $\Fl^2_{27}$};
\node[below,yshift=1] at (H) {\scriptsize $\Fl^2_{04}$};
\node[below,yshift=-18] at (H) {\scriptsize $\Fl^0_4\!=\!\widebar{\Al_4}$};
\node[below left,xshift=-10,yshift=3] at (H) {\scriptsize  $\Fl^0_1\!=\!\widebar{\Al_1}$};
\node[above,yshift=20] at (-K) {\scriptsize $\Fl^0_3\!=\!\widebar{\Al_3}$};
\node[below right,xshift=10,yshift=3] at (H) {\scriptsize $\Fl^0_2\!=\!\widebar{\Al_2}$};
\node[above,yshift=5] at (H) {\scriptsize $\Fl^0_0\!=\!\widebar{\Al_0}$};
\end{tikzpicture}
&&
\begin{tikzpicture}[scale=4.5,font=\footnotesize,baseline=(-K)]
\coordinate (E1) at (0:0) {};
\coordinate (E2) at (0:1) {};
\coordinate (P) at (60:1) {};
\coordinate (H) at (barycentric cs:E1=1,E2=1,P=1) {};
\coordinate (P1) at (barycentric cs:E1=0,E2=1,P=1) {};
\coordinate (P2) at (barycentric cs:E1=1,E2=0,P=1) {};
\coordinate (-K) at (barycentric cs:E1=2,E2=2,P=3) {};
\draw (E1) to ["\scriptsize $Z_1/Z_5\!=\!\p^1$",xshift=1] (P2)
      (P2) to ["\scriptsize $Z_3/Z_5\!=\!\p^1$",xshift=1] (P)
      (P) to ["\scriptsize $Z_3/Z_6\!=\!\p^1$",xshift=-1] (P1)
      (P1) to ["\scriptsize $Z_2/Z_6\!=\!\p^1$",xshift=-1] (E2)
      (E2) to ["\scriptsize $Z_4/Z_7\!=\!\pt$",yshift =0.3mm] (E1);
\draw (E1) to (H)
      (E2) to (H)
      (H) to (P2)
      (P1) to (H)
      (P1) to (P2);
\node at (H) {$\bullet$};
\node at (P1) {$\bullet$};
\node at (P2) {$\bullet$};
\node at (E1) {$\bullet$};
\node at (E2) {$\bullet$};
\node at (P) {$\bullet$};
\node[below] at (E1) {\scriptsize $Z_1/Z_7\!=\!\pt$};
\node[left] at (P2) {\scriptsize $Z_0/Z_5\!=\!\p^1$};
\node[above] at (P) {\scriptsize $Z_3/Z_7\!=\!\pt$};
\node[right] at (P1) {\scriptsize $Z_0/Z_6\!=\!\p^1$};
\node[below] at (E2) {\scriptsize $Z_2/Z_7\!=\!\pt$};
\node[below,yshift=-18] at (H) {\scriptsize $Z_4\!=\!\p^2$};
\node[below left,xshift=-10,yshift=3] at (H) {\scriptsize  $Z_1\!=\!\F_1$};
\node[above,yshift=20] at (-K) {\scriptsize $Z_3\!=\!\F_0$};
\node[below right,xshift=10,yshift=3] at (H) {\scriptsize $Z_2\!=\!\F_1$};
\node[above,yshift=5] at (H) {\scriptsize $Z_0\!=\!Z$};
\end{tikzpicture}
\end{align*}
\caption{rank~$r$ fibrations in Example~\ref{ex:P2blownup}.} \label{fig:facesAndFib}
\end{figure}

\begin{example}
\label{ex:facesAndFib}
On Figure~\ref{fig:facesAndFib} we label the boundary faces from Example~\ref{ex:P2blownup} with their corresponding rank~$r$ fibration, as given by Proposition~\ref{pro:boundary} ($r = 1$ or $2$ here).
We also indicate the images of ample models corresponding to chambers of maximal dimension.
\end{example}

Applying Proposition \ref{pro:boundary} to faces of codimension 2 or 3, we obtain the following corollary.
Observe that the first point is well known (see e.g. \cite[Theorem 3.7]{HMcK}), and the second one is a natural generalisation.
 
\begin{corollary} \label{cor:face_codim}~
\begin{enumerate}
\item \label{cor:face_codim_2}
If the intersection $\Fl_i^1\cap\Fl_j^1$ of non-big codimension $1$ faces has codimension~$2$, then there is a Sarkisov link between the corresponding Mori fibre spaces.

\item \label{cor:face_codim_3}
Let $\Fl^3$ be a face in $\partial^+\Cl$ of codimension $3$ 
and $T/B$ be the associated rank~$3$ fibration, as given in Proposition~$\ref{pro:boundary}$. 
Then the elementary relation associated to $T/B$ corresponds to the finite collection of codimension $1$ faces $\Fl_1^1,\dots,\Fl_s^1$ containing $\Fl^3$, and ordered such that $\Fl_j^1$ and $\Fl_{j+1}^1$ share a codimension $2$ face for all $j$ $($where indexes are taken modulo $s)$.
\end{enumerate}
\end{corollary}

\begin{proof}
\ref{cor:face_codim_2}
By Proposition~\ref{pro:boundary} there is a rank~$2$ fibration corresponding to the codimension $2$ face $\Fl^2:=\Fl_i^1\cap\Fl_j^1$ that factorises through the rank~$1$ fibrations associated to $\Fl_i^1$ and $\Fl_j^1$. 
This is exactly the definition of a Sarkisov link (Definition~\ref{def:sarkisovLink}).

\ref{cor:face_codim_3}
This is just a rephrasing of Proposition~\ref{pro:from T3}, using Proposition~\ref{pro:boundary} to associate a rank~$1$ or $2$ fibration dominated by $T/B$ to each codimension $1$ or $2$ face containing $\Fl^3$, and using \ref{cor:face_codim_2} to associate a Sarkisov link to each pair of codimension 1 faces sharing a common codimension 2 face.
\end{proof}

Let $X/B$ be a terminal Mori fibre space.
We denote by $\BirMori(X)$ the groupoid of birational maps between terminal Mori fibre spaces birational to $X$.
The group of birational selfmaps $\Bir(X)$ is a subgroupoid of $\BirMori(X)$.
The motivation for introducing the notion of elementary relation is the following result. 
The first part is a reformulation of \cite[Theorem 1.1]{HMcK}. 
The second part is strongly inspired by \cite[Theorem 1.3]{Kaloghiros}, observe however that our notion of elementary relation is more restrictive, and so Theorem \ref{thm:sarkisov}\ref{sarkisov2} does not follow from \cite{Kaloghiros}.

In the statement we use the formalism of presentations by generators and relations for groupoids. 
This is very similar to the more familiar setting of groups: we have natural notions of a free groupoid, and of a normal subgroupoid generated by a set of elements.
We refer to \cite[\S8.2 and 8.3]{Brown} for details.

\begin{theorem}\label{thm:sarkisov} 
Let $X/B$ be a terminal Mori fibre space. 
\begin{enumerate}
\item\label{sarkisov1} The groupoid $\BirMori(X)$ is generated by Sarkisov links and automorphisms.
\item\label{sarkisov2} Any relation between Sarkisov links in $\BirMori(X)$ is generated by elementary relations.
\end{enumerate}
\end{theorem}

\begin{proof}
\ref{sarkisov1} is the main result of \cite{HMcK}. 
The idea of the proof is to take $Z$ a resolution of a given birational map $\phi\colon X_1/B_1 \rat X_2/B_2$, and to consider the cone $\Cl$ with a choice of ample divisors as given by Proposition~\ref{pro:ResolutionZ} (applied with $t = 2$, $\theta_1 = \phi$, $\theta_2 = \phi^{-1}$). 
Then one takes a general 2-dimensional affine slice of $\Cl$ that passes through the codimension 1 faces associated to $X_1/B_1$ and $X_2/B_2$. 
The intersection of this slice with $\partial^+ \Cl$ is a polygonal path corresponding to successive pairwise neighbour codimension 1 faces, and by Corollary~\ref{cor:face_codim}\ref{cor:face_codim_2} this gives a factorisation of $\phi$ into Sarkisov links.  

\ref{sarkisov2}.
The proof is essentially the same as in \cite[Proposition 3.15]{LZ17}, we repeat the argument for the convenience of the reader. 

Let 
\[
X_0/B_0 \stackrel{\chi_1}\longrat X_1/B_1 \stackrel{\chi_2}\longrat \dots \stackrel{\chi_t}\longrat X_t/B_t
\]
be a relation between $t$ Sarkisov links, meaning that $\chi_t \circ \cdots \circ \chi_1$ is the identity on $X_0 = X_t$.
We take a smooth resolution $Z$ dominating all the $X_i/B_i$, and consider the cone $\Cl \subset N^1(Z)$ constructed from a choice of ample divisors as in Proposition~\ref{pro:ResolutionZ}.
We may assume $\rho(Z) \ge 4$ (otherwise we simply blow-up some points on $Z$), so that by Lemma~\ref{lem:disc_or_sphere} the non-big boundary $\partial^+ \Cl$ is a cone over a polyhedral complex $\Sl$ homeomorphic to a disc or a sphere of dimension $\rho(Z) - 2 \ge 2$.
In particular, the section $\Sl$ is simply connected.
Now we construct a 2-dimensional simplicial complex $\Bl$ embedded in $\Sl$ as follows.
Vertices are the barycenters $p(\Fl^k)$ of codimension $k$ faces $\Fl^k$ for $k = 1, 2$ or $3$. 
We call $k$ the type of the vertex.
We put an edge between $p(\Fl^j)$ and $p(\Fl^k)$ if $\Fl^j$ is a proper face of $\Fl^k$, and a 2-simplex for each sequence $\Fl^3 \subset \Fl^2 \subset \Fl^1$.
The complex $\Bl$ is homeomorphic to the barycentric subdivision of the $2$-skeleton of the dual cell complex of $\Sl$.
It follows that $\Bl$ is simply connected (recall that the 2-skeleton of a simply connected complex is again simply connected, see e.g. \cite[Corollary 4.12]{Hatcher}).
Then we restrict to the subcomplex $\Il \subseteq \Bl$ corresponding to inner faces, which are the ones that intersect the relative interior of $\Sl$.
The simplicial complex $\Il$ is a deformation retract of the interior of $\Bl$, so $\Il$ again is simply connected.
By Proposition~\ref{pro:boundary} we can associate a rank~$r$ fibration to each vertex of type~$r$, and two vertices are connected by an edge if and only if the corresponding fibrations factorise through each other.
By Corollary~\ref{cor:face_codim}\ref{cor:face_codim_3}, around each vertex of type 3 there is a unique disc whose boundary loop encodes an elementary relation.
The $2$-dimensional components of the complex $\Il$ are unions of these discs.
The initial relation corresponds to a loop in $\Il$ that only passes through vertices of types 1 and 2.
We can realise the homotopy of this loop to the constant loop inside the simply connected complex $\Il$ by using these elementary relations, and this translates as a factorisation of the initial relation as a product of conjugates of elementary relations.
\end{proof}

The whole construction leading to the previous theorem can be made in a relative setting, that is, where all involved varieties admit a morphism to a fixed base variety $B$.
In fact the paper \cite{BCHM} on which relies \cite{HMcK} is written with this level of generality.
In the particular case where the base $B$ has dimension $n-1$, we obtain the following statement, slightly more precise than Theorem~\ref{thm:sarkisov}\ref{sarkisov1}.

\begin{figure}[ht]
\[
\begin{tikzcd}[column sep=1.2cm,row sep=0.16cm]
X=X_0\ar[rrrr,dashed,"\phi", bend left=13] \ar[dd,"\eta_X",swap] \ar[r,dashed,"\chi_1"] & X_1\ar[dd] \ar[r,dashed,"\chi_2"] & \cdots \ar[r,dashed,"\chi_{t-1}"]& X_{t-1} \ar[r,dashed,"\chi_t"]\ar[dd]  & X_t=Y\ar[dd,"\eta_Y"]  \\ \\
B_0=B\ar[drr,"\id",swap] & B_1\ar[dr]& \cdots& B_{t-1} \ar[dl]  & B_t=B\ar[dll,"\id"]  \\
&& B
\end{tikzcd}
\]
\caption{The diagram of Lemma \ref{lem:SarkisovOverBase}} \label{fig:SarkisovOverBase}
\end{figure}

\begin{lemma}\label{lem:SarkisovOverBase}
Let $\eta_X\colon X\to B$ and $\eta_Y\colon Y\to B$ be two conic bundles over the same base. 
Then any birational map $\phi\colon X\rat Y$ over $B$ decomposes into a sequence of Sarkisov links of conic bundles over $B$. More precisely, we have a commutative diagram as in Figure~$\ref{fig:SarkisovOverBase}$,
such that for each $i=1,\dots,t$, $B_i/B$ is a birational morphism, $X_i/B_i$ is a conic bundle and $\chi_i$ is a Sarkisov link. 
\end{lemma}

\subsection{Examples of elementary relations}
\label{sec:examples}

In this section we give examples of elementary relations, mostly in dimension $n = 3$.

\begin{example}
\label{ex:kaloghiros}
Let $X$ be a Fano variety with $\Q$-factorial terminal singularities and Picard rank~$3$.
Then $X/\pt$ is a rank 3 fibration (Example~\ref{ex:rank_r}\ref{ex:rank_r:2}), hence there is an associated elementary relation.
In the case where $X$ is smooth of dimension~$3$, these relations were studied systematically by Kaloghiros, using a classification result by Mori-Mukai: see \cite[Example 4.9 and Figures 3,4 \& 5]{Kaloghiros}.
With respect to the setting of \S\ref{sec:kaloghiros}, in these examples we have $Z = X$, $N^1(Z) \simeq \R^3$ and $\partial^+\Cl$ is the cone over a complex homeomorphic to a circle, which encodes the elementary relation. 
Observe that the simple 2-dimensional Example \ref{ex:facesAndFib} also belongs to this family of examples. 
\end{example}

\begin{example}
\label{ex:P3_2lines}
Let $L \cup L' \subset \p^3$ be two secant lines, and $P$ the plane containing them.
Let $X \to \p^3$ be the blow-up of $L$ with exceptional divisor $E$, let $\ell \subset E$ be the fibre intersecting the strict transform of $L'$, and let $T \to X$ be the blow-up of $L'$, with exceptional divisor $E'$.

From $T$ we can flop $\ell$ to get a 3-fold $T'$, which is obtained by the same two blow-ups in the reverse order: first the blow-up $X'\to \p^3$ of $L'\subset \p^3$ and then the blow-up $T'\to X'$ of (the strict transform of) $L$ on $X'$.

From $T$ or $T'$ one can contract the strict transform of $P$ onto a smooth point, obtaining two 3-folds $Y$ and $Y'$ also related by the flop of $\ell$.

The elementary relation associated to the rank~$3$ fibration $T/\pt$ (or equivalently to $T'/\pt$), is depicted on Figure~\ref{fig:relation_P3_2lines}.
There are five links in the relation, where $\chi_1$ has type \I, $\chi_2$ and $\chi_4$ have type \II, $\chi_3$ has type \IV, and $\chi_5$ has type \III.
\end{example}

\begin{figure}[ht]
\[
\begin{tikzpicture}[scale=1.5,font=\small]
\node (pt) at (90:0) {$\pt$};
\node (P1) at (-90:\Rc) {$\p^1$};
\node (P3) at (90:\Rc) {$\p^3$};
\node (T) at (180:\Ra) {$T$};
\node (T') at (0:\Ra) {$T'$};
\node (X) at (170:\Rb) {$X$};
\node (X') at (10:\Rb) {$X'$};
\node (Y) at (-135:\Rb) {$Y$};
\node (Y') at (-45:\Rb) {$Y'$};

\draw[->] (P1) to (pt);
\draw[->] (P3) to (pt);
\draw[->] (X) to (P1);
\draw[->] (X) to ["\tiny $E$",swap] (P3);
\draw[->] (Y) to (P1);
\draw[->] (T) to ["\tiny $E'$"] (X);
\draw[->] (T) to ["\tiny $P$",swap] (Y);
\draw[->] (X') to (P1);
\draw[->] (X') to ["\tiny $E'$"] (P3);
\draw[->] (Y') to (P1);
\draw[->] (T') to ["\tiny $E$",swap] (X');
\draw[->] (T') to ["\tiny $P$"] (Y');
\draw[-,dotted] (T) to[bend left = 80,"\tiny flop"] (T');

\draw[->,dashed] (P3) to ["$\chi_1$", bend right=15,swap] (X);
\draw[->,dashed] (X) to ["$\chi_2$", bend right=25] (Y);
\draw[->,dotted] (Y) to[bend right=20] node[above]{$\chi_3$} node[below]{\tiny flop} (Y');
\draw[->,dashed] (Y') to ["$\chi_4$", bend right=25] (X');
\draw[->,dashed] (X') to ["$\chi_5$", bend right=15,swap] (P3);
\end{tikzpicture}
\]
\caption{The elementary relation from Example~\ref{ex:P3_2lines}.}\label{fig:relation_P3_2lines}
\end{figure}

\begin{example} \label{ex:X3X2X1=1}
Consider the blow-up $\F_1 \to \p^2$ of a point, with exceptional curve $\Gamma \subset \F_1$.
In $\p^1 \times \F_1$, write $D = \p^1 \times \Gamma$, and $C = \{0\} \times \Gamma$.
Let $T$ be the blow-up of $C$, with exceptional divisor $E$.
Then $T/\p^2$ is a rank 3 fibration, and we now describe the associated elementary relation (see Figure~\ref{fig:X3X2X1=1}).
We let the reader verify the following assertions (since all varieties are toric, one can for instance use the associated fans). 

First the two-rays game $T/\F_1$ gives a link of type II 
\[\chi_1\colon \p^1 \times  \F_1 \rat \p^1 \boxtimes \F_1 ,\]
where $\p^1 \boxtimes \F_1$ denotes a smooth Mori fibre space over $\F_1$ that is a non-trivial but locally trivial $\p^1$-bundle.
The link $\chi_1$ involves the pair $D \cup E$ of divisors of type \II for $T/\F_1$.  

The divisor $D$ on $T$ can be contracted in two ways to a curve $\p^1$, that is, $T$ dominates a flop between $\p^1 \boxtimes \F_1$ and another variety $X$.
This variety $X$ admits a divisorial contraction to $\p^1 \times \p^2$, with exceptional divisor the strict transform of $E$, which here is a divisor of type \I for $X/\p^2$.
This corresponds to a link of type \III 
\[\chi_2\colon \p^1 \boxtimes \F_1 \rat \p^1 \times \p^2.\]

Finally the two-rays game $\p^1 \times \F_1/\p^2$, which factorises via $\F_1$ and $\p^1\times\p^2$, gives a link of type \I 
\[\chi_3\colon \p^1 \times \p^2 \rat  \p^1 \times \F_1.\]

In conclusion we get an elementary relation $\chi_3 \circ \chi_2 \circ \chi_1 = \id$.

In contrast with Lemma~\ref{lem:gonalityOfTypeII}, observe that $D$ and $E$ are divisors of type \II for $T/\F_1$, but divisors of type \I for $T/\p^2$.
\end{example}

\begin{figure}[ht]
\[
\begin{tikzpicture}[scale=1.5,font=\small]
\node (T) at (180:\Ra) {$T$};
\node (F1xP1) at (135:\Rb) {$\p^1 \times \F_1$};
\node (F1xP1bis) at (-135:\Rb) {$\p^1 \boxtimes \F_1$};
\node (X) at (-45:\Rb) {$X$};
\node (P2) at (0:0cm) {$\p^2$};
\node (F1) at (180:\Rc) {$\F_1$};
\node (P2xP1) at (0:\Rc) {$\p^1\times \p^2$};
\draw[->] (T) to [auto,"\tiny $E_{\textup{I}} / E_{\textup{II}}$"] (F1xP1);
\draw[->] (T) to [auto,swap,"\tiny $D_{\textup{II}}$"] (F1xP1bis);
\draw[->] (T) to [bend right = 70,"\tiny $D_{\textup{I}}$",auto,swap] (X);
\draw[->] (F1xP1)--  (F1);
\draw[->] (F1xP1) to["\tiny $D_{\textup{I}}$"]  (P2xP1);
\draw[->] (F1xP1bis)--  (F1);
\draw[->] (X) to [auto,swap,pos=.3,"\tiny $E_{\textup{I}}$"] (P2xP1);
\draw[->] (F1)--  (P2);
\draw[->] (P2xP1)--  (P2);
\draw[dashed,->] (F1xP1) to [bend right=38] node[auto,swap]{$\chi_1$} (F1xP1bis);
\draw[dashed,->] (F1xP1bis) to [bend right=42] node[auto,swap]{$\chi_2$} (P2xP1);
\draw[dashed,->] (P2xP1) to [bend right=42] node[auto,swap]{$\chi_3$} (F1xP1);
\draw[dotted,-] (F1xP1bis) to [bend right=20,"\tiny\text{flop}",swap] (X);
\end{tikzpicture}
\]
\caption{The elementary relation from Example~\ref{ex:X3X2X1=1}.
We indicate the type of contracted divisors in index.}\label{fig:X3X2X1=1}
\end{figure}

\begin{example}[Example~\ref{ex:type_I_and_II_div} over $B = \p^2$]
\label{ex:with_flip}
Consider $\p^1 \times \p^2$, and let $\Gamma \subset \p^2$ be a line, $D \simeq \p^1 \times \Gamma$ the pull-back of $\Gamma$ in $\p^1 \times \p^2$, $\Gamma' = \{t\} \times \Gamma \subset D$ a section and $p \in D \setminus \Gamma'$ a point.
Let $T \to \p^1 \times \p^2$ be the blow-up of $\Gamma'$ and $p$, with respective exceptional divisors $D'$ and $E$, and denote again $D$ the strict transform of $\p^1 \times \Gamma$ in $T$. 
Then the induced morphism $\eta\colon T \to \p^2$ is a rank 3 fibration that gives rise to the relation on Figure~\ref{fig:flip}.

The figure was computed using toric fans in $\Z^3$, starting from the standard fan of $\p^1 \times \p^2$ with primitive vectors $(1,0,0)$, $(0,1,0)$, $(-1,-1,0)$, $(0,0,1)$, $(0,0,-1)$, and with the following choices 
\begin{align*}
D: (1,0,0), && D':(1,0,1), && E: (1,1,-1).
\end{align*}
The varieties $T'$ and $W'$  both have one terminal singularity, all other varieties are smooth.
There are two distinct Francia flips from $T'$, which are $T' \ps T$ and $T' \ps T''$.
Observe also that the link $\chi_1$ is exactly Example~\ref{ex:simple_links}\ref{simple_link:2}.
\end{example}

\begin{figure}[t]
\[
\begin{tikzpicture}[scale=1.4,font=\small]
\node (P2) at (0:0) {$\p^2$};
\node (T) at (0:\Ra) {$T$};
\node (T') at (120:\Ra) {$T'$};
\node (T'') at (-120:\Ra) {$T''$};
\node (X) at (55:\Rb) {$X$};
\node (W') at (150:\Rb) {$W'$};
\node (W) at (-150:\Rb) {$W$};
\node (P1xF1) at (-90:\Rb) {$\p^1 \boxtimes \F_1$};
\node (Y) at (-55:\Rb) {$Y$};
\node (F1) at (-120:\Rc) {$\F_1$};
\node (P1xP2) at (0:\Rc) {$\p^1\times \p^2$};
\node (P1xP2bis) at (120:\Rc) {$\p^1\boxtimes \p^2$};
\draw[->] (F1) to (P2);
\draw[->] (P1xP2) to (P2);
\draw[->] (P1xP2bis) to (P2);
\draw[->] (P1xF1) to (F1);
\draw[->] (W) to (F1);
\draw[->] (T) to ["$E$",swap] (X);
\draw[->] (T') to ["$D$",swap] (W');
\draw[->] (T) to ["$D'$"] (Y);
\draw[->] (T'') to ["$D$"] (W);
\draw[->] (T'') to ["$D'$",swap] (P1xF1);
\draw[->] (X) to ["$D'$"] (P1xP2);
\draw[->] (X) to ["$D$",swap] (P1xP2bis);
\draw[->] (W') to ["$E$",swap] (P1xP2bis);
\draw[->] (Y) to ["$E$",swap] (P1xP2);
\draw[->,dotted] (T') to ["\tiny flip",bend left=50] (T);
\draw[->,dotted] (T') to ["\tiny flip",swap, bend right=50] (T'');
\draw[->,dotted] (W') to ["\tiny flip",swap] (W);
\draw[dotted] (Y) to ["\tiny flop", pos=1] (P1xF1);
\draw[->,dashed] (P1xP2) to ["$\chi_1$", bend right=20] (P1xP2bis);
\draw[->,dashed] (P1xP2bis) to ["$\chi_2$", bend right=20] (W);
\draw[->,dashed] (W) to ["$\chi_3$", bend right=20] (P1xF1);
\draw[->,dashed] (P1xF1) to ["$\chi_4$", bend right=20] (P1xP2);
\end{tikzpicture}
\]
\caption{Elementary relation from Example~\ref{ex:with_flip}.}
\label{fig:flip}
\end{figure}

\begin{example}
\label{ex:AZ}
The article \cite{AZ2017} contains a beautiful example of an elementary relation involving five Sarkisov links.
In Figure~\ref{fig:relation_AZ} we reproduce the diagram from \cite[\S5.2]{AZ2017}, and we refer to their paper for a detailed description of the varieties.
The Sarkisov links $\chi_1$ and $\chi_3$ have type \II, $\chi_2$ has type \I, $\chi_4$ has type \IV and $\chi_5$ type \III.
The relation is associated to the rank 3 fibration $Z_1'/\pt$, or equivalently to $Z_2'/\pt$.
In fact other equivalent choices of varieties of Picard rank 3 are omitted from the picture (dominating respectively $Y_1'$, $X_3'$, $X_1'$ and $X_1''$).
The morphisms from $Z, \bar Z$ and $\tilde Z$ to $\p^1$ are fibrations in cubic surfaces.
Observe that the top rows of the Sarkisov diagrams display non trivial pseudo-isomorphisms, involving flips and flops. 
Note that each pseudo-isomorphism labeled ``$n$ flops'' really corresponds to a single flop with $n$ components (which by definition are all numerically proportional), in accordance with Remark~\ref{rem:top_row}. 
\end{example}

\begin{figure}
\[
\begin{tikzpicture}[scale=1.5,font=\small]
\node (pt) at (0:0) {$\pt$};
\node (P1) at (0:\Rc) {$\p^1$};
\node (Y) at (120:\Rc) {$Y$};
\node (X) at (-120:\Rc) {$X$};
\node (barZ) at (-15:\Rb) {$\bar Z$};
\node (Z) at (15:\Rb) {$Z$};
\node (Y1') at (105:\Rb) {$Y_1'$};
\node (X3') at (135:\Rb) {$X_3'$};
\node (X1') at (-135:\Rb) {$X_1'$};
\node (X1'') at (-105:\Rb) {$X_1''$};
\node (X2') at (180:\Rb) {$X_2'$};
\node (tildeZ) at (-45:\Rb) {$\tilde Z$};
\node (X2'') at (-75:\Rb) {$X_2''$};
\node (Z1') at (-45:\Ra) {$Z_1'$};
\node (Z2') at (-15:\Ra) {$Z_2'$};
\draw[->] (P1) to (pt);
\draw[->] (Y) to (pt);
\draw[->] (X) to (pt);
\draw[->] (Z) to (P1);
\draw[->] (barZ) to (P1);
\draw[->] (tildeZ) to (P1);
\draw[->] (Y1') to (Y);
\draw[->] (X3') to (Y);
\draw[->] (X1') to (X);
\draw[->] (X1'') to (X);
\draw[->] (Z1') to (tildeZ);
\draw[->] (Z2') to (barZ);
\draw[dotted,-] (X1') to [bend left=15,"\tiny $11$ flops"] (X2');
\draw[dotted,->] (X2') to [bend left=15,"\tiny flip"] (X3');
\draw[dotted,-] (Y1') to [bend left=30,"\tiny $9$ flops"] (Z);
\draw[dotted,-] (X1'') to [bend right=12,"\tiny $7$ flops",swap] (X2'');
\draw[dotted,->] (X2'') to [bend right=12,"\tiny flip",swap] (tildeZ);
\draw[dotted,-] (Z1') to [bend right=12,"\tiny $6$ flops",swap] (Z2');
\draw[->,dashed] (Y) to ["$\chi_1$", bend right=35,swap] (X);
\draw[->,dashed] (X) to ["$\chi_2$", bend right=25] (tildeZ);
\draw[->,dashed] (tildeZ) to["$\chi_3$", bend right=15] (barZ);
\draw[->] (barZ) to[bend right=12] node[left]{\tiny $\simeq$} node[right]{$\chi_4$} (Z);
\draw[->,dashed] (Z) to ["$\chi_5$", bend right=25,swap] (Y);
\end{tikzpicture}
\]
\caption{Elementary relation from Example~\ref{ex:AZ}.}\label{fig:relation_AZ}
\end{figure}
\section{Elementary relations involving Sarkisov links of conic bundles of type \II} 
\label{sec:constraints}

This section is devoted to the study of elementary relations involving Sarkisov links of conic bundles of type \II that are complicated enough, meaning their covering gonality is large. 
We give some restriction on such relations that will allow us to prove Theorem~\ref{TheoremBirMori}. 
Firstly in Proposition~\ref{pro:boundOnGenus} we cover the case of relations over a base of dimension  $\le n- 2$, where $n$ is the dimension of the Mori fibre spaces, using the BAB conjecture and working with Sarkisov links of large enough covering gonality. 
Secondly, the case of relations over a base of dimension $n-1$ is handled in Proposition~\ref{pro:X4X3X2X1=id}, using only the assumption that the covering gonality is $>1$.

\subsection{A consequence of the BAB conjecture}
\label{sec:BAB}

The following is a consequence of the BAB conjecture, which was recently established in arbitrary dimension by C.~Birkar.

\begin{proposition}\label{pro:BAB}
Let $n$ be an integer, and let $\Ql$ be the set of weak Fano terminal varieties of dimension~$n$. 
There are integers $d,l,m\ge 1$, depending only on $n$, such that for each $X \in \Ql$ the following hold:
\begin{enumerate}
\item
$\dim(H^0(-mK_X))\le l$;
\item
The linear system $\lvert -mK_X \rvert$ is base-point free;
\item \label{BAB:3}
The morphism $\phi\colon X\stackrel{\lvert -mK_X\rvert}{\longto} \p^{\dim(H^0(-mK_X))-1}$ is birational onto its image and contracts only curves $C\subseteq X$ with $C\cdot K_X=0$;
\item
$\deg \phi(X) \le d$.
\end{enumerate} 
\end{proposition}

\begin{proof}
By \cite[Theorem 1.1]{BirkarS}, varieties in $\Ql$ form a bounded family (here we use the observation that for a given $X \in \Ql$, the pair $(X,\emptyset)$ is $\eps$-lc for any $0 < \eps < 1$).
In particular, by \cite[Lemma 2.24]{BirkarA}, the Cartier index of such varieties is uniformly bounded.
Then \cite[Theorem 1.1]{Kollar93} gives the existence of $m = m(n)$ such that $\lvert -mK_X \rvert$ is base-point free for each  $X \in \Ql$. 
By \cite[Theorem 1.2]{BirkarA}, we can increase $m$ if needed, and assume that the associated morphism 
\[
\begin{tikzcd}[column sep=1.5cm]
\phi\colon X \ar[r,"\lvert -mK_X\rvert"] & \p^{\dim(H^0(-mK_X))-1}
\end{tikzcd}
\]
is birational onto its image. 
As it is a morphism, this implies that it contracts only curves $C\subseteq X$ with $C\cdot K_X=0$.
Finally, since $\Ql$ is a bounded family, the two integers $\dim(H^0(-mK_X))$ and $\deg \phi(X)$ are bounded. 
\end{proof}

\begin{corollary} \label{cor:BAB}
Let $\pi\colon Y \to X$ be the blow-up of a reduced but not necessarily irreducible codimension $2$ subvariety $\Gamma \subset X$, $Y \ps \hat Y$ a pseudo-isomorphism, and assume that both $X$ and $\hat Y$ are weak Fano terminal varieties of dimension~$n \ge 3$,
whose loci covered by curves with trivial intersection against the canonical divisor have codimension at least $2$. 
Let $\phi$ be the birational morphism associated to the linear system $\lvert -mK_X \rvert$, with $m$ given by  Proposition~$\ref{pro:BAB}$, and assume that $\Gamma$ is not contained in the exceptional locus $\Ex(\phi)$.
Then through any point of $\Gamma \setminus \Ex(\phi)$ there is an irreducible curve $C\subseteq \Gamma$ with $\gon(C)\le d$ and $C\cdot (-mK_X)\le d$, where $d$ is the integer from Proposition~$\ref{pro:BAB}$.
\end{corollary}

\begin{proof}
We choose the integers $d,l,m\ge 1$ associated to the dimension $n$ in Proposition~\ref{pro:BAB}.
We write $a=\dim(H^0(-mK_{X}))-1$ and $b=\dim(H^0(-mK_{Y}))-1$. 
Using the pseudo-isomorphism  $Y \ps \hat{Y}$, we also have 
$b=\dim(H^0(-mK_{\hat{Y}}))-1$.
By Proposition~\ref{pro:BAB} the morphisms given by the linear systems $\lvert -mK_{X}\rvert$ and $\lvert -mK_{\hat{Y}}\rvert$ are birational onto their images and are moreover pseudo-isomorphisms onto their images, because of the assumption that the locus covered by curves with non-positive intersection against the canonical divisor has codimension at least 2.  

Since $Y\to X$ is the blow-up of $\Gamma$, each effective divisor equivalent to $-mK_{Y}$ is the strict transform of an effective divisor equivalent to $-mK_{X}$ passing through $\Gamma$ (with some multiplicity). In particular, we have $b\le a$ and obtain a commutative diagram
\[\begin{tikzcd}[column sep=1.6cm,row sep=0.3cm]
X \ar[d,"\lvert -mK_{X}\rvert"] \ar[d,"\phi",swap]
& Y\ar [rd,"\lvert -mK_{Y}\rvert",pos=.25,dashed,swap] \ar[l]\ar[rr,dotted,-]
&&\hat{Y} \ar[ld,"\lvert -mK_{\hat{Y}}\rvert",pos=.25]\\
\p^{a}\ar[rr,"\pi",dashed,swap]&&\p^{b}
\end{tikzcd}
\]
where $\pi$ is a linear projection away from a linear subspace $\Ll\simeq \p^r$ of $\p^{a}$ containing the image of $\Gamma$.
Recall that we write $\phi\colon X\to \p^{a}$ the morphism given by $\lvert -mK_{X}\rvert$. 
The variety $\phi(X) \subseteq \p^a$ has dimension $n$ and degree $\le d$ (Proposition~\ref{pro:BAB}), and is not contained in a hyperplane section. 
Since by assumption $\Gamma \subsetneq \Ex(\phi)$, we get that $\phi$ induces a birational morphism from $\Gamma$ to $\phi(\Gamma)$.

We now prove that there is no (irreducible) variety  $S\subseteq \phi(X)\cap \Ll$ of dimension $n-1$ (recall that $\phi(\Gamma)\subseteq \phi(X)\cap \Ll$ has dimension $n-2$).
Indeed, otherwise the strict transform of $S$ on $X$ would be a variety $S_X\subset X$ birational to $S$, so its strict transform in $\hat Y$, and in $\p^b$ is again birational to $S$ (as the birational map from $Y$ to its image in $\p^b$ is a pseudo-isomorphism). 
The linear system of the rational map $X\rat \p^b$ is obtained from the linear system associated to $X\rat \p^a$ by taking the subsystem associated to hyperplanes through $\Ll$.
Hence, if $S\subseteq \Ll$, then every element of the linear system $\lvert-mK_{Y}\rvert$ contains the strict transform $S_Y$ of $S$ in $\hat Y$.
This is impossible, as $\lvert-mK_{\hat Y}\rvert$ is base-point free (Proposition~\ref{pro:BAB}).

Now, the fact that $\phi(X)\cap \Ll\subseteq\p^a$ does not contain any variety of dimension $\ge n-1$ implies, by B\'ezout Theorem, that all its irreducible components of dimension $n-2$ have degree $\le d$
(indeed, $\phi(X)$ is irreducible of degree $\le d$ and dimension $n-1$, and $\Ll$ is a linear subspace).
Therefore, each of the irreducible components of $\phi(\Gamma)$ (birational to $\Gamma$) has degree $\le d$. 

We are now able to finish the proof, by showing that through any point $q\in\Gamma \setminus \Ex(\phi)$ 
there is an irreducible curve $C\subseteq \Gamma$ with $\gon(C)\le d$ and $C\cdot (-K_X)\le d$. 
Since $\Gamma\to \phi(\Gamma)$ is a local isomorphism at $q$, it suffices to take a general linear projection from $\p^a$ to a linear subspace of dimension $n-2$, and to take $C$ equal to the preimage of a line through the image of $\phi(q)$.
\end{proof}

\begin{proposition}\label{pro:boundOnGenus}
For each dimension $n\ge 3$, there exists an integer $d_n\ge 1$ depending only on $n$ such that the following holds.
If $\chi$ is a Sarkisov link of conic bundles of type \II that arises in an elementary relation induced by rank~$3$ fibration $T/B$ with $\dim(T)=n$ and $\dim(B)\le n-2$, then $\covgon(\chi)\le \max\{d_n,8\conngon(T)\}$. 
\end{proposition}

\begin{proof}
We choose $d_n\ge 8$ to be bigger than the integers $d$ given by Proposition~\ref{pro:BAB} for the dimensions $3,\dots,n$, and prove the result for this choice of $d_n$.

The Sarkisov link $\chi$, which is dominated by $T/B$ by assumption, has the form 
\[\begin{tikzcd}[link]
 Y_1\ar[dd]\ar[r,dotted,-]& T \ar[r,dotted,-] & Y_2\ar[dd]  \\ \\
X_1 \ar[rr,"\chi",dashed] \ar[rd]&& X_2 \ar[ld]\\
&\tilde{B}\ar[dd]\\ \\
&B
\end{tikzcd}
\]
where $X_1,X_2,Y_1,Y_2$ have dimension $n$ and $\tilde{B}$ has dimension $n-1$.
Since $\dim B \le n-2$, we have $\rho(\tilde{B}/B) \ge 1$, and on the other hand $\rho(Y_i/B) \le 3$,  for $i=1,2$, which implies that $\rho(\tilde{B}/B) = 1$, and that the birational contractions $T\ps Y_1$, $T\ps Y_2$ are pseudo-isomorphisms.  
Moreover, $Y_1\to X_1$ contracts a divisor $E$ onto a variety $\Gamma_1\subset X_1$ of dimension $n-2$, birational to its image $\tilde{\Gamma}\subset \tilde{B}$ via the morphism $X_1\to \tilde{B}$ (Lemma~\ref{lem:SarkiIIConic}). 
We need to check that $\covgon(\Gamma_1)=\covgon(\tilde{\Gamma})\le d_n$, where $d_n$ is chosen as above.
We may then assume that $\covgon(\tilde{\Gamma})>1$. 

Now, $\tilde B/B$ is a klt Mori fibre space by Lemma~\ref{lem:B'toB} and $X_1/B$ is a rank $2$ fibration by Lemma~\ref{lem:StableUnderMMP}\ref{stab_under_MMP:1}.
By Lemma~\ref{lem:weakFano}, the rank~$2$ fibration $X_1/B$ is pseudo-isomorphic, via a sequence of log-flips over $B$, to another rank~$2$ fibration $X/B$ such that $-K_{X}$ is relatively nef and big over $B$. 
We then use Lemma~\ref{lem:2raysOverFlip} to obtain a sequence of log-flips $Y_1\ps Y$ over $B$ such that the induced map $Y\to X$ is a divisorial contraction. 
By Lemma~\ref{lem:weakFano} again, we get a sequence of log-flips over $B$ from $Y/B$ to another rank~$3$ fibration $\hat{Y}/B$ such that $-K_{\hat{Y}}$ is relatively nef and big over $B$.
\[\begin{tikzcd}[link]
 Y_1\ar[dd]\ar[r,dotted,-]&Y\ar[dd]\ar[r,dotted,-]& \hat{Y}\ar[ldddd]  \\ \\
X_1  \ar[rdd]\ar[r,dotted,-]& X\ar[dd]\\ \\
&B
\end{tikzcd}
\]
As $\covgon(\tilde{\Gamma})>1$, by Lemma \ref{lem:fibresRatConnected}\ref{connected:flips} the codimension $2$ subvariety $\Gamma_1\subset X_1$ is not contained in the base-locus of the pseudo-isomorphism $X_1\ps X$. 
So the image $\Gamma\subset X$ of $\Gamma_1$ is birational to $\Gamma_1$, and it suffices to show that $\covgon(\Gamma)\le d_n$. 

We take a general point $p\in B$, and consider the fibres over $p$ in $X$, $Y$ and $\hat{Y}$ respectively, that we denote by $X_p$, $Y_p$ and $\hat{Y}_p$, and which are varieties of dimension 
\[n_0=n-\dim B \in \{ 2,\dots,n\}.\] 

By Corollary~\ref{cor:weakFanoFiber} the two varieties $X_p$ and $\hat{Y}_p$ are weak Fano terminal varieties. Moreover, $Y_p$ and $\hat{Y}_p$ are pseudo-isomorphic, as $Y\ps \hat{Y}$ is a sequence of log-flips over $B$.

Observe that $\tilde{\Gamma}\subset \tilde{B}$ is a hypersurface and that $\tilde{\Gamma}\to B$ is surjective.
Indeed, otherwise $\tilde{\Gamma}$ would be the preimage of a divisor on $B$, and we would have $\covgon(\tilde{\Gamma}) = 1$, as the preimage of each point of $\tilde{B}\to B$ is covered by rational curves (Lemma~\ref{lem:B'toB}), in contradiction with our assumption. This implies that the morphism $\Gamma\to B$ induced by the restriction of $X/B$ is again surjective.

We then denote by $\Gamma_p\subset X_p$ the codimension $2$ subscheme $\Gamma_p=\Gamma\cap X_p$, which is the fibre of $\Gamma\to B$ over $p$, and which is not necessarily irreducible.  
Observe that $Y_p\to X_p$ is the blow-up of $\Gamma_p$, as $Y\to X$ is locally the blow-up of $\Gamma$ (by Lemma~\ref{lem:divContToCodim2}) and because the fibre over $p$ is transverse to $\Gamma$ (Lemma~\ref{lem:SarkiIIConic}\ref{SarkiII:3}).

Suppose first that $n_0=2$, which corresponds to $\dim(\Gamma)=\dim(B)$. 
In this case, $X_p$ and $Y_p\simeq \hat{Y}_p$ are smooth del Pezzo surfaces, because by Corollary~\ref{cor:weakFanoFiber} the locus covered by curves trivial against the canonical divisor has codimension 2, hence is empty in the case $n_0 = 2$. 
Moreover $\Gamma_p$ is a disjoint union of $r$ points, where $r$ is the degree of the field extension $\C(B)\subseteq \C(\Gamma_1)$. As $Y_p$ is obtained from $X_p$  by blowing-up $\Gamma_p$, the degree of the field extension is at most $8$, which implies that $\covgon(\Gamma)\le 8\cdot \covgon(B)\leq8\conngon(T)$ (Lemma~\ref{lem:CovgonRatMorphism}). 

We now consider the case $n_0\ge 3$, which implies that $\Gamma_p$ has dimension $n_0-2\ge 1$. 
We consider the morphism $\varphi :=  \lvert -mK_X \rvert \times \eta \colon X \rat \p^N \times B$, where $\eta\colon X\to B$ is the morphism already considered and $m$ is given by Proposition~$\ref{pro:BAB}$ applied in dimension $n_0$.
The restriction of $\varphi$ to the fibre of $p$ is a birational morphism $\varphi_p\colon X_p\to \p^N$, described in Proposition~$\ref{pro:BAB}$.
We apply Corollary~\ref{cor:BAB} to the blow-up $Y_p\to X_p$ of $\Gamma_p$, the pseudo-isomorphism $Y_p\ps \hat{Y}_p$. 
The fact that the loci on $X_p$ or $\hat Y_p$ covered by curves with trivial intersection against the canonical divisor has codimension at least 2 follows from Corollary~\ref{cor:weakFanoFiber}.  
Lemma~\ref{lem:fibresRatConnected}\ref{connected:flips} implies that $\Gamma_p$ is not contained in $\Ex(\phi_p)$ because $\covgon(\Gamma_p)>1$.
We obtain from Corollary~\ref{cor:BAB} that for a general $p$, $\Gamma_p \setminus \Ex(\phi_{p})$ is covered by curves of gonality at most $d_{n}$.

In conclusion, we have found an open set $U \subseteq \Gamma \setminus \Ex(\phi)$ covered by curves of gonality at most $d_n$, as expected. 
\end{proof}

\begin{remark} \label{rem:BAB} \phantomsection
It is not clear to us whether Proposition~\ref{pro:boundOnGenus} could also hold for a link $\chi$ of type \II between arbitrary Mori fibre spaces.

For instance in the case of threefolds, if $\chi$ is a link of type \II between del Pezzo fibrations that starts with the blow-up a curve of genus $g$ contained in one fibre, we suspect that $g$ cannot be arbitrary large but we are not aware of any bound in the literature.
\end{remark}

\subsection{Some elementary relations of length 4}

\begin{proposition} \label{pro:X4X3X2X1=id}
Let $\chi_1$ be a Sarkisov link of conic bundles of type~\II with $\covgon(\chi_1) > 1$. 
Let $T/B$ be a rank~$3$ fibration with $\dim B = \dim T -1$, which factorises through the Sarkisov link $\chi_1$.
Then, the elementary relation associated to $T/B$ has the form 
\[\chi_4\circ \chi_3\circ \chi_2\circ \chi_1=\id,\]
where $\chi_3$ is a Sarkisov link of conic bundles of type \II that is equivalent to $\chi_1$.  
\end{proposition}

\begin{figure}[ht]
\begin{align*}
\begin{tikzpicture}[scale=1.1,font=\small,outer sep=-1pt]
\node (Y1) at (60+90:\Ra) {$Y_1$};
\node (Y2) at (120+90:\Ra) {$Y_2$};
\node (T3') at (180+90:\Ra) {$T_3'$};
\node (T3) at (240+90:\Ra) {$T_3$};
\node (T4) at (300+90:\Ra) {$T_4$};
\node (T4') at (360+90:\Ra) {$T_4'$};
\node (X1) at (60+90:\Rb) {$X_1$};
\node (X2) at (120+90:\Rb) {$X_2$};
\node (X3) at (300:\Rc) {$X_3$};
\node (X4) at (420:\Rc) {$X_4$};
\node (Y3') at (180+90:\Rb) {$Y_3'$};
\node (Y3) at (240+90:\Rb) {$Y_3$};
\node (Y4) at (300+90:\Rb) {$Y_4$};
\node (Y4') at (360+90:\Rb) {$Y_4'$};
\node (B) at (0:0cm) {$B$};
\node (Bhat) at (180:\Rc) {$\hat{B}$};
\draw[dotted,-] (Y1) to [bend right=21] (Y2)
		(Y2) to [bend right=21] (T3')
		(T3') to [bend right=21] (T3)
		(T3) to [bend right=21] (T4)
		(T4) to [bend right=21] (T4')
		(T4') to [bend right=21] (Y1)
		(X2) to [bend right=21] (Y3')
		(Y3) to [bend right=21] (Y4)
		(Y4') to [bend right=21] (X1);
\draw[->] (Bhat) to(B);
\draw[->] (X3) to(B);
\draw[->] (X4) to(B);
\draw[->] (T4') to [auto,"\tiny $E_1$"] (Y4');
\draw[->] (Y1) to [auto,pos=.1,"\tiny $E_1$"] (X1);
\draw[->] (Y2) to [auto,swap,pos=.1,"\tiny $E_2$"] (X2);
\draw[->] (T3') to [auto,swap,"\tiny $E_2$"] (Y3');
\draw[->] (T3) to [auto,swap,"\tiny $G$"] (Y3);
\draw[->] (T4) to [auto,"\tiny $G$"] (Y4);
\draw[->] (Y3') to [auto,swap,"\tiny $G$"] (X3);
\draw[->] (Y4') to [auto,"\tiny $G$"] (X4);
\draw[->] (Y3) to [auto,swap,pos=.7,"\tiny $E_2$"] (X3);
\draw[->] (Y4) to [auto,pos=.7,"\tiny $E_1$"] (X4);
\draw[->] (X1) to  (Bhat);
\draw[->] (X2) to  (Bhat);
\draw[dashed,->] (X1) to [bend right=30,auto,swap,"$\chi_1$"] (X2);
\draw[dashed,->] (X2) to [auto,swap,pos=.3,"$\chi_2$"] (X3);
\draw[dashed,->] (X3) to [bend right=38,auto,swap,"$\chi_3$"] (X4);
\draw[dashed,->] (X4) to [auto,swap,pos=.7,"$\chi_4$"] (X1);
\end{tikzpicture}
&&
\begin{tikzpicture}[scale=1.1,font=\small,outer sep=-1pt]
\node (Y1) at (135:\Ra) {$Y_1$};
\node (Y2) at (-135:\Ra) {$Y_2$};
\node (Y3) at (-45:\Ra) {$Y_3$};
\node (Y4) at (45:\Ra) {$Y_4$};
\node (X1) at (135:\Rb) {$X_1$};
\node (X2) at (-135:\Rb) {$X_2$};
\node (X3) at (-45:\Rb) {$X_3$};
\node (X4) at (45:\Rb) {$X_4$};
\node (B) at (0:0cm) {$B$};
\node (Bl) at (180:\Rc) {$\hat{B}$};
\node (Br) at (0:\Rc) {$\hat{B'}$};
\draw[->] (Y1) to [auto,swap,"\tiny $E_1$"] (X1);
\draw[->] (Y2) to [auto,"\tiny $E_2$"] (X2);
\draw[->] (Y3) to [auto,swap,"\tiny $E_2$"] (X3);
\draw[->] (Y4) to [auto,"\tiny $E_1$"] (X4);
\draw[->] (X1) to  (Bl);
\draw[->] (X2) to  (Bl);
\draw[->] (X3) to  (Br);
\draw[->] (X4) to  (Br);
\draw[->] (Bl) to  (B);
\draw[->] (Br) to  (B);
\draw[dashed,->] (X1) to [bend right=38,auto,swap,"$\chi_1$"] (X2);
\draw[dotted,->] (X2) to [bend right=38,auto,swap,"$\chi_2$"] (X3);
\draw[dashed,->] (X3) to [bend right=38,auto,swap,"$\chi_3$"] (X4);
\draw[dotted,->] (X4) to [bend right=38,auto,swap,"$\chi_4$"] (X1);
\draw[dotted,-] (Y1) to [bend right=38] (Y2)
		(Y2) to [bend right=38] (Y3)
		(Y3) to [bend right=38] (Y4)
		(Y4) to [bend right=38] (Y1)
		(Bl) to [bend right=20] (Br);
\end{tikzpicture}
\end{align*}
\vspace{-7mm}
\[
\begin{tikzpicture}[scale=1.1,font=\small,outer sep=-1pt]
\foreach \x in {1,...,4}{
\foreach \n in {1,2}{
\FPeval{\internal}{clip(2*(\x-2)+\n+1)}
\ifthenelse{\internal=0}{\FPeval{\internal}{8}}{}
\FPeval{\thetest}{clip(\x-\n)}
\ifthenelse{\thetest = 0 \OR \thetest = 2}
{\def\exponent{'}}
{\def\exponent{}}
\node (Y\internal) at ({\n*45+\x*90-22.5}:\Rb) {$Y_{\x}\exponent$};
\node (T\internal) at ({\n*45+\x*90-22.5}:\Ra) {$T_{\x}\exponent$};
}}
\node (B) at (0:0cm) {$B$};
\foreach \x in {1,...,4}{
\node (X\x) at (\x*90+45:\Rc) {$X_{\x}$};
\draw[->] (X\x) to (B);
}
\draw[dotted,-] (T1) to [bend right=15] (T2)
		(T2) to [bend right=15] (T3)
		(T3) to [bend right=15] (T4)
		(T4) to [bend right=15] (T5)
		(T5) to [bend right=15] (T6)
		(T6) to [bend right=15] (T7)
		(T7) to [bend right=15] (T8)
		(T8) to [bend right=15] (T1)
		(Y1) to [bend right=15] (Y2)
		(Y3) to [bend right=15] (Y4)
		(Y5) to [bend right=15] (Y6)
		(Y7) to [bend right=15] (Y8);
\draw[dashed,->] (X1) to [bend right=35,swap,"$\chi_1$"] (X2);
\draw[dashed,->] (X2) to [bend right=35,swap,"$\chi_2$"] (X3);
\draw[dashed,->] (X3) to [bend right=35,swap,"$\chi_3$"] (X4);
\draw[dashed,->] (X4) to [bend right=35,swap,"$\chi_4$"] (X1);
\draw[->] (Y1) to [pos=0,"\tiny $E_1$"] (X1);
\draw[->] (Y2) to [swap,pos=0,"\tiny $E_2$"] (X2);
\draw[->] (Y3) to [pos=0,"\tiny $F_1$"] (X2);
\draw[->] (Y4) to [swap,pos=0,"\tiny $F_2$"] (X3);
\draw[->] (Y5) to [pos=0,"\tiny $E_2$"] (X3);
\draw[->] (Y6) to [swap,pos=0,"\tiny $E_1$"] (X4);
\draw[->] (Y7) to [pos=0,"\tiny $F_2$"] (X4);
\draw[->] (Y8) to [swap,pos=0,"\tiny $F_1$"] (X1);
\draw[->] (T1) to [pos=0.2,"\tiny $F_1$"] (Y1);
\draw[->] (T2) to [swap,pos=0.2,"\tiny $F_1$"] (Y2);
\draw[->] (T3) to [pos=0.2,"\tiny $E_2$"] (Y3);
\draw[->] (T4) to [swap,pos=0.2,"\tiny $E_2$"] (Y4);
\draw[->] (T5) to [pos=0.2,"\tiny $F_2$"] (Y5);
\draw[->] (T6) to [swap,pos=0.2,"\tiny $F_2$"] (Y6);
\draw[->] (T7) to [pos=0.2,"\tiny $E_1$"] (Y7);
\draw[->] (T8) to [swap,pos=0.2,"\tiny $E_1$"] (Y8);
\end{tikzpicture}
\]
\caption{The elementary relation associated to $T/B$ in cases~\ref{BisnotBhat}, \ref{BisnotBhat_bis} and \ref{BisBhat} of the proof of Proposition~\ref{pro:X4X3X2X1=id}.
Varieties are organized in circles according to their Picard rank over $B$.}
\label{fig:relations}
\end{figure}

\begin{proof}
The Sarkisov link $\chi_1$ is given by a diagram
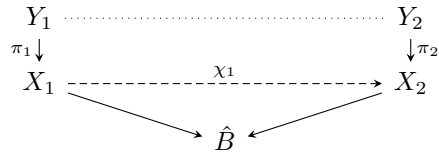
\[\begin{tikzcd}[link]
 Y_1\ar[dd,"\pi_1",swap]\ar[rr,dotted,-]&& Y_2\ar[dd,"\pi_2"]  \\ \\
X_1\ar[rr,"\chi_1",dashed,->]  \ar[rd]&& X_2 \ar[ld]\\
&\hat{B}
\end{tikzcd}
\]
where $X_1,X_2,Y_1,Y_2$ are varieties of dimension $n$, and $\dim \hat{B} = n-1$. 
Denote by $E_1\subset Y_1$ and $E_2\subset Y_2$ the respective exceptional divisors of the divisorial contractions $\pi_1$ and $\pi_2$. 
We denote again by $E_1,E_2\subset T$ the strict transforms of these divisors, under the birational contractions $T\rat Y_1$ and $T\rat Y_2$. 
Then by Lemma~\ref{lem:SarkiIIConic}\ref{SarkiII:4}, $E_1 \cup E_2$ is a pair of divisors of type \II for $Y_1/\hat{B}$, hence also for $T/B$ by Lemma~\ref{lem:gonalityOfTypeII}. 
By Proposition~\ref{pro:rankrFibrations}\ref{rankr:4}, we are in one of the following mutually exclusive three cases:
\begin{enumerate}
\item\label{BisnotBhat}
$B$ is $\Q$-factorial, and there exists a divisor $G$ of type \I for $T/B$.

\item\label{BisnotBhat_bis}
$B$ is not $\Q$-factorial. 

\item\label{BisBhat}
$B$ is $\Q$-factorial, and there exists another pair $F_1 \cup F_2$ of divisors of type \II for $T/B$. 
\end{enumerate}

We denote $\{X_i/B_i\}$ the finite collection of all rank~$1$ fibrations dominated by $T/B$.
In each case we are going to show that this collection has cardinal 4.

Suppose first that \ref{BisnotBhat} holds.
By Proposition~\ref{pro:rankrFibrations}\ref{rankr:0}\&\ref{rankr:3} and Lemma~\ref{lem:factorThroughRank1}, we can obtain such an $X_i/B_i$ by a birational contraction contracting one the following four sets of divisors: 
$\{E_1\}$, $\{E_2\}$, $\{E_1,G\}$ or $\{E_2,G\}$. 
Moreover $X_i/B_i$ is determined up to isomorphism by such a choice of contracted divisors:
\begin{itemize}
\item If $T \rat X_i$ contracts $\{E_1,G\}$ or $\{E_2,G\}$, then $\rho(X_i) = \rho(T)  - 2$ which implies $\rho(B_i/B) = 0$, that is, $B_i \to B$ is an isomorphism. 
Then $X_i$ is uniquely determined by Lemma~\ref{lem:corti2.7}\ref{corti2.7:1}.
\item If $T \rat X_i$ contracts $\{E_1\}$ or $\{E_2\}$, then $\rho(B_i/B) = 1$, and $B_i \rat B$ is a birational contraction contracting the image of the divisor $G$. 
Then such a $B_i$ is uniquely determined by Lemma~\ref{lem:corti2.7}\ref{corti2.7:2}.
\end{itemize}
In conclusion the relation given by Proposition~\ref{pro:from T3} has the form 
\[\chi_4\circ \chi_3\circ \chi_2\circ \chi_1=\id,\]
and more precisely, up to a cyclic permutation exchanging the role of $\chi_1$ and $\chi_3$, we have a commutative diagram as in Figure~\ref{fig:relations}, top left, where $\chi_2$ and $\chi_4$ have respectively type \III and \I, and $\chi_1$ and $\chi_3$ are equivalent Sarkisov links of type II.

Now consider Case~\ref{BisnotBhat_bis}. 
As $\hat B$ is $\Q$-factorial (Proposition~\ref{pro:sing_of_B}), we have $\hat B \neq B$, hence $\rho(\hat B/B) = 1$ and the morphism $\hat B\to B$ is a small contraction. 
By uniqueness of log-flip, there are exactly two small contractions from a $\Q$-factorial variety to $B$.
Denote $\hat B' \to B$ the other one.
Then for each $X_i/B_i$, we have $B_i \simeq \hat B$ or $\hat B'$, and $\rho(X_i/B) = 2$.
Hence the birational contraction $T \rat X_i$ contracts exactly one divisor, which must be $E_1$ or $E_2$. 
Again this gives four possibilities.
The actual existence of $X_3/ \hat B'$ and $X_4/ \hat B'$ arises  from the two-rays games $X_1/B$ and $X_2/B$. 
We get a relation as in Figure~\ref{fig:relations}, top right, with $\chi_1, \chi_3$ of type \II and $\chi_2, \chi_4$ of type \IV.

Finally consider Case~\ref{BisBhat}. 
Then by Proposition~\ref{pro:rankrFibrations}\ref{rankr:0}\ref{rankr:3}, each birational contraction $T \rat X_i$ contracts one divisor among $E_1 \cup E_2$, and another one among $F_1 \cup F_2$. 
Again this gives four possibilities.
In each case $\rho(B_i/B) = 0$ hence $B_i$ is isomorphic to $B$, and then Lemma~\ref{lem:corti2.7}\ref{corti2.7:1} says that $X_i$ is determined up to isomorphism by such a choice. 
We obtain a relation with four links of type \II, as on Figure~\ref{fig:relations}, bottom.
\end{proof}

\begin{remark}
Example~\ref{ex:X3X2X1=1} illustrates why the assumption on the covering gonality is necessary in Proposition~\ref{pro:X4X3X2X1=id}.
\end{remark}

\subsection{Proof of Theorem \ref{TheoremBirMori}}
\label{sec:proofBirMori}
In order to prove Theorem~\ref{TheoremBirMori}, we use the generators and relations of $\BirMori(X)$ which are given in Theorem~\ref{thm:sarkisov}. The key results are then Propositions~\ref{pro:boundOnGenus} and~\ref{pro:X4X3X2X1=id}. 

\begin{proof}[Proof of Theorem \ref{TheoremBirMori}]
We choose the integer $d$ associated to the dimension $n$ by Proposition~\ref{pro:boundOnGenus}, and set $M = \max\{d, 8\conngon(X)\}$.
By Theorem~\ref{thm:sarkisov}\ref{sarkisov1}, the groupoid $\BirMori(X)$ is generated by Sarkisov links and automorphisms of Mori fibre spaces. 
By Theorem~\ref{thm:sarkisov}\ref{sarkisov2}, the relations are generated by elementary relations, so it suffices to show that every elementary relation is sent to the neutral element in the group \[\bigast_{C\in \CB(X)} \left(\bigoplus_{\MC(C)}\Z/2\right).\] 

Let $\chi_t\circ\cdots\circ\chi_1=\id$ 
be an elementary relation, coming from a rank~$3$ fibration $T/B$.  
We may assume that one of the $\chi_i$ is a Sarkisov link of conic bundles of type \II with $\covgon(\chi_i) > M$, otherwise the relation is sent onto the neutral element as all $\chi_i$ are sent to the neutral element.
We may moreover conjugate the relation and assume that $\chi_1$ is a Sarkisov link of conic bundles of type \II with $\covgon(\chi_1)>M$. 
By Proposition~\ref{pro:boundOnGenus}, we have $\dim(B)=n-1$. 
Then, Proposition~\ref{pro:X4X3X2X1=id} implies that $t=4$ and that $\chi_1$ and $\chi_3$ are equivalent Sarkisov links of conic bundles of type \II. 
Applying the same argument to the relation $\chi_1\circ \chi_4\circ \chi_3\circ \chi_2=\id$ we either find that both $\chi_2$ and $\chi_4$ are sent to the neutral element or are equivalent Sarkisov links of conic bundles of type \II (again by Proposition~\ref{pro:X4X3X2X1=id}). 
Moreover, all the conic bundles involved in this relation are equivalent. 
This proves the existence of the groupoid homomorphism.

The fact that it restricts to a group homomorphism from $\Bir(X)$ is immediate, and the fact that it restricts as a group homomorphism
\[
\Bir(X/B)\to \bigoplus_{\MC(X/B)} \Z/2
\]
follows from Lemma~\ref{lem:SarkisovOverBase}.
\end{proof}

\section{Image of the group homomorphism given by Theorem \ref{TheoremBirMori}}
\label{sec:surjectivity}

In this section, we study the image of $\Bir(X)$ under the group homomorphism given by Theorem~\ref{TheoremBirMori}, and more precisely the image of $\Bir(X/B)\to \bigoplus_{\MC(X/B)} \Z/2$ for some conic bundles $X/B$. 

\subsection{A criterion}
\label{sec:criterion}

Given a birational map between conic bundles over a curve $B$, for each point $p\in B$ one can define the number of base-points that lies on the fibre $p$, as proper or infinitely near points.
This amounts to counting how many links one has to perform above the point. 
In the next definition we generalise this to any dimension, by replacing the point $p$ with an irreducible hypersurface of $B$. 
As the targets of our group homomorphisms are direct sums of $\Z/2\Z$, it is natural to only count the parity of the multiplicity.

\begin{definition}\label{def:ParityPhiGamma}
Let $(X/B,\Gamma)$ be a marked conic bundle, and $\phi\colon X/B \rat Y/B$ a birational map over $B$ between conic bundles.
For a general point $p\in \Gamma$, and an irreducible curve $C\subseteq B$ transverse to $\Gamma$ at $p$,
let $b \in \N$ be the number of base-points of the birational surface map $\eta_X^{-1}(C)\rat \eta_Y^{-1}(C)$ induced by $\phi$ that are equal or infinitely near to a point of the fibre of $p$.
We call the class $\bar b \in \Z/2$ \emph{the parity of $\phi$ along $\Gamma$}.
\end{definition}

The following lemma shows that this definition does not depend on the choice of $p$ or $C$. 
We shall use it to compute the image of the group homomorphism of Theorem~\ref{TheoremBirMori} by studying locally a birational map near a hypersurface $\Gamma$ of the base.

\begin{lemma}\label{lem:SequenceParity}
Let $\eta_X\colon X\to B$ and $\eta_Y\colon Y\to B$ be two conic bundles,  $\phi\colon X\rat Y$  a birational map over $B$, and  $\Gamma\subset B$ an irreducible hypersurface not contained in the discriminant locus of $X/B$. 

For any decomposition $\phi=\chi_t\circ \cdots\circ \chi_1$ as in Lemma~$\ref{lem:SarkisovOverBase}$, the
parity of $\phi$ along $\Gamma$ is equal to the parity of the number of indexes $i\in \{1,\dots,t\}$ such that $\chi_i$ is a Sarkisov link of type \II whose marking $\Gamma_i\subset B_i$ is sent to $\Gamma$ via $B_i/B$. 
\end{lemma}

\begin{proof}
Fix a decomposition $\phi=\chi_t\circ \cdots\circ \chi_1$ as in Lemma~$\ref{lem:SarkisovOverBase}$, a general point $p\in \Gamma$ and  an irreducible curve $C\subseteq B$ transverse to $\Gamma$ at $p$. 
In particular $p$ is a smooth point of both $\Gamma$ and $C$. 
For $i=0,\dots,t$, we denote by $\eta_i\colon X_i\to B$ the morphism given by the composition $X_i\to B_i\to B$.

It suffices to prove, for $i=0,\dots,t$, that the following holds:

\begin{enumerate}[$(a)$]
\item \label{parity:a}
The morphism $\eta_i^{-1}(C)\to C$ has general fibre $\p^1$, and the fibre over $p$ is $\p^1$ (this means that $\Gamma$ is not in the discriminant locus).
\item \label{parity:b}
If $i\ge 1$, then $\chi_i\circ \cdots\circ \chi_1$ induces a birational map between surfaces over $C$ 
\[\eta_0^{-1}(C)=\eta_X^{-1}(C)\rat \eta_i^{-1}(C)\] 
and the number of base-points that are equal or infinitely near to a point of the fibre of $p$ has the same parity as the number of integers $j\in \{1,\dots,i\}$ such that $\chi_j$ is a Sarkisov link of type \II with marking $\Gamma_j\subset B_j$, sent to $\Gamma$ via $B_j/B$.
\end{enumerate}

We proceed by induction on $i$. 
If $i=0$,  \ref{parity:a} follows from the assumption that $\Gamma$ is not contained in the discriminant locus of $X/B$, and \ref{parity:b} is clear.

For $i\ge1$, the birational map $\chi_i$ induces a birational map over $C$ 
\[\theta_i\colon \eta_{i-1}^{-1}(C)\rat \eta_i^{-1}(C).\] 
If $\chi_i$ is a Sarkisov link of type \II with marking $\Gamma_i\subset B_i$, sent to $\Gamma$ via $B_i/B$, 
it follows from the description of $\chi_i$ given in Lemma~\ref{lem:SarkiIIConic} that the restriction $\theta_i$ is the composition of the blow-up of a point on the fibre of $p$, the contraction of the strict transform of the fibre and of a birational map that is an isomorphism over an open subset of $C$ that contains the fibre of $p$. 
This achieves the proof of \ref{parity:a} and \ref{parity:b} in this case, using the induction hypothesis.

If $\chi_i$ is a Sarkisov link of type \II with a marking not sent to $\Gamma$ or if $\chi_i$ is a Sarkisov link of type \I or \III, then the restriction $\theta_i$ of $\chi_i$ is an isomorphism  over an open subset of $C$ that contains the fibre of $p$. 
This follows again from Lemma~\ref{lem:SarkiIIConic} if the Sarkisov link is of type \II and from Corollary~\ref{cor:SarkiIConic} if it is of type \I or \III.
As before, this gives the result by applying the induction hypothesis.
\end{proof}

To simplify the notation in the group $\bigoplus_{\MC(X/B)} \Z/2$, we will identify an equivalence class of marked conic bundles in $\MC(X/B)$ with the associated generator of $\Z/2$.  
We can then speak about sums of elements of $\MC(X/B)$, which we see in $\bigoplus_{\MC(X/B)} \Z/2$, twice the same class being equal to zero. 

\begin{corollary}\label{cor:Parity}
Let $X/B$ be a conic bundle, where $\dim(X)\ge 3$, and $\phi\in \Bir(X/B)$. 
The image of $\phi$ under the group homomorphism 
\[ \Bir(X/B)\to \bigoplus_{\MC(X/B)} \Z/2\]
of Theorem~\ref{TheoremBirMori}  is equal to the 
sum of equivalence classes of marked conic bundles $(X/B,\Gamma)$ with $\covgon(\Gamma) > \max\{ d, 8\conngon(X) \}$ such that the parity of $\phi$ along $\Gamma$ is odd. 
\end{corollary}

\begin{proof}
Set $M=\max\{ d, 8\conngon(X) \}$.
Using Lemma~\ref{lem:SarkisovOverBase}, we decompose $\phi$ as $\phi=\chi_t\circ \cdots \circ\chi_1$ where each $\chi_i$ is a Sarkisov link of conic bundles from $X_{i-1}/B_{i-1}$ to $X_i/B_i$. 
Denote by $J\subseteq\{1,\dots,t\}$ the subset of indexes $i$ such that the Sarkisov link $\chi_i$ is of type \II and satisfies $\covgon(\chi_i) >M$. 
By definition of the group homomorphism 
\[\Bir(X/B)\to \bigoplus_{\MC(X/B)} \Z/2\]
of Theorem~\ref{TheoremBirMori}, the image of $\phi$ is the sum of the equivalence classes of marked conic bundles of $\chi_i$  where $i$ runs over $J$. For each $i\in J$, the marked conic bundle of $\chi_i$ is equal to $(X_i/B_i,\hat\Gamma_i)$ for some irreducible hypersurface $\hat\Gamma_i\subset B_i$ (see Definition~\ref{def:marked_cb_link}); moreover, one has
with $\covgon(\chi_i)=\covgon(\hat\Gamma_i)$ (Definition~\ref{def:Covgonchi}), so $\covgon(\hat\Gamma_i)>M$.
Hence, $(X_i/B_i,\hat\Gamma_i)$ is equivalent to $(X/B,\Gamma_i)$, where $\Gamma_i\subset B$ is the image of $\hat\Gamma_i\subset B_i$ via $B_i/B$. This implies that the image of $\phi$ is the sum of the classes of $(X/B,\Gamma_i)$, where $i$ runs over $J$.

By Lemma~\ref{lem:SequenceParity}, this sum is equal to the sum of equivalence classes of marked conic bundles $ (X/B,\Gamma)$ with $\covgon(\Gamma)>M$ and such that the parity of $\phi$ along $\Gamma$ is odd. 
\end{proof}

\subsection{The case of trivial conic bundles and the proof of Theorem~\ref{TheoremBirPnnotsimple}}
\label{sec:trivialBundle}

Given a variety $B$, let $X=\p^1\times B$, and $X/B$ the second projection. 
The group $\Bir(X/B)$ is canonically isomorphic to $\PGL_2(\C(B))$, via the action 
\[
\begin{tikzcd}[map]
\PGL_2(\C(B)) \times X & \rat &X\\
 \left(\begin{pmatrix} a(t) & b(t)\\ c(t) & d(t) \end{pmatrix} , ([u:v],t)\right) & \mapsrat &([a(t)u+b(t)v:c(t)u+d(t)v],t).
\end{tikzcd}
\]

For $B=\p^{n-1}$, the group $\Bir(X/B)$ corresponds, via a birational map $X\rat \p^n$ sending the fibres of $X/B$ to lines through a point $p\in \p^n$, to the subgroup of the Jonqui\`eres group associated to $p$ consisting of birational maps of $\p^n$ that preserves a general line through $p$ (in general a Jonqui\`eres element permutes such lines). Hence, $\Bir(X/B)$ corresponds to the factor $\PGL_2(\C(x_2, \dots, x_n))$ of the group $\PGL_2(\C(x_2, \dots, x_n)) \rtimes \Bir(\p^{n-1}) \subseteq \Bir(\p^n)$ described in \S\ref{sec:Generators}.

For $B$ general, we obtain many different varieties $X=\p^1\times B$. It can also be that $X$ is rational even if $B$ is not (in \cite[Th\'eor\`eme~1]{BCSS} a non-rational variety $Y$ of dimension $3$ is given such that $Y\times \p^3$ is rational, so  $B=Y$ or $B=Y\times \p^1$ or $B=Y\times \p^2$ gives such an example), but then the conic bundle $X/B$ is not equivalent to the trivial one $\p^n\times \p^1/\p^n$.

\begin{lemma} \label{lem:PGL2toG}
Any surjective group homomorphism $\tau\colon \PGL_2(\C(B)) \tto G$ that is not an isomorphism factorises through the quotient 
\[
\PGL_2(\C(B))/\PSL_2(\C(B)) \simeq \C(B)^*/(\C(B)^*)^2,
\]
where the isomorphism corresponds to the determinant.
In particular the target group $G$ is abelian of exponent $2$. 
\end{lemma}

\begin{proof}
There exists a non trivial element $A \in \Ker \tau$ by assumption.
Since the group $\PGL_2(\C(B))$ has trivial centre, we can find $N \in \PGL_2(\C(B))$ that does not commute with $A$.
Then $\id\neq ANA^{-1}N^{-1} \in \PSL_2(\C(B)) \cap \Ker \tau$, and since $\PSL_2(\C(B))$ is a simple group we get $\PSL_2(\C(B)) \subseteq \Ker \tau$, which gives the result.
\end{proof}

Write $\div \colon \C(B)^*\to \Div(B)$  the classical group homomorphism that sends a rational function onto its divisor of poles and zeros, and whose image is the group of principal divisors on $B$.
Denoting by $\Pl_B$ the set of prime divisors on $B$, the group homomorphism $\div$ naturally gives a group homomorphism
\[\PGL_2(\C(B))/\PSL_2(\C(B)) \simeq\C(B)^*/(\C(B)^*)^2\to \bigoplus_{\Pl_B} \Z/2.\] 
We project onto the sum of prime divisors with large enough covering gonality and identify the ones which are equivalent up to a birational map of $B$.
This identification corresponds to taking orbits of the action of $\Aut_\C(\C(B))$ on $\C(B)$.
The following lemma shows that the resulting group homomorphism extends from $\Bir(X/B)$ to $\Bir(X)$, and coincides with the group homomorphism from Theorem~\ref{TheoremBirMori}. 

Observe that for each $A\in\PGL_2(\C(B))$, we can speak about the parity of the multiplicity of $\det(A)\in \C(B)^*/(\C(B)^*)^2$ as pole or zero
along an irreducible hypersurface $\Gamma\subset B$, as the multiplicity of an element of $(\C(B)^*)^2$ is always even.

\begin{lemma}\label{lem:ImagePGL2}
Let $B$ be a smooth variety of dimension at least $2$,  $X=\p^1\times B$, and let $\phi_M\in \Bir(X/B)\simeq \PGL_2(\C(B))$ be the birational map
\[\phi_M\colon ([u:v],t)\mapsrat ([a(t)u+b(t)v:c(t)u+d(t)v],t),\]
where
\[M=\begin{pmatrix} a(t) & b(t)\\ c(t) & d(t) \end{pmatrix} \in \PGL_2(\C(B)).\]
The image of $\phi_M$ under the group homomorphism
\[ \Bir(X/B)\to \bigoplus_{\MC(X/B)} \Z/2\]
of Theorem~\ref{TheoremBirMori}
is equal to the sum of the equivalence classes of marked conic bundles $(X/B,\Gamma)$ such that $\Gamma\subset B$ is a irreducible hypersurfaces of $B$ with $\covgon(\Gamma) > \max\{ d, 8\conngon(X) \}$ and such that the multiplicity of $\det(M)$ along $\Gamma$ is odd.  
\end{lemma}

\begin{proof}
We first observe that the image of $\PSL_2(\C(B))\subseteq \PGL_2(\C(B))\simeq \Bir(X/B)$ under the group homomorphism 
\[
\Bir(X/B)\to \bigoplus_{\MC(X/B)} \Z/2
\]
is trivial, since $\PSL_2(\C(B))$ is simple and not abelian. 
Hence, the image of an element  $\phi\in \Bir(X/B)\simeq\PGL_2(\C(B))$ is uniquely determined by its determinant $\delta\in \C(B)^*/(\C(B)^*)^2$ (Lemma~\ref{lem:PGL2toG}), and is the same as the image of the dilatation
\[\psi_\delta\colon ([u:v],t)\mapsto ([\delta(t)u:v],t).\] 
So we only need to prove the result for $M$ equal to such a dilatation.

We denote as before by $\Pl_B$ the set of prime divisors on $B$.
For $\delta\in \C(B)^*$ and $\Gamma\in \Pl_B$, we denote by $m_\delta(\Gamma)\in \Z$ the multiplicity of $\delta$ along $\Gamma$, so that 
\[\div(\delta)=\sum_{\Gamma\in \Pl_B} m_\delta(\Gamma)\, \Gamma.\] 
We also denote by $P_\delta(\Gamma)\in \{0,1\}$ the parity  of $\psi_\delta$ along $\Gamma$ as defined in Definition~\ref{def:ParityPhiGamma} and Lemma~$\ref{lem:SequenceParity}$.
The image of the dilatation $\psi_\delta\in \Bir(X/B)$ under the group homomorphism 
\[\Bir(X/B)\to \bigoplus_{\MC(X/B)} \Z/2\]
is equal to the sum of equivalence classes of marked conic bundles $(X/B,\Gamma)$ such that $\Gamma\subset B$ is an irreducible hypersurface with 
\[\covgon(\Gamma) > \max\{ d, 8\conngon(X) \}\]
and such that $P_\delta(\Gamma)$ is odd (Corollary~\ref{cor:Parity}). To prove the result, it suffices to show that $P_\delta(\Gamma)$ and $m_\delta(\Gamma)$ have the same parity.
For all $\delta,\delta'\in \C(B)^*$, we have
\[
m_\delta(\Gamma)+m_{\delta'}(\Gamma)=m_{\delta\cdot \delta'}(\Gamma)
\quad\text{ and }\quad
P_\delta(\Gamma)+P_{\delta'}(\Gamma) \equiv P_{\delta\cdot \delta'}(\Gamma)\pmod{2}.
\]
Indeed, the first equality follows from the definition of the multiplicity and the second follows from Lemma~$\ref{lem:SequenceParity}$, since $\psi_\delta\circ \psi_{\delta'}=\psi_{\delta\cdot \delta'}$. The local ring $\Ol_\Gamma(B)$ being a DVR, the group $\C(B)^*$ is generated by elements $\delta\in \C(B)^*$ with $m_\delta(\Gamma)=0$, and by one single element $\delta_0$ which satisfies $m_{\delta_0}(\Gamma)=1$. 
It therefore suffices to consider the case where $m_{\delta}(\Gamma)\in \{0,1\}$.

We take a general point $p\in \Gamma$, an irreducible curve $C\subseteq B$ transverse to $\Gamma$ at $p$, and compute the number  of base-points of the birational map $\theta\colon \p^1\times C \rat \p^1\times C$ given by $([u:v],t)\mapsto ([\delta(t)u:v],t)$ that are equal or infinitely near to a point of the fibre of $p$. 
If $m_{\delta}(\Gamma)=0$, then $\delta$ is well defined on $p$, so the birational map $\theta$ induces an isomorphism $\p^1 \times\{ p\} \to \p^1\times \{p\}$, which implies that $P_\delta(\Gamma)=0$. 
If $m_{\delta}(\Gamma)=1$, then $\delta$ has a zero of multiplicity one at $p$, so $\theta$ has exactly one base-point on $\p^1\times \{p\}$, namely $([1:0],p)$. 
The composition of $\theta$ with the blow-up of $Z\to \p^1 \times C$ of $([1:0],p)$ yields a birational map $Z\rat \p^1\times C$ with no more base-point on the exceptional divisor, as the multiplicity of both $\delta$ and $v/u$ at the point is $1$, so $P_\delta(\Gamma)=1$. 
This achieves the proof.
\end{proof}

We can now give the proof of Theorem~\ref{TheoremBirPnnotsimple}.
 
\begin{proof}[Proof of Theorem~\ref{TheoremBirPnnotsimple}]
We denote by $\Dil_\k$ the subgroup of birational dilatations 
\begin{align*}
\Dil_\k &= \{(x_1,\dots,x_n)\mapsrat \left(x_1\alpha(x_2,\cdots,x_n),x_2,\dots,x_n\right)\mid
\alpha\in\k(x_2,\dots,x_{n})^*\} \\
&\subseteq \Bir_\k(\p^n) \simeq \Aut_\k(\k(x_1,\dots,x_n)).
\end{align*}
We denote $B=\p^{n-1}$ and use the birational map (defined over $\k$) 
\[
\begin{tikzcd}[map]
 X=\p^1\times B & \rat & \p^n\\
( [u:1], [t_1,\dots,t_{n-1}:1]) & \mapsrat & {[1:u:t_1:\cdots:t_{n-1}]}
\end{tikzcd}
\]
that conjugates $\Bir(X)$ to $\Bir(\p^n)$, sending elements of the form 
\[\{([u:v],t)\mapsrat ([\alpha(t)u:v],t)\mid\alpha\in \C(B)^*\}\]
onto elements locally given by $(x_1,\dots,x_n)\mapsrat \left(x_1\alpha(x_2,\cdots,x_n),x_2,\dots,x_n\right)$.

Now we pick a large enough integer $D$ and consider the set $\Hl_D$ of degree $D$ irreducible hypersurfaces in $\p^{n-1}$. 
For each element $\Gamma\in \Hl_D$, we consider an irreducible polynomial $P \in \k[x_0,\dots,x_n]$ of degree $D$ defining the hypersurface $\Gamma$,  choose $\alpha=P/x_0^D \in \k(\p^{n-1})$ and associate to $\Gamma$ the element $\phi_{\alpha}\in \Bir(X/B)$ given by
\[\phi_{\alpha}\colon ([u:v],t)\mapsrat ([\alpha(t)u:v],t).\]
By Lemma~\ref{lem:ImagePGL2}, the image of $\phi_{\alpha}$ under the group homomorphism 
\[ \Bir(X/B)\to \bigoplus_{\MC(X/B)} \Z/2\]
of Theorem~\ref{TheoremBirMori} is the unique marked conic bundle $(X/B,\Gamma)$ (as the hypersurface $\Gamma_0\subset B$ given by $x_0=0$ satisfies $\covgon(\Gamma_0)=1$). 
It remains to observe that we have enough elements in $\Hl_D$, up to birational maps of $\p^{n-1}$, namely as much as in the field $\k$. 
Indeed, if we take two general hypersurfaces $\Gamma_1,\Gamma_2\subset \p^{n-1}$ of degree $\ge n+1$, then every birational map $\Gamma_1\rat \Gamma_2$ extends to a linear automorphism of $\p^{n-1}$; this can be shown by taking the suitable Veronese embedding of $\p^{n-1}$ such that the canonical divisors of $\Gamma_1$ and $\Gamma_2$ become hyperplane sections. 
The dimension of $\PGL_{n}(\k)$ being bounded, for a large enough degree $D$ we obtain a quotient of $\Hl_D$ by $\PGL_{n}(\k)$ which has positive dimension, hence which has the same cardinality as the ground field $\k$. 
This quotient can be taken as the indexing set $I$ in the statement of Theorem~\ref{TheoremBirPnnotsimple}. 
 \end{proof}

\begin{remark} \label{rem:thmAandB} \phantomsection
\begin{enumerate}
\item 
As all birational dilatations in Theorem~\ref{TheoremBirPnnotsimple} belong to the Jonqui\`eres subgroup of elements preserving a pencil of lines, the restriction of the group homomorphism $\Bir(\p^n)\to \bigoplus_I \Z/2$ to the Jonqui\`eres subgroup also is surjective.  
We will need other conic bundle structures on rational varieties to obtain Theorem~\ref{TheoremTame}.

\item
The proof of Theorem~\ref{TheoremBirPnnotsimple} uses Lemma~\ref{lem:ImagePGL2} in the case where $B=\p^{n-1}$. 
For a general basis $B$ we can prove along the same lines that the image of the subgroup of $\Bir(X/B)$ given by
\[\{([u:v],t)\mapsrat ([\delta(t)u:v],t)\mid\delta\in \C(B)^*\}\]
under the group homomorphism $\Bir(X/B)\to \bigoplus_{\MC(X/B)} \Z/2$
of Theorem~\ref{TheoremBirMori} is infinite.  
We omit the proof here, as it is similar to the case of $B=\p^{n-1}$, and moreover we will prove a more general result in Proposition~\ref{pro:largeimageAnyCB}.
\end{enumerate}
\end{remark}

\subsection{The case of non-trivial conic bundles and the proof of Theorem~\ref{TheoremBirXBnotsimple}}\label{sec:nontrivialBundle}

Recall that given a smooth conic $C \subset \p^2$ and a point $p \in \p^2 \setminus C$, there is an involution $\iota(p,C) \in \Bir(\p^2)$ that preserves the conic $C$. 
It is defined on each general line $L$ through $p$ as the involution that fixes $p$ and exchange the two intersection points $L \cap C$.
We say that $\iota(p,C)$ is the \emph{involution induced by the projection from $p$}.
We now use this construction in family to produce interesting involutions on some conic bundles.

\begin{lemma}\label{lem:CBInvolution}
Let $B$ be a smooth variety, $\hat{\eta}\colon P\to B$ a locally trivial $\p^2$-bundle, and $X\subset P$ a closed subvariety such that the restriction of $\hat{\eta}$ is a conic bundle $\eta\colon X\to B$. Let $s\colon B\rat P$ be a rational section $($i.e.~a rational map, birational to its image, such that $\hat{\eta}\circ s=\id_B)$, whose image is not contained in $X$. 
Let $\iota\in \Bir(X/B)$ be the birational involution whose restriction to a general fibre $\eta^{-1}(b)$ is the involution induced by the projection from $s(b)$.
Let $\Gamma\subset B$ be an irreducible hypersurface not contained in the discriminant locus of $\eta$, 
and let $F$ be a local equation of $X$ in $P$.

If the multiplicity of $F(s)$ along $\Gamma$ is equal to $0$ or $1$, 
then the parity of $\iota$ along $\Gamma$ $($in the sense of Definition~$\ref{def:ParityPhiGamma})$ is equal to this multiplicity $($modulo $2)$.
\end{lemma}

\begin{proof}
We choose a dense open subset $U \subseteq B$ that intersects $\Gamma$ and trivialises the  $\p^2$-bundle.
Inside $\p^2 \times U$, a local equation of $X$ is given by $F\in \C(B)[x,y,z]$, homogeneous of degree $2$ in $x,y,z$. 
The  fibre of $\eta \colon X\to B$ over a general point of $\Gamma$ (respectively of $B$) is a smooth conic. The section $s$ corresponds to $[\alpha:\beta:\gamma]$, where $\alpha,\beta,\gamma\in \C(B)$ are not all zero and are uniquely determined up to multiplication by an element of $\C(B)^*$.  
As $\Gamma$ is a hypersurface of $B$, the local ring $\Ol_\Gamma(B)$ is a DVR.
One can choose $\alpha,\beta,\gamma\in \Ol_\Gamma(B)$, not all vanishing on $\Gamma$: this defines $\alpha,\beta,\gamma$ uniquely, up to multiplication by an element of $\Ol_\Gamma(B)^*$. 
The evaluation $F(\alpha,\beta,\gamma)\in\C(B)$ at $s$ is then  uniquely determined by $s$ up to multiplication by the square of an element of $\Ol_\Gamma(B)^*$, so that the multiplicity of $F(\alpha,\beta,\gamma)$ along $\Gamma$ is well defined. 

The restriction of $\alpha,\beta,\gamma$ gives an element $(\bar{\alpha},\bar{\beta},\bar{\gamma})\in\C(\Gamma)^3\setminus \{0\}$. 
There exists a matrix in $\GL_3(\C(\Gamma))$  that sends $(\bar{\alpha},\bar{\beta},\bar{\gamma})$ to $(1,0,0)$.
By extending this matrix as an element of $\GL_3(\Ol_\Gamma(B))$, we can assume that $(\alpha,\beta,\gamma)=(1,0,0)$. 
We write the equation of $X$ as
\[F=ax^2+bxy+cxz+dy^2+eyz+fz^2\]
where $a,b,c,d,e,f\in \C(B)$ have no pole along $\Gamma$ and are not all simultaneously zero on $\Gamma$, and obtain that 
\(F(\alpha,\beta,\gamma)= F(1,0,0) = a.\) 
As $s$ is not contained in $X$, $a \in \C(B)$ is not identically zero.
With these coordinates, one checks that the involution $\iota \in \Bir(P/B)$ is given by the simple expression
\[\iota\colon [x:y:z]\mapsto [-(x+\tfrac{b}{a}y+\tfrac{c}{a}z):y:z].\]

Now we proceed to show that the parity of the multiplicity $r\in \{0,1\}$ of $F(s)=a$ along $\Gamma$ is equal to the parity of $\iota$ along $\Gamma$. For this, we take as in Definition~\ref{def:ParityPhiGamma} a general point $p\in \Gamma$ and an irreducible curve $C\subseteq B$ transverse to $\Gamma$ at $p$, and show that $r$ is the number of base-points of the birational surface map $\iota_C\colon \eta^{-1}(C)\rat \eta^{-1}(C)$ induced by $\iota$ that are equal or infinitely near to a point of the fibre of $p$.

If $r=0$, then $a$ does not vanish on $\Gamma$, hence the involution $\iota$ is a local isomorphism above a general point of $\Gamma$, so $\iota_C$ is an isomorphism on the fibre of $p$. This achieves the proof in this case. 

We now assume that $r=1$, or equivalently that $a$ is  a generator of the maximal ideal $\mathfrak{m}$ of $\Ol_\Gamma(B)$.
It implies that either $b$ or $c$ is not zero on $\Gamma$, otherwise $\Gamma$ would be contained in the discriminant locus of $\eta$.
As $\Gamma$ is an irreducible hypersurface of $B$, the local ring $\Ol_{B,\Gamma}$ of rational functions of $B$ defined on an open subset of $\Gamma$ is a DVR. 
We write $\iota_C\in \Bir(\eta^{-1}(C))$ as the restriction of $\nu^{-1}\circ\theta\circ \nu$, where $\theta,\nu\in \Aut_{\C(C)}(\p^2)\subset \Bir(C\times \p^2)$ are the birational maps
\[
\nu\colon[x:y:z]\mapsto [a x: y: z]\;\text{ and }\theta\colon[x:y:z]\mapsto [-( x+by+cz):y:z]\;.
\]
We denote $S = \nu (\eta^{-1}(C)) \subset \p^2\times C$, and one checks that $S$ is the surface with equation 
\[ x^2+b xy+c  xz+a (dy^2+eyz+fz^2)=0.\]

The fibre $\eta^{-1}(p) \subset \eta^{-1}(C)$ is a smooth conic. 
On the other hand the fibre of $p$ in $S$ is $\ell\cup \ell'$, where $\ell$ and $\ell'$ are the lines given by $x=0$ and $ x+by+cz=0$ respectively.
Observe that $\ell \neq \ell'$ since  $(b(p),c(p))\not=(0,0)$.

Since $S$ is the image of $\eta^{-1}(C)$ by $\nu$, the map $\theta$ induces a birational involution $\theta_S \colon S\rat S$.
The map $\theta_S$ is a local isomorphism in a neighborhood of the fibre of $p$, which exchanges $\ell$ and $\ell'$. 
Moreover, $\nu$ maps the conic $\eta^{-1}(p)$ to the line $\ell$, $\nu$ is not defined at the point $q = [1:0:0]$, and $\nu^{-1}\colon[x:y:z]\mapsto [ x: ay: az]$ contracts $\ell'$ on $q$. 
As $a$ has multiplicity $1$, $\nu$ is simply the blow-up of $q$, so the birational map $\iota_C\colon \eta^{-1}(C)\rat \eta^{-1}(C)$ is given, in a neighbourhood of $\eta^{-1}(p)$, by the blow-up of $q$ followed by the contraction of the strict transform of $\eta^{-1}(p)$. 
So the parity of $\iota_C$ along $\Gamma$ is $1$, as desired. 
\end{proof}

\begin{definition}\label{def:DecomposableCB}
We say that a conic bundle $X/B$ is a \emph{decomposable conic bundle} if $X$ and $B$ are smooth, and if we have closed embeddings $B\hookrightarrow \p^m$ and $X\hookrightarrow P$ where $P$ is a $\p^2$-bundle over $\p^m$ that is the projection of a decomposable vector bundle of rank $3$, i.e.~$P=\p(\mathcal{O}_{\p^m}\oplus \mathcal{O}_{\p^m}(a)\oplus \mathcal{O}_{\p^m}(b))$ for some $a,b\in \Z$.
We moreover ask that the morphism $X/B$ comes from the restriction of the $\p^2$-bundle $P\to \p^m$ and that $X\subset P$ is locally given by equations of degree $2$ in the $\p^2$-bundle.
\end{definition}

\begin{proposition}\label{pro:largeimageAnyCB}
For each decomposable conic bundle $\eta\colon X\to B$ with $\dim B \ge 2$, 
there are infinitely many involutions in $\Bir(X/B)$ which have distinct images via the group homomorphism $ \Bir(X/B)\to \bigoplus_{\MC(X/B)} \Z/2$ of Theorem~\ref{TheoremBirMori}. In particular, the image is infinite.
\end{proposition}

\begin{proof}
We can see $B$ as a closed subset $B\subseteq \p^m$, and obtain that $X\subset P$, where $\hat{\eta}\colon P\to \p^m$ is the projectivisation of a rank $3$ vector bundle. We can thus write $P=\p(\mathcal{O}_{\p^m}\oplus \mathcal{O}_{\p^m}(a)\oplus \mathcal{O}_{\p^m}(b))$ for some $a,b\ge 0$ (up to twisting and exchanging the factors). 
We view $P$ as the quotient of $(\A^3\setminus \{0\})\times (\A^{m+1}\setminus \{0\})$ by $(\mathbb{G}_m)^2$ via 
\[\left((\lambda,\mu),(x_0,x_1,x_2,y_0,\ldots,y_m)\right)\mapsto (\lambda x_0,\lambda\mu^{-a} x_1,\lambda\mu^{-b} x_2,\mu y_0,\ldots,\mu y_m)\]
and denote by $[x_0\colon x_1 \colon x_2 \semicolon y_0 \colon \ldots \colon y_m]$ the class of $(x_0,x_1,x_2,y_0,\ldots,y_m)$ (see \cite[Definition 2.3, Remark 2.4]{AhmOka} for more details).

Then $X$ is equal to the preimage of $B$ cut by the zero locus of an irreducible polynomial $G\in \C[x_0,x_1,x_2,y_0,\ldots,y_m]$, that has degree $2$ in $x_0,x_1,x_2$ (and suitable degree in $y_0,\ldots,y_m$ so that the polynomial is homogeneous for the above action). For each integer $d\ge 1$, and for general homogeneous polynomials 
\begin{align*}
u_0,v_0\in \C[y_0,\ldots,y_m]_d,&&
u_1,v_1\in \C[y_0,\ldots,y_m]_{d+a},&&
u_2,v_2\in \C[y_0,\ldots,y_m]_{d+b}, 
\end{align*} 
(the subscript corresponding to the degree), the closed subvariety $\hat\Gamma\subset X$ of codimension $2$ given by
\[
\hat\Gamma=\Bigl\{([x_0:x_1:x_2 \semicolon y_0:\ldots:y_m])\in X\subseteq P \Bigm| \sum_{i=0}^2 x_i u_i=\sum_{i=0}^2 x_i v_i=0\Bigr\}
\]
is smooth, by Bertini theorem.

We now prove that the projection $X\to B$ induces a birational morphism from $\hat\Gamma$ to its image $\Gamma\subset B$, an irreducible hypersurface of $B$. Solving the linear system $\sum_{i=0}^2 x_i u_i=\sum_{i=0}^2 x_i v_i=0$ in $x_0,x_1,x_2$, we obtain that the preimage of $[y_0:\ldots:y_m]$ is 
$[u_1v_2-u_2v_1:u_2v_0-u_0v_2:u_0v_1-u_1v_0 \semicolon y_0:\ldots:y_m],$
so the projection induces a birational morphism from $\hat\Gamma$ to the hypersurface $\Gamma\subset B$ given by the polynomial $G(P_0,P_1,P_2,y_0,\ldots,y_m)$, where $P_0,P_1,P_2\in \C[y_0,\dots,y_m]$ are the polynomials $P_0=u_1v_2-u_2v_1$, $P_1=u_2v_0-u_0v_2$ and $P_2=u_0v_1-u_1v_0$.

We now show that the covering gonality $\covgon(\hat\Gamma)=\covgon(\Gamma)$  is large if $d$ is large enough.
We denote by $H_i,F_j\subset P$ the hypersurfaces given respectively by $x_i=0$ and $y_j=0$, and obtain that $\Pic(P)=\Z H_i\bigoplus \Z F_j$ for all $i\in \{0,1,2\}, j\in \{0,\ldots,m\}$. The class of all $F_j$ is the same and denoted by $F$ and $H_0\sim H_1+aF\sim H_2+bF$.
Note that $\hat\Gamma$ is a complete intersection in $\hat{\eta}^{-1}(B)\subseteq P$ of 3 hypersurfaces equivalent to $H_0+dF, H_0+dF, 2H_0+d_0F$ for some $d_0\in \Z$ (depending on the equation of $X$). The canonical divisor of $P$ being equivalent to $-H_0-H_1-H_2-F_0-\ldots-F_m=-3H_0-(m+1-a-b)F$, we obtain by adjunction that the canonical divisor of $\hat\Gamma$ is the restriction to $\hat \Gamma$ of a divisor of $P$ equivalent to  $H_0+(2d+d_0-m-1+a+b)F$.
 The morphism associated to $F$ is simply the projection $\hat\Gamma\to \p^m$, which is birational onto its image. By Lemma~\ref{lem:BVApEasy}\ref{BVAmor}-\ref{BVAPn}, the divisor $pF$ satisfies $\BVA_p$, for each integer $p\ge 0$, and thus $K_{\hat\Gamma}$ satisfies $\BVA_p$ for  $p=2d+d_0-m-1+a+b\ge 0$ if $d$ is large enough, by Lemma~\ref{lem:BVApEasy}\ref{BVAeff}.  
This implies that $\covgon(\hat\Gamma)\ge p+2$ by Theorem~\ref{thm:BVAK}. 
By choosing $d$ large enough we obtain that $\covgon(\Gamma)=\covgon(\hat\Gamma)$ is large.

We now use the construction in Lemma~\ref{lem:CBInvolution} of the involution $\iota\in \Bir(X/B)$ associated with the $\p^2$-bundle $P/B$ and the rational section $s\colon B\rat P$ 
given by 
\[[y_0:\ldots:y_m] \mapsrat [u_1v_2-u_2v_1:u_2v_0-u_0v_2:u_0v_1-u_1v_0 \semicolon y_0:\ldots:y_m].\]

By Lemma~\ref{lem:CBInvolution}, the parity of $\iota$ along $\Gamma$ is one and the parity of $\iota$ along any other irreducible hypersurface of $B$ is zero (as $\Gamma$ is the zero locus of $G(s)$ by construction). For a large integer $d$, the image of $\iota$ under the group homomorphism 
\[ \Bir(X/B)\to \bigoplus_{\MC(X/B)} \Z/2\]
of Theorem~\ref{TheoremBirMori} is the equivalence class of $(X/B,\Gamma)$. 
Taking larger and larger $d$, we obtain infinitely many involutions in the image of the group homomorphisms, which are distinct and thus generate  a group isomorphic to an infinite direct sum of $\Z/2$, as the covering gonality of the hypersurfaces goes to infinity with $d$. 
\end{proof}

\begin{proof}[Proof of Theorem~\ref{TheoremBirXBnotsimple}]
We use the group homomorphism 
\[\Bir(X) \to \bigast_{C\in \CB(X)} \left(\bigoplus_{\MC(C)}\Z/2\right)\]
of Theorem~\ref{TheoremBirMori}. By assumption, $X/B$ is a decomposable conic bundle (in the sense of Definition~\ref{def:DecomposableCB}). By Proposition~\ref{pro:largeimageAnyCB}, the image of $\Bir(X/B)$ contains a group isomorphic to an infinite direct sum of $\Z/2$. 

To finish the proof of Theorem~\ref{TheoremBirXBnotsimple}, we take a subfield $\k\subseteq \C$ over which $X,B$ and $\eta$ are defined, and check that the involutions in $\Bir(X/B)$ that are used to provide the large image are defined over $\k$.  Firstly, the involutions provided by Lemma~\ref{lem:CBInvolution} are defined over $\k$ as soon as the rational section $s\colon B\rat P$ is. 
Secondly, the construction of Proposition~\ref{pro:largeimageAnyCB} works for general polynomials in $\C[y_0,\ldots,y_m]$ of some fixed degrees.

Since a dense open subset of an affine space $\A^n_\C$ contains infinitely many $\k$-points for each subfield $\k\subseteq \C$ (follows from the fact that the $\Q$-points of $\A^n$ are dense), we can assume that the polynomials, and thus the section, are defined over $\k$.
\end{proof}

\section{Non-equivalent conic bundles}
\label{sec:freeProduct}

In this section, we construct infinitely many non-equivalent conic bundles on $\p^n$, showing that the set $\CB(\p^n)$ is infinite for $n\ge 3$ (by contrast, observe that $\CB(\p^2)$ consists of one element). 
This allows us to prove Theorems~\ref{TheoremManyCB} and \ref{TheoremTame}.

\subsection{Studying the discriminant locus}

The main result of this section is Proposition~\ref{pro:Isogenies}. 
We prove in particular that if two standard conic bundles (defined in Definition~\ref{Defi:StandardCB}) $X_1/\p^2$ and $X_2/\p^2$ with discriminants $\Delta_1$ and $\Delta_2$ such that the conic bundles $(X_1\times \p^n)/(\p^2\times \p^n)$ and  $(X_1\times \p^n)/(\p^2\times \p^n)$ are equivalent, then there exist surjective morphisms $\Delta_1\tto \Delta_2$ and $\Delta_2\tto \Delta_1$.
The standard conic bundles are classical in the literature and are conic bundles having nice properties.
They can be in particular embedded in a $\p^2$-bundle, as it was the case for the decomposable conic bundles (Definition~\ref{def:DecomposableCB}). This notion is defined below as \emph{embedded conic fibration}. See remark~\ref{Rem:DifferentCBdef} for a comparison of the different notions.

The following notion is called an \emph{embedded conic} in \cite[page~358]{sarkisov}.

\begin{definition}\label{def:conicfibration}
Let $V$ be a smooth quasi-projective variety. An \emph{embedded conic fibration} is a projective morphism $\eta \colon X\to V$ that is the restriction of a locally trivial $\p^2$-bundle $\hat{\eta}\colon P\to V$, and such that $X\subset P$ is a closed subvariety, given locally by an equation of degree $2$.
Precisely, for each $p\in V$, there exists an affine open subset $U\subseteq V$ containing $p$ such that $\hat{\eta}^{-1}(U)$ is isomorphic to $U\times \p^2$, and the image of $\eta^{-1}(U) \subset U\times \p^2$ is a closed subvariety, irreducible over $\C(U)$, and defined by a  polynomial $F\in \C[U][x,y,z]$ homogeneous of degree $2$ in the coordinates $x,y,z$.
\end{definition}

\begin{remark}\label{rem:FlatisConicFib}
Let $\eta\colon X\to V$ be a flat projective morphism between smooth quasi-projective varieties, with generic fibre an irreducible conic. 
Then, $\eta$ is an embedded conic fibration in a natural way. This is done by taking the locally trivial $\p^2$-bundle $P=\p(\eta_*(\omega_X^{-1}))$ over $V$, where $\omega_X$ is the canonical line bundle of $X$ (see \cite[\S 1.5]{sarkisov}).
If $\eta$ is not flat, this is false, as some fibres can for instance have dimension $\ge 3$ even if $X/V$ is a Mori fibre space and thus a conic bundle (see \cite[Example~5]{AlzatiRusso}).
\end{remark}

The following definition is equivalent to the one of \cite[Definition 1.4]{sarkisov}. 

\begin{definition}\label{Defi:StandardCB}
A \emph{standard conic bundle} is a morphism $\eta\colon X\to B$ which is a conic bundle $($in the sense of Definition~$\ref{def:conicBundle})$, and which is moreover \emph{flat} with $X$ and $B$ smooth. 
This implies that $\eta$ is also an \emph{embedded conic fibration} in the $\p^2$-bundle $\p(\eta_*(\omega_X^{-1}))\to B$ (see Remark~\ref{rem:FlatisConicFib}).
\end{definition}

\begin{remark}\label{Rem:DifferentCBdef}Let us make some comparisons between the above definitions.

An \emph{embedded conic fibration} (Definition~\ref{def:conicfibration}) over a projective base is not necessarily a \emph{conic bundle} (Definition~\ref{def:conicBundle}), as the relative Picard rank can be $>1$.
Conversely, a conic bundle $X/B$ is not necessarily an embedded conic fibration (for instance when some fibres have dimension $\ge 3$), but it is one if the conic bundle is \emph{standard} (Definition~\ref{Defi:StandardCB}) (as explained just above) or \emph{decomposable} (Definition~\ref{def:DecomposableCB}).

Moreover, a decomposable conic bundle is not always standard, as some fibres can be equal to $\p^2$. It is not clear to us if there exist standard conic bundles which are not decomposable.
\end{remark}

\begin{definition}\label{def:MultDiscriminant}
Let $V$ be a smooth quasi-projective variety and $\eta\colon X\to V$ a flat embedded conic fibration.

For each irreducible closed subset  $\Gamma\subseteq V$, we define the \emph{multiplicity of the discriminant of $\eta$ along $\Gamma$} as follows. 
We take an open subset $U\subseteq V$ that intersects $\Gamma$ and such that $\eta^{-1}(U)$ is a closed subset of $U\times \p^2$, of degree $2$, and consider a symmetric matrix $M\in \Mat_{3\times 3}(\C(U))$ that defines the equation of $\eta^{-1}(U)$. 
We choose $M$ such that all coefficients of $M$ are contained in the local ring $\mathcal{O}_\Gamma(U)\subset \C(U)$ of rational functions defined on a general point of $\Gamma$, and such that the residue matrix $\widebar{M}\in \Mat_{3\times 3}(\C(\Gamma))$ is not zero. 
This is possible as the morphism is flat, and defines $M$ uniquely, up to multiplication by an invertible element of  $\mathcal{O}_\Gamma(U)$. 

Now we define the multiplicity of the discriminant of $\eta$ along $\Gamma$ to be the least integer $m\ge 0$ such that the determinant lies in $\mathfrak{m}_\Gamma(U)^m$, where $\mathfrak{m}_\Gamma(U)$ is the maximal ideal of $\mathcal{O}_\Gamma(U)$, kernel of the ring homomorphism $\mathcal{O}_\Gamma(U)\tto \C(\Gamma)$.

We define the \emph{discriminant divisor of $\eta$} to be $\sum m_D D$, where the sum runs over all irreducible hypersurfaces $D\subset V$ and where $m_D\in \N$ is the the multiplicity of the discriminant of $\eta$ along $D$ as defined above.
\end{definition}

\begin{remark}
If $\eta\colon X\to V$ is moreover a conic bundle, the definition of the discriminant given in Definition~\ref{def:MultDiscriminant} is compatible with the definition of discriminant locus given in Definition~\ref{def:conicBundle}: the discriminant locus is the reduced part of the discriminant divisor of $\eta$. 
Moreover, if $\eta$ is a standard conic bundle, the discriminant divisor is reduced \cite[Corollary 1.9]{sarkisov}. 
The multiplicity of the discriminant divisor along irreducible hypersurfaces of $V$ is always $0$ or $1$ in this case. 
We will however not only consider hypersurfaces but also closed subvarieties of lower dimension.
\end{remark}

Using the local description of the matrix that defines $\eta$ as a flat embedded conic fibration, one can prove the following:

\begin{proposition}[{\cite[Proposition 1.8]{sarkisov}}] \label{pro:CBDivisorSingFibres} 
Let $V$ be a smooth quasi-projective variety, let $\eta\colon X\to V$ be a flat embedded conic fibration, such that $X$ is smooth. 
The discriminant divisor $\Delta$ of $\eta$ has the following properties: for each point $p\in V$, the fibre $f_p=\eta^{-1}(p)$ is given as follows:
\[
f_p  \text{ is }
\begin{cases}
\text{a smooth conic}\\ \text{the union of two distinct lines}\\ \text{a double line}
\end{cases}
\iff 
p\text{ is }
\begin{cases}
\text{not on }\Delta\\ \text{a smooth point of }\Delta \\ \text{a singular point of }\Delta.
\end{cases}
\]
\end{proposition}

We shall need the following folklore result:

\begin{lemma}\label{lem:FlatConicFibDisc}
Let $V$ be a smooth quasi-projective variety and let $\eta_1\colon X_1\to V$ and $\eta_2\colon X_2\to V$ be two flat embedded conic fibrations. Let $\psi\colon X_1\rat X_2$ be a birational map over $V$. Let $\Delta\subseteq V$ be a closed irreducible subvariety such that $\eta_1^{-1}(\Delta$ is not contained in the base-locus of $\psi$, that the preimage $\eta_2^{-1}(\Delta)$ is irreducible and that a general fibre of $\eta_2^{-1}(\Delta)\to \Delta$ is the union of two distinct lines.
We moreover assume that the multiplicity of the discriminant of $\eta_2$ along $\Delta$ is $1$.
Then, one of the following holds:
\begin{enumerate}
\item\label{GammaDiscNonRed}
Every fibre of $\eta_1^{-1}(\Delta)\to \Delta$ is a double line $($non-reduced fibre$)$.
\item\label{GammaDiscIrrpi}
The preimage $\eta_1^{-1}(\Delta)$ is irreducible and a general fibre of $\eta_1^{-1}(\Delta)\to \Delta$ is the union of two distinct lines.
\end{enumerate}
\end{lemma}

\begin{proof}
Replacing $V$ by an open subset that intersects $\Delta$, we can assume that $X_1$ and $X_2$ are closed subvarieties of $V\times \p^2$ given by a polynomial of degree $2$ in the coordinates of $\p^2$.
We denote by $\mathcal{O}_\Delta(V)\subset \C(V)$ the subring of rational functions that are defined on a general point of $\Delta$ and consider the surjective residue homomorphism $\mathcal{O}_\Delta(V)\tto \C(\Delta)$.
The quadratic equations of $X_1$ and $X_2$ correspond to symmetric matrices $M_1,M_2\in \Mat_{3\times 3}(\C(V))$, defined up to scalar multiplication.  
Since both $\eta_1$ and $\eta_2$ are flat, we can choose $M_1,M_2\in \Mat_{3\times 3}(\mathcal{O}_\Delta(V))$ such that the residue matrices $\widebar{M}_1,\widebar{M}_2\in \Mat_{3\times 3}(\C(\Delta))$ are not zero.

The fact that $\eta_2^{-1}(\Delta)$ is irreducible  and that a general fibre of $\eta_2^{-1}(\Delta)\to \Delta$ is the union of two distinct lines is equivalent to asking that the quadratic form associated to $M_2$
corresponds to a singular irreducible conic over the field $\C(\Delta)$. It then corresponds to the union of two lines defined over an extension of degree $2$ of $\C(\Delta)$, which intersect into a point defined over $\C(\Delta)$.
After a change of coordinates on $X_2\subset V\times \p^2$, applying an element of $\PGL_3(\C(V))$ which restricts to an element of $\PGL_3(\C(\Delta))$, we can assume that the point is $[0:0:1]$ and completing the square we assume that the restriction is given by $F=ax^2+by^2$ for some $a,b\in \C(\Delta)^*$, where $-\frac{a}{b}\in \C(\Delta)^*$ is not a square.
This corresponds to saying that $\widebar M_2$ is equal to the diagonal matrix $\diag(a,b,0)$.

The birational map $\psi$ is given by 
\[\left(v, 
\left[\begin{smallmatrix} x\\ y \\ z \end{smallmatrix}\right]
\right) \mapsrat  \left(v, A(v)\cdot
\left[\begin{smallmatrix} x\\ y \\ z \end{smallmatrix}\right] \right)
\] for some $A\in \GL_3(\C(V))$. 
This implies that 
$M_1$ and $\tr{A}\cdot M_2\cdot A$ are collinear in $\Mat_{3\times 3}(\C(V))$.
 
As $\eta_1^{-1}(\Delta)$ is not contained in the base-locus of $\psi$, we can assume that $A\in \Mat_{3\times 3}(\mathcal{O}_\Delta(V))$ is such that its residue $\widebar{A}\in \Mat_{3\times 3}(\C(\Delta))$ is not zero.
We can moreover choose an element $S\in \GL_{3}(\mathcal{O}_\Delta(V))$, with residue $\widebar{S}\in  \GL_{3}(\C(\Delta))$, and replace $A$ with $A\cdot S$.
This corresponds to a coordinate change of $\p^2\times V$ at the source, which only affects $X_1$ and not $X_2$. We can then reduce to the following possibilities for $\widebar{A}$, according to the rank of the $2\times 3$ matrix obtained from the first two rows of $\widebar{A}$:
\begin{align*}
\begin{pmatrix} 1 & 0 & 0 \\ 0 & 1 & 0 \\  \mu_1 & \mu_2 & \mu_3 \end{pmatrix}, &&  
\begin{pmatrix} \alpha & 0 & 0 \\ \beta & 0 & 0 \\  \mu_1 & \mu_2 & \mu_3 \end{pmatrix}, &&
\begin{pmatrix} 0 & 0 & 0 \\ 0 & 0 & 0 \\  \mu_1 & \mu_2 & \mu_3 \end{pmatrix},
\end{align*}
 where $\alpha,\beta,\mu_1,\mu_2,\mu_3\in \C(\Delta)$ and $(\alpha,\beta)\not=(0,0)$.
 
In the first case, $\tr{\widebar{A}}\cdot \widebar{M_2}\cdot \widebar{A}=\widebar{M_2}$, so $\eta_2^{-1}(\Delta)$ has the same properties as $\eta_1^{-1}(\Delta)$, which gives \ref{GammaDiscIrrpi}.

The second case gives $\tr{\widebar{A}}\cdot \widebar{M_2}\cdot \widebar{A}= \diag(\alpha^2 a+\beta^2 b,0,0)$.
As $(\alpha,\beta)\not=(0,0)$ and $-\frac{a}{b}\in \C(\Delta)^*$ is not a square, we have $\alpha^2 a+\beta^2b\not=0$.
The quadratic form associated to this matrix is then a double line, and we obtain \ref{GammaDiscNonRed}. 

It remains to study the last case, which yields $\tr{\widebar{A}}\cdot \widebar{M_2}\cdot \widebar{A}=0$. This means that all coefficents of $\tr{A}\cdot M_2\cdot A$ belong to the maximal ideal $\mathfrak{m}=\mathfrak{m}_\Delta(V)$ of $\mathcal{O}_\Delta(V)$, kernel of the residue homomorphism $\mathcal{O}_\Delta(V)\tto \C(\Delta)$.
Applying $S$ as before, we can assume that $\mu_1=1$, $\mu_2=\mu_3=0$, since the rank of $\widebar{A}$ is $1$.
We write
$M_2 = \diag(a,b,0) + (\nu_{i,j})_{1 \le i,j \le 3}$
where $\nu_{i,j}\in \mathfrak{m}$ for all $i,j$, and obtain $\det(M_2)\equiv a\cdot b\cdot \nu_{3,3}\pmod{\mathfrak{m}^2}$.
As the multiplicity of the discriminant of $\eta_2$ along $\Delta$ is $1$, this implies that $\nu_{3,3}\in \mathfrak{m}\setminus \mathfrak{m}^2$.
We compute 
$\tr{A}\cdot M_2\cdot A\equiv \diag(\nu_{3,3},0,0) \pmod{\mathfrak{m}^2}$.
The quadratic form associated to this matrix is a double line, so again we obtain \ref{GammaDiscNonRed}.
\end{proof}

We give two examples to illustrate the need for all the assumptions in Lemma~\ref{lem:FlatConicFibDisc}:

\begin{example}
We work over the affine plane $V=\A^2$ and consider 
\begin{align*}
X_1&= \{([x:y:z],(u,v))\in \p^2\times \A^2 \mid x^2v+y^2-z^2=0\},\\
X_2&= \{([x:y:z],(u,v))\in \p^2\times \A^2 \mid x^2v+y^2-u^2z^2=0\},\\
X_2'&= \{([x:y:z],(u,v))\in \p^2\times \A^2 \mid x^2uv+y^2-z^2=0\}.
\end{align*} 
The projection onto the second factor gives three flat embedded conic fibrations $\eta_1\colon X_1\to \A^2$,  $\eta_2\colon X_2\to \A^2$,  $\eta_2'\colon X_2'\to \A^2$, with discriminant divisors being respectively given by $v=0$, $u^2v=0$ and $uv=0$.    
The birational maps of $\p^2\times \A^2$ given by
$([x:y:z],(u,v))\mapsto ([xu:yu:z],(u,v))$ and $([x:y:z],(u,v))\mapsto ([2x:(u+1)y+(u-1)z:(u-1)y+(u+1)z],(u,v))$ provide two birational maps 
 $\psi\colon X_1\rat X_2$ and $\psi'\colon X_1\rat X_2'$ over $\A^2$.

Choosing $\Delta\subset \A^2$ to be the line $\{u=0\}$, the result of Lemma~\ref{lem:FlatConicFibDisc} does not hold for $\psi$ and for $\psi'$, because a general fibre of  $\eta_1^{-1}(\Delta)\to \Delta$ is a smooth conic. In both cases, exactly one hypothesis is not satisfied.
Namely, the multiplicity of the discriminant of $\eta_2$ along $\Delta$ is $2$ instead of $1$, and the surface $\eta_2'^{-1}(\Delta)$ is not irreducible.
\end{example}

The idea of the proof of the following statement was given to us by C. B\"ohnig and H.-C. Graf von Bothmer. 

\begin{proposition}\label{pro:Isogenies}
Let $B$ be a smooth surface, and for $i=1,2$, let $\eta_i\colon X_i\to B$ be a standard conic bundle with discriminant a smooth irreducible curve $\Delta_i\subset B$. 
Assume that there exists a commutative diagram
\[\begin{tikzcd}[link]
X_1\times Y \ar[dd,"\eta_1\times \id",swap]\ar[r,"\psi",dashed]& X_2\times Y\ar[dd,"\eta_2\times \id"]  \\ \\
B\times Y \ar[r,"\theta",dashed]& B\times Y
\end{tikzcd}\]
where $Y$ is smooth and $\psi$, $\theta$ are birational. 

Then, for a general point $p\in Y$, the image of $\Delta_1\times \{p\}$ is contained in $\Delta_2\times Y$ and the morphism $\Delta_1\to \Delta_2$ obtained by composing
\[\Delta_1 \iso  \Delta_1 \times \{p\}  \stackrel{\theta}{\rat} \Delta_2 \times Y\stackrel{\pr_1}{\longto} \Delta_2\]
is surjective $($here $\pr_1\colon \Delta_2\times Y\to \Delta_2$ is the first projection$)$.
\end{proposition}

\begin{proof}
For $i=1,2$, the discriminant divisor of $\eta_i$ is reduced \cite[Corollary 1.9]{sarkisov}, so consists of $\Delta_i$. As $\Delta_i$ is smooth, $\eta_i^{-1}(p)$ is the union of two distinct lines for each $p\in \Delta_i$ (Proposition~\ref{pro:CBDivisorSingFibres}). Since $\rho(X_i/B_i)=1$, the preimage $\eta_i^{-1}(\Delta_i)$ is irreducible.
The morphism $(X_i\times Y)/(B\times Y)$ is a standard conic bundle whose discriminant divisor is reduced, consisting of the smooth hypersurface $\Delta_i\times Y\subset B\times Y$. 

 We choose a dense open subset $U\subseteq B\times Y$ on which $\theta$ is defined and whose complement is of codimension $2$ (since $B\times Y$ is smooth). In particular, $U\cap (\Delta_1\times Y)$ is not empty, so $U\cap (\Delta_1\times \{p\})\not=\emptyset$ for a general point $p\in Y$. After restricting the open subset, we can moreover assume that $\eta_1^{-1}(U)$ is a closed subset of $U\times\p^2$, given by the quadratic form induced by a matrix $M_1\in \GL_3(\C(U))$. The coefficients of the matrix can moreover be chosen in $\C(B)\subseteq \C(B\times Y)=\C(U)$, as the equation of $X_1\times Y$ in $\p^2\times Y$ is locally the equation of $X_1$ in $\p^2$, independent of $Y$.

We define $C\subset B\times Y$ to be image of $\Delta_1\times \{p\}$ by $\theta$, which is a point or an irreducible curve, as $\Delta_1$ is an irreducible  curve. 
The aim is now to show that $C\subseteq \Delta_2\times Y$ and that $\pr_1(C)=\Delta_2$. 
We choose an open subset $V\subseteq B\times Y$ intersecting $C$ such that $\eta_2^{-1}(V)$ is contained in $ \p^2\times V$ and is given by the quadratic form given by a symmetric matrix $M_2\in \Mat_{3\times 3}(\C(V))$. The morphism $\eta_2$ being flat, we can choose the coefficients of $M_2$ to be defined on $C$ and such that the residue matrix in $\widebar{M_2}\in \Mat_{3\times 3}(\C(C))$ is not zero. The birational map $\psi$ is locally given by 
\[
\begin{tikzcd}[map]
U \times \p^2  & \rat & V \times \p^2\\
\left(u, \left[\begin{smallmatrix} x\\ y \\ z \end{smallmatrix}\right]\right)& \mapsrat & \left(\theta(u), A(u) \cdot \left[\begin{smallmatrix} x\\ y \\ z \end{smallmatrix}\right]\right)
\end{tikzcd}
\]
for some $A\in \GL_3(\C(U))$.  The explicit form of the map $\psi$ gives
\[\lambda\cdot M_1= \tr{A}\cdot \theta^{*}(M_2)\cdot A\]
where $\lambda\in \C(U)^*$ is a scalar and $\theta^{*}(M_2)$ is the matrix obtained from $M_2$ by applying to its coefficients the field isomorphism $\theta^{*}\colon \C(V)\to \C(U)$. As the rational map $\theta$ induces a dominant rational map $\Delta_1\times \{p\}\rat C$, we have a field homomorphism $\C(C)\to \C(\Delta_1\times \{p\})\simeq \C(\Delta_1)$, that we denote by $\bar\theta^*$. It induces a commutative diagram
\[\begin{tikzcd}[link]
\mathcal{O}_{C}(V) \ar[dd,swap]\ar[r,"\theta^*"]& \mathcal{O}_{\Delta_1\times \{p\}}(U) \ar[dd]  \\ \\
\C(C)  \ar[r,"\bar\theta^*"]& \C(\Delta_1\times \{p\})  \ar[r,phantom,"\simeq"]&[-4em] \C(\Delta_1).
\end{tikzcd}\]

We denote by $X'\subset U\times \p^2$ the subvariety given by the quadratic form associated to the matrix $\theta^{*}(M_2)$. We observe that the coefficients of $\theta^{*}(M_2)$ are defined over $\Delta_1\times \{p\}$ and that the residue gives a matrix $\widebar{\theta^{*}(M_2)}\in \Mat_{3\times 3}(\C(\Delta_1))$ which is obtained by applying the field homomorphism $\widebar{\theta^*}$ to the entries of
$\widebar{M_2}\in \Mat_{3\times 3}(C)$.
The morphism  
$\pr_1\colon X'\to U$ is then an embedded conic fibration, which is flat after maybe reducing the open subset $U$ (but still having $U\cap (\Delta_1\times \{p\})\not=\emptyset$). 

We can apply Lemma~\ref{lem:FlatConicFibDisc} to the birational map $X'\rat X$ given by \[
\left(u, \left[\begin{smallmatrix} x\\ y \\ z \end{smallmatrix}\right]\right) \mapsrat  \left(u, A(u)^{-1} \cdot \left[\begin{smallmatrix} x\\ y \\ z \end{smallmatrix}\right]\right)\]
and to $\Delta=\Delta_1\times \{p\}$.
Indeed, $(\eta_1 \times \id)^{-1}(\Delta_1\times \{p\})$ is irreducible as $\eta_1^{-1}(\Delta_1)$ is irreducible, and every fibre of $(\eta_1 \times \id)^{-1}(\Delta_1\times \{p\})\to \Delta_1\times \{p\}$  is the union of two distinct lines, as the same holds for $\eta_1^{-1}(\Delta_1)\to \Delta_1$ by  Proposition~\ref{pro:CBDivisorSingFibres}. Lemma~\ref{lem:FlatConicFibDisc}
gives two possibilities for the matrix $\bar\theta^*(M_2)\in \Mat_{3\times 3}(\C(\Delta_1))$: either it is of rank~$1$ (case \ref{GammaDiscNonRed}) or it is of rank~$2$, corresponding to a singular irreducible conic (case \ref{GammaDiscIrrpi}).
This gives the same two possibilities for $\widebar{M_2}\in \Mat_{3\times 3}(C)$ as $\bar{\theta}^*$ is a field homomorphism. As the rank of $M_2$ is smaller than~$3$, the variety $C$ is in the discriminant of $(X_2\times Y)/(B\times Y)$ and is thus contained in $\Delta_2\times Y$ as desired. It remains to see that $C$ is not contained in $\{q\}\times Y$ for some point $q$. Indeed, the preimage $(\eta_2\times \id)^{-1}(\{q\}\times Y)$ is isomorphic to $\eta_2^{-1}(\{q\})\times Y$, which is not irreducible, as $\eta_2^{-1}(\{q\})$ is the union of two lines (again by Proposition~\ref{pro:CBDivisorSingFibres}), but which is reduced.
\end{proof}

\subsection{Conic bundles associated to smooth cubic curves}
The principal result in this section is Proposition~\ref{pro:EnoughCB}, which provides a family of conic bundles that we shall use in the next section to prove Theorem~\ref{TheoremManyCB}.

\begin{lemma}\label{lem:ThreeConics}
For each $p=[\alpha:\beta]\in \p^1$, the set 
\[\Sl_p=\{[x_0:x_1:x_2]\in \p^2 \mid \alpha x_0^2+\beta x_1x_2=\alpha x_1^2+\beta x_0x_2=\alpha x_2^2+\beta x_0x_1=0\} \]
consists of three points if $\alpha(\alpha^3+\beta^3)=0$ and is empty otherwise.
\end{lemma}

\begin{proof}
As $\Sl_{[0:1]}=\{[1:0:0],[0:1:0],[0:0:1]\}$ and $\Sl_{[1:0]}=\emptyset$, we may assume that $\alpha\in \C^*$ and $\beta=1$. 
If $[x_0:x_1:x_2]\in \Sl_p$, then $\alpha(x_0^3-x_1^3)=x_0( \alpha x_0^2+ x_1x_2)-x_1( \alpha x_1^2+ x_0x_2)=0$. The equations being symmetric, we get $x_0^3=x_1^3=x_2^3$. In particular $x_0x_1x_2\not=0$, so the three equations are equivalent to 
\[\alpha=-\frac{x_1x_2}{x_0^2}=-\frac{x_0x_1}{x_2^2}=-\frac{x_0x_2}{x_1^2},\] 
which implies that $\alpha^3=-1$. 
For the three possible values of $\alpha$, we observe that $\Sl_{[\alpha:1]}=\{[1:x_1:-\alpha/x_1]\mid x_1^3=1\}$ consists of three points.
\end{proof}

\begin{lemma}\label{lem:Xxi}
For each $\xi\in \C$ such that $\xi^3\not=-\frac{1}{8}$, the hypersurface $X_\xi\subset \p^2\times \p^2$  of bidegree $(2,1)$ given by 
\[
X_\xi=\Bigl\{([x_0:x_1:x_2],[y_0:y_1:y_2])\in \p^2 \times \p^2 \Bigm| \sum_{i=0}^2  (x_i^2+ 2\xi \frac{x_0x_1x_2}{x_i})y_i=0\Bigr\}
\] 
is smooth, irreducible, rational over $\Q(\xi)$, and satisfies $\rho(X_\xi)=2$. 
The second projection gives a standard conic bundle $X_\xi/\p^2$. 
The discriminant curve $\Delta_\xi\subset \p^2$ is given by
\[
-\xi^2(y_0^3+y_1^3+y_2^3)+(2\xi^3+1)y_0y_1y_2=0
\]
and is the union of three lines if $\xi=0$ or if $\xi^3=1$, and is a smooth cubic otherwise. 
\end{lemma}

\begin{proof}
To show that $X_\xi$ is smooth, irreducible, rational over $\Q(\xi)$ and that $\rho(X_\xi)=2$, it suffices to show that the first projection $X_\xi\to \p^2$ is a (Zariski locally trivial) $\p^1$-bundle. 
This amounts to showing that the coefficients of the linear polynomial in the variables $y_i$  defining $X_\xi$ are never zero, i.e.~that for each $[x_0:x_1:x_2]\in \p^2$ we cannot have
$x_0^2+ 2\xi x_1x_2=x_1^2+ 2\xi x_0x_2=x_2^2+ 2\xi x_0x_1=0.$
This follows from Lemma~\ref{lem:ThreeConics} and from the hypothesis $\xi^3\not=-\frac{1}{8}$.

The equation of $X_\xi$ is given by 
\[
( x_0 \, x_1 \, x_2 ) \cdot M \cdot \begin{pmatrix} x_0 \\ x_1 \\ x_2\end{pmatrix}=0
\text{ with }
M=\begin{pmatrix}
y_0& \xi y_2& \xi y_1\\
\xi y_2& y_1& \xi y_0\\
 \xi y_1& \xi y_0& y_2
\end{pmatrix} \in \Mat_{3\times 3}(\C[y_0,y_1,y_2]).
\]
The polynomial $\det(M)$ is equal to 
 \[\det(M)=\lambda (y_0^3+y_1^3+y_2^3)+\mu y_0y_1y_2,\text{ with }\lambda=-\xi^2\text{ and }\mu=2\xi^3+1.\]
In particular, the fibres of the second projection $X_\xi/ \p^2$ are all conics (the coefficients of $x_i^2$ is $y_i$ so not all coefficients can be zero) and a general one is irreducible. 
As the threefold $X_\xi$ is smooth, irreducible and satisfies $\rho(X_\xi)=2$,  the morphism $X_\xi/ \p^2$ is a standard conic bundle. Its discriminant is given by the zero locus of $\det(M)$, which is a polynomial of degree $3$ which has the classical Hesse Form. 
The discriminant corresponds to a smooth cubic if  $\lambda(27\lambda^3+\mu^3)\not=0$, and to the union of three lines in general position otherwise.  
To prove this classical fact, we compute the partial derivatives of $\det(M)$, which are $(3\lambda y_0^2+\mu y_1y_2,3\lambda y_1^2+\mu y_0y_2,3\lambda y_2^2+\mu y_0y_1).$
By Lemma~\ref{lem:ThreeConics}, this has no zeroes in $\p^2$ if $\lambda(27\lambda^3+\mu^3)\not=0$ and has three zeroes otherwise. It remains to observe that $27\lambda^3+\mu^3=(8\xi^3+1)(\xi^3-1)^2$.
\end{proof}

\begin{remark}
Let $\k$ be a subfield of $\C$ and $\xi\in\k$. Then the curve $\Delta_\xi$ of Lemma~\ref{lem:Xxi} is defined over $\k$ and has a $\k$-rational point, namely the inflexion point $[0:1:-1]$.
When $\k=\C$, one can prove that all elliptic curves are obtained in this way; for smaller fields this does not seem to be true. We will however show that there are enough such curves.
\end{remark}

We thank P. Habegger for helpful discussions concerning the next lemma.

\begin{lemma}\label{lem:Isogenies}
Let $\k\subseteq \C$ be a subfield. 
\begin{enumerate}
\item\label{jinvariant}
For each $\xi\in \k$,  with $\xi^3\notin \{0,-\frac{1}{8},1\}$, we denote $($as in Lemma~$\ref{lem:Xxi})$ by $\Delta_\xi$ the smooth cubic curve  defined over $\k$ given by
\[-\xi^2(y_0^3+y_1^3+y_2^3)+(2\xi^3+1)y_0y_1y_2=0.\]
The $j$-invariant of $\Delta_\xi$ is equal to
\[\left(\dfrac{16\xi^{12}+464\xi^9+240\xi^6+8\xi^3+1}{\xi^2(8\xi^9-15\xi^6+6\xi^3+1)}\right)^3.\]
\item
\label{isogenies}
There is a subset $J\subseteq \k$ having the same cardinality as $\k$ such that for all $\xi,\xi'\in J$, the following are equivalent:
\begin{enumerate}
\item
There exist surjective morphisms $\Delta_\xi\tto \Delta_{\xi'}$ and $\Delta_{\xi'}\tto \Delta_{\xi}$ defined over $\C$;
\item
$\xi=\xi'$.
\end{enumerate}
\end{enumerate}
\end{lemma}

\begin{proof}
\ref{jinvariant}. 
By Lemma~\ref{lem:Xxi}, $\Delta_\xi$ is a smooth cubic curve if $\xi^3\notin \{0,-\frac{1}{8},1\}$.
We choose the inflexion point $[0:1:-1]\in \Delta_\xi$ to be the origin, make a coordinate change so that the inflexion line is the line at infinity, and thusly obtain a Weierstrass form. Then we compute the $j$-invariant as in \cite[\III.1 page 42]{silverman}; this is tedious but straightforward. 
This can also be done using the formulas from \cite[page 240]{ArtebaniDolgachev}.

\ref{isogenies}.
Let $\xi,\xi'\in \k$ be such that $\xi^3,(\xi')^3\notin \{0,-\frac{1}{8},1\}$. We see the curves $\Delta_\xi$ and $\Delta_{\xi'}$ as elliptic curves defined over $\k$ with origin $O=[0:1:-1]$. Suppose that there is a surjective morphism $\phi\colon \Delta_\xi\tto \Delta_{\xi'}$ defined over $\C$. It sends the origin of $\Delta_\xi$ onto a $\C$-rational point of $\Delta_{\xi'}$. Applying a translation at the target, we can assume that $\phi(O)=O$, which means that $\phi$ is an isogeny, and that $\Delta_{\xi}$ and $\Delta_{\xi'}$ are isogenous over $\C$ (see \cite[Definition, \S III.4 page 66]{silverman}).  

We now choose a sequence $p_1,p_2,\ldots$ of increasing prime numbers such that for each $i\ge 2$, the prime number $p_i$ does not appear in the denominator of the $j$-invariant of $\Delta_{p_{i'}}$ for each $i'<i$. For each $i\ge 1$, the  $j$-invariant of $\Delta_{p_i}$ is an element of $\Q$ having a denominator divisible by $p_i$ (follows from~\ref{jinvariant}), so $\Delta_{p_i}$ does not have potential good reduction modulo $p_i$ but this does not hold for $\Delta_{p_{i'}}$ for $i'>i$, which then has potential good reduction modulo $p_{i}$ \cite[Proposition 5.5, \S VII.5, page 197]{silverman}. This implies that there is no isogeny $\Delta_{p_i}\to \Delta_{p_{i'}}$ defined over any number field $K$ and where one curve has good reduction and the other has bad reduction \cite[Corollary 7.2, \S VII.7, page 202]{silverman}, and thus no isogeny defined over $\C$ \cite[Lemma 6.1]{MasserWustholz}.
If $\k$ is countable, this achieves the proof of \ref{isogenies}.

It remains to consider the case where $\k$ is an uncountable subfield of $\C$. 
The set of $j$-invariants of curves $\Delta_\xi$, where $\xi\in \k$ is such that $\xi^3\notin \{0,-\frac{1}{8},1\}$, is then uncountable too. 

We denote by $\Omega\subseteq \C^2$ the set consisting of pairs $(j_1,j_2)\in \widebar{\Q}^2$ such that the curves of $j$-invariants $j_1$ and $j_2$ are isogenous. The set $\Omega$ is a countable union of algebraic curves defined over $\Q$, given by the zero set of  the so-called modular transformation polynomials (see \cite[5\S3]{lang} and in particular \cite[Theorem 5, Chapter~5\S3, page 59]{lang}). Moreover, these curves are irreducible and invariant under the exchanges of variables $(x,y)\mapsto (y,x)$ \cite[Theorem 3, Chapter~5\S3, page 55]{lang}, so are not vertical or horizontal lines in $\C^2$.

We write  $S=\{\xi\in \k\mid \xi^3\notin \{0,-\frac{1}{8},1\}\}$. 
Then, by the previous paragraph, for each element $\xi\in S$ the curve $\Delta_\xi$ is isogeneous (over $\C$) to only countably many isomorphism classes of $\Delta_{\xi'}$ with $\xi'\in\k$. 
Putting an equivalence relation on $S$ saying that two elements are equivalent if the curves are isogeneous over $\C$ (see \cite[III.6, Theorem 6.1\parent{a}]{silverman}), we obtain that each equivalence class is countable, so the set of equivalence classes has the cardinality of $S$, or equivalently of $\k$. This achieves the proof.
\end{proof}

\begin{proposition}\label{pro:EnoughCB}
Let $\k$ be a subfield of $\C$. For each $n\ge 3$, there is a set $J$ having the cardinality of $\k$ indexing decomposable conic bundles $X_i/B_i$ defined over $\k$, where $X_i,B_i$ are smooth varieties rational over $\k$, and such that two conic bundles $X_i/B_i$ and $X_j/B_j$ are equivalent $($over $\C)$ if and only if $i=j$.
\end{proposition}

\begin{proof}
We choose the set $J\subseteq \k$ of Lemma~\ref{lem:Isogenies}\ref{isogenies}, and consider, for each $\xi\in J$, the hypersurface $X_\xi\subset \p^2\times \p^2$ of Lemma~\ref{lem:Xxi}, which is given by 
\[
X_\xi=\Bigl\{([x_0:x_1:x_2],[y_0:y_1:y_2])\in \p^2 \times \p^2 \Bigm| \sum_{i=0}^2  (x_i^2+ 2\xi \frac{x_0x_1x_2}{x_i})y_i=0\Bigr\}
\]
By Lemma~\ref{lem:Xxi}, the second projection gives a standard conic bundle $X_\xi\to \p^2$ whose discriminant curve $\Delta_\xi\subset \p^2$ is given by $-\xi^2(y_0^3+y_1^3+y_2^3)+(2\xi^3+1)y_0y_1y_2$. 
Note that $(X_\xi \times \p^{n-3})/(\p^2\times \p^{n-3})$ (or simply $X_\xi/\p^2$ if $n=3$) is a decomposable conic bundle defined over $\k$, as it is embedded in the trivial $\p^2$-bundle $(\p^2\times \p^2\times \p^{n-3})/(\p^2\times \p^{n-3})$ by construction. Moreover, $X_\xi \times \p^{n-3}$ is birational to $\p^n$ over $\k$ (Lemma~\ref{lem:Xxi}).
By Proposition~\ref{pro:Isogenies}, two conic bundles $(X_\xi \times \p^{n-3})/(\p^2\times \p^{n-3})$ and $(X_{\xi'} \times \p^{n-3})/(\p^2\times \p^{n-3})$ are equivalent only if there exist surjective morphisms $\Delta_\xi\tto \Delta_{\xi'}$ and $\Delta_{\xi'}\tto \Delta_{\xi}$. This is only possible if $\xi=\xi'$, by Lemma~\ref{lem:Isogenies}\ref{isogenies}. 
\end{proof}

\subsection{Proofs of Theorems~\ref{TheoremManyCB} and \ref{TheoremTame}}

\begin{proof}[Proof of Theorem~\ref{TheoremManyCB}]
By Theorem~\ref{TheoremBirMori}, we have respectively a group homomorphism and a groupoid homomorphism:
\[\begin{tikzcd}[link]
\Bir(\p^n) \ar[rr]&&\bigast_{C\in \CB(\p^n)} \left(\bigoplus_{\MC(C)}\Z/2\right)  \\ \\
\BirMori(\p^n)  \ar[uu,phantom,"\vertsubseteq"]\ar[uurr]&
\end{tikzcd}\]
For each subfield $\k\subseteq \C$, we can embed $\Bir_\k(\p^n)$ into $\Bir_\C(\p^n)$ and look at the image in $\bigast_{C\in \CB(\p^n)} \left(\bigoplus_{\MC(C)}\Z/2\right)$.
We  consider the set of decomposable conic bundles $X_i/B_i$ defined over $\k$ indexed by $J$  of Proposition~\ref{pro:EnoughCB}, which give pairwise distinct elements of $C_i\in \CB(\p^n)$, and associate to these birational maps $\psi_i\colon X_i\rat \p^n$ defined over $\k$. 
For each $i\in J$, there is an involution $\iota_i\in\psi_i\Bir_\k(X_i/B_i)\psi_i^{-1}\subseteq \Bir_\k(\p^n)$ whose image in $\bigoplus_{\MC(C_i)}\Z/2$ is not trivial by Proposition~\ref{pro:largeimageAnyCB}. One can thus take a projection $\bigoplus_{\MC(C_i)}\Z/2\to \Z/2$ such that the image of $\iota_i$ is  non-trivial.
We obtain a surjective group homomorphism from $\Bir_\k(\p^n)$ to $\smash{\bigast_{i\in J} \Z/2}$ where $J$ has the cardinality of $\k$ and such that each involution $\iota_i\in \Bir_\k(\p^n)$ is sent onto the generator indexed by $i$. There is thus a section of this surjective group homomorphism.
\end{proof}

\begin{remark}\label{rem:product_and_sum}
As Proposition~\ref{pro:EnoughCB} gives an infinite image, the above proof naturally gives a surjective homomorphism to the group $\bigast_{J} (\bigoplus_\Z \Z/2)$, but since there is an abstract surjective homomorphism from $\bigast_{J}  \Z/2$ to this group, we chose not to mention the direct sum in the statement of the theorem. 

Moreover, with the alternative form the existence of a section would be far less clear. 
Indeed, $(\Z/2)^3$ does not embed in $\Bir(X/B)$ and $(\Z/2)^7$ does not embed in $\Bir(X)$, for $X$ rationally connected of dimension $3$ \cite{Pro2011,Pro2014}, so it seems probable that $\bigoplus_\Z \Z/2$ does not embed in $\Bir(X)$ for any variety $X$.
\end{remark}

\begin{proof}[Proof of Theorem~\ref{TheoremTame}]
We consider a subfield $\k$ of $\C$, an integer $n\ge 3$, and a subset $S\subset  \Bir_\k(\p^n)$
of cardinality smaller than the one of $\k$. 
We want to construct a surjective homomorphism $\Bir_\k(\p^n) \tto \Z/2$ such that the group $G$  generated by $\Aut_\k(\p^n)$, by all Jonqui\`eres elements and by $S$ is contained in the kernel.
We use the group homomorphism 
\[\tau\colon \Bir_\k(\p^n) \to \bigast_J \Z/2 \] 
given by  Theorem~\ref{TheoremManyCB}. 
Each $j \in J$ corresponds to a conic bundle $X_j/B_j$.
The group $\Aut_\k(\p^n)$ is in the kernel of $\tau$. 
The group of Jonqui\`eres elements is conjugated to the subgroup $J\subset \Bir(\p^1 \times \p^{n-1})$ consisting of elements sending a general fibre of $\p^1\times \p^{n-1}/\p^{n-1}$ onto another one. 
The action on the base yields an exact sequence
\[ 1\to \Bir(\p^1 \times \p^{n-1}/\p^{n-1})\to J\to \Bir(\p^{n-1})\to 1.\]
This gives $J=\Bir(\p^1 \times \p^{n-1}/\p^{n-1}) \rtimes J'$, where $J'\subset J$ is the group isomorphic to $\Bir(\p^{n-1})$ that acts on $\p^1\times \p^{n-1}$ with trivial action on the first factor.
We can assume that $\p^1 \times \p^{n-1}/\p^{n-1} = X_{j_0}/B_{j_0}$ for some $j_0 \in J$. The image of $\Bir(\p^1 \times \p^{n-1}/\p^{n-1})$ by $\tau$  is contained in the group $\Z/2$ indexed by $j_0$. 
Now observe that $J'\subset \Ker\tau$. 
Indeed, we first decompose an element of $J'\simeq \Bir(\p^{n-1})$ as a product of Sarkisov links between terminal Mori fibre spaces $Y_i\to S_i$, where $Y_i$ has dimension $n-1$, and observe that taking the product with $\p^1$ gives Sarkisov links between the Mori fibre spaces $Y_i\times \p^1 \to S_i\times \p^1$ of dimension $n$.
Each of the Sarkisov links of type $\II$ arising in such decomposition has covering gonality $1$, as $\covgon(\Gamma\times \p^1)=1$ for each variety $\Gamma$.

We consider the group homomorphism 
$\hat\tau\colon \Bir_\k(\p^n)\to \bigast_{J \setminus \{j_0\}}\Z/2$ obtained by composing $\tau$ with the projection $\bigast_{J}\Z/2\to \bigast_{J \setminus \{j_0\}}\Z/2$ obtained by forgetting the factor indexed by $j_0$.

The image by $\hat\tau$ of all Jonqui\`eres elements is trivial, hence the group $\hat\tau(G)$ has at most the cardinality of $S$, which by assumption is strictly smaller than the cardinality of $J$.
We construct the expected morphism by projecting from $\hat\tau(\Bir_\k(\p^n))$ onto a factor $\Z/2$ which is not in the image of $G$.
\end{proof}

\section{Complements}\label{sec:Complements}

\subsection{Quotients and SQ-universality} 
\label{sec:SQ}

A direct consequence of Theorem~\ref{TheoremManyCB} is that we have a lot of quotients of $\Bir_\k(\p^n)$ for $n\ge 3$.

Firstly, we can have quite small quotients (which is not the case for $\Bir_\C(\p^2)$ which has no non-trivial countable quotient, as mentioned before):

\begin{corollary}
For each $n\ge 3$, each subfield $\k\subseteq \C$, and each integer $m\ge 1$, there are $($abstract$)$ surjective group homomorphisms from $\Bir_\k(\p^n)$ to the dihedral group $D_{2m}$ of order $2m$ and the symmetric group $\Sym_m$. 
In particular, there is a normal subgroup of $\Bir_\k(\p^n)$ of index $r$ for each even integer $r> 1$.
\end{corollary}
\begin{proof}
Follows from Theorem~\ref{TheoremManyCB} and the fact that $D_{2m}$ and $\Sym_m$ are generated by involutions.
\end{proof}
 
 Secondly, we get much larger quotients:

\begin{corollary}
For any $n\ge 3$, any subfield $\k\subseteq \C$ and any integer $m\ge 1$, there are $($abstract$)$ surjective group homomorphisms
\begin{align*}
\Bir_\k(\p^n)\tto \SL_m(\k),&&\Bir_\k(\p^n)\tto \Bir_{\overline{\Q}}(\p^2).
\end{align*}
\end{corollary}

\begin{proof}
We observe that $\SL_m(\k)$ has the cardinality of $\k$ and that $\Bir_{\overline{\Q}}(\p^2)$ is countable. Hence, both groups have at most the cardinality of $\k$. 
Both groups are generated by involutions: for $\Bir_{\overline{\Q}}(\p^2)$ this is by the Noether-Castelnuovo Theorem which says that $\Bir_{\overline{\Q}}(\p^2)$ is generated by the standard quadric involution and by $\Aut_{\overline{\Q}}(\p^3)\simeq \PGL_3(\overline{\Q})=\PSL_3(\overline{\Q})$, and thus is generated by involutions. 
Hence, the two groups are quotients of $\bigast_{J} \Z/2$. 
The result then follows from Theorem~\ref{TheoremManyCB}.
\end{proof}

Similarly, over $\C$ we get:
 
\begin{corollary}
For any $n \ge 3$, there exists a surjective group homomorphism 
\[
\Bir_\C(\p^n)\tto  \Bir_\C(\p^2).
\]
\end{corollary}

Recall that a group $G$ is \emph{SQ-universal} if any countable group embeds in a quotient of $G$.
The free group $\Z * \Z$ was an early example of SQ-universal group.
More generally any nontrivial free product $G_1 * G_2$ distinct from $\Z/2 * \Z/2$ is SQ-universal, see \cite[Theorem 3]{Schupp}. 
From a modern point of view, this also follows from \cite{MinasyanOsin}, by looking at the action of any loxodromic isometry on the associated Bass-Serre tree.
In particular, taking $G_1 = \Z/2 * \Z/2$ and $G_2 = \Z/2$, we get that $\Z/2 * \Z/2 * \Z/2$ is SQ-universal.

\begin{corollary}
For any field $\k\subseteq \C$ and any $n \ge 3$, the Cremona group $\Bir_\k(\p^n)$ admits a surjective morphism to the SQ-universal group $\Z/2 * \Z/2 * \Z/2$.
In particular, $\Bir_\k(\p^n)$ also is SQ-universal.
\end{corollary}

\begin{proof}
Follows from Theorem~\ref{TheoremManyCB} and from the fact that $\Z/2 * \Z/2 * \Z/2$ is SQ-universal.
\end{proof}

\subsection{Hopfian property}\label{sec:Hopfian}
Recall that a group $G$ is \emph{hopfian} is every surjective group homomorphism $G\tto G$ is an isomorphism.
It was proven in \cite{Deserti} that the group $\Bir_\C(\p^2)$ is hopfian. An open question, asked by I. Dolgachev (see \cite{Des17}), is whether the Cremona group $\Bir_\C(\p^n)$ is generated by involutions for each $n$, the answer being yes in dimension $2$ and open in dimension $\ge 3$. Theorem~\ref{TheoremManyCB} relates these two notions and shows that we cannot generalise both results at the same time (being hopfian and generated by involutions) to higher dimension.

\begin{corollary}\label{Cor:Hopfian}
For each $n\ge 3$ and each subfield $\k\subseteq \C$, the group $\Bir_\k(\p^n)$ is not hopfian if it is generated by involutions.
\end{corollary}

\begin{proof}
Follows from Theorem~\ref{TheoremManyCB}, as the group homomorphisms provided by  Theorem~\ref{TheoremManyCB} is not injective, and because $\Bir_\k(\p^n)$ has the same cardinality as $\k$ (the set of all polynomials of degree $n$ with coefficients in $\k$ has the same cardinality as $\k$).
\end{proof}

\subsection{More general fields}
\label{sec:field}

Every field isomorphism $\k\iso\k'$ naturally induces an isomorphism $\Bir_\k(\p^n)\iso \Bir_{\k'}(\p^n)$. More generally, it associates to each variety and each rational map defined over $\k$, a variety and a rational map defined over $\k'$. It then induces an isomorphism between the group of birational maps defined over $\k$ and $\k'$ of the varieties obtained. This implies that the five
Theorems~\ref{TheoremBirPnnotsimple}-\ref{TheoremManyCB} also hold for each ground field which is abstractly isomorphic to a subfield of $\C$.
This includes any field of rational functions of any algebraic variety defined over a subfield of $\C$ as these fields have characteristic zero and cardinality smaller or equal than the one of $ \C$. 
 
\subsection{Amalgamated product structure}

 We work over the field $\C$. 
 In the next result, an element of $\CB(X)$ is said to be decomposable if it is the class of a decomposable conic bundle (in the sense of Definition~\ref{def:DecomposableCB}).

\begin{theorem}\label{thm:AmalgamatedProduct}
For each integer $n\ge 3$, and let $X/B$ be a conic bundle, where $X$ is a terminal variety of dimension $n$. We denote by $\rho$ the group homomorphism
\[\rho\colon \Bir(X) \to \bigast_{C\in \CB(X)} \left(\bigoplus_{\MC(C)}\Z/2\right)\]
given by Theorem~\ref{TheoremBirMori}. For each $C\in \CB(X)$ we fix a choice of representative $X_C/B_C$, and we denote $G_C = \rho^{-1} (\rho (\Bir(X_C/B_C)))\subseteq \Bir(X)$.
Then, the following hold:
\begin{enumerate}
\item\label{IntersectA} For all $C \neq C'$ in $\CB(X)$, the group $A = G_C \cap G_{C'}$ contains $\ker \rho$ and does not depend on the choice of $C,C'$;
\item \label{BirXFreeproduct}
The group $\Bir(X)$ is the free product of the groups $G_C$, $C\in \CB(X)$,  amalgamated over their common intersection $A$:
\[
\Bir(X) = \bigast_{A} G_C.
\]
\item\label{NonTrivial}
For each decomposable $C\in \CB(X)$ we have $A\subsetneq G_C$. 
Moreover, the free product of \ref{BirXFreeproduct} is non-trivial $($i.e.~$A\subsetneq G_C\subsetneq \Bir(X)$ for each $C)$ as soon as $\CB(X)$ contains two distinct decomposable elements. 
This is for instance the case when $X$ is rational, as $\CB(X)$ then contains uncountably many decomposable elements.
\end{enumerate}
\end{theorem}

\begin{proof}
\ref{IntersectA}.
For each $C\in \CB(X)$, we denote by $H_C=\left(\bigoplus_{\MC(C)}\Z/2\right)$ the factor indexed by $C$ in the free product $\bigast_{C\in \CB(X)} \left(\bigoplus_{\MC(C)}\Z/2\right)=\bigast_{C\in \CB(X)} H_C$. 
By definition of the group homomorphism, for each $C\in \CB(X)$ we have $\rho(\Bir(X_C/B_C)) \subseteq H_C$. As $H_C$ is a $\F_2$-vector space with basis $\MC(C)$ and $\rho(\Bir(X_C/B_C))$ is a linear subspace, there exists a projection $H_C\to \rho(\Bir(X_C/B_C))$. We then denote by
\[
\rho'\colon \Bir(X) \to \bigast_{C\in \CB(X)} \rho(\Bir(X_C/B_C))
\]
the group homomorphism induced for each $C$ by the projection 
\[H_C\to\rho(\Bir(X_C/B_C)).\] 
By definition of the free product, we obtain $H_C\cap H_{C'}=\id$ for all $C\not=C'$. 
This implies that $G_C \cap G_{C'} = \ker \rho' \supseteq \ker \rho$. 

\ref{BirXFreeproduct}.
We first observe that by construction the groups $G_C$ generate the group $\Bir(X)$. 
The fact that $\Bir(X) = \bigast_{A} G_C$ corresponds to saying that all relations in $\Bir(X)$ lie in the groups $G_C$. This follows from the group homomorphism $\rho$ to a free product, where no relation between the groups $H_C$ exists.

\ref{NonTrivial}.
The fact that $A\subsetneq G_C$ for each decomposable $C$ follows from Proposition~\ref{pro:largeimageAnyCB}. 
Hence, the free product of \ref{BirXFreeproduct} is non-trivial if there are least two $C$ corresponding to decomposable conic bundles. 
If $X$ is rational, then we moreover have uncountably many such elements by Proposition~\ref{pro:EnoughCB}.
\end{proof}

In Theorem~\ref{thm:AmalgamatedProduct}, one could be tempted to say that $A=\ker\rho$, but this is not clear. 
Indeed, it could be that some elements of $\bigoplus_{\MC(C)}\Z/2$ are in the image of $\Bir(X)$ but not in the image of $\Bir(X/B)$.

\subsection{Cubic varieties}\label{Sec:CubicVar}

Here again we work over $\C$.
We recall the following result, which allows to apply Theorem~\ref{TheoremBirXBnotsimple} to any smooth cubic hypersurface of dimension $\ge 3$:

\begin{lemma}\label{Lem:SmoothCubicHypersurfaceLine}
Let $n\ge 4$ and let $\ell\subset X\subset \p^n$ be a line on a smooth cubic hypersurface. 
We denote by $\hat{X}$ and $P$ the respective blow-ups of $X$ and $\p^n$ along $\ell$. 
Then, the projection $pr_{\ell}$ away from $\ell$ gives rise to a decomposable conic bundle and a decomposable $\p^2$-bundle
\[\hat{X}\subset P=\p(\mathcal{O}_{\p^{n-2}}\oplus \mathcal{O}_{\p^{n-2}}\oplus \mathcal{O}_{\p^2}(1))\stackrel{pr_{\ell}}\to \p^{n-2}.\]
Moreover, the discriminant of the conic bundle is a hypersurface of degree $5$.
\end{lemma}
\begin{proof}
We take coordinates $[y_0:y_1:\cdots:y_{n-2}:u:v]$ on $\p^n$ and assume that $\ell\subset \p^n$ is the line given by $y_0=y_1=\cdots=y_{n-2}=0$.
The equation of $X$ is then given by 
\[Au^2+2Buv+Cv^2+2Du+2Ev+F=0\]
where $A,B,C,D,E,F\in \C[y_0,\ldots,y_{n-2}]$ are homogeneous polynomials of degree $1,1,1,2,2,3$ respectively.

As in the proof of Proposition~\ref{pro:largeimageAnyCB}, we view $P=\p(\mathcal{O}_{\p^{n-2}}\oplus \mathcal{O}_{\p^{n-2}}\oplus \mathcal{O}_{\p^{n-2}}(1))$ as the quotient of $(\A^{2}\setminus \{0\})\times (\A^{n-1}\setminus \{0\})$ by $(\mathbb{G}_m)^2$ via 
\[\left((\lambda,\mu),(x_0,x_1,x_2,y_0,y_1,\cdots,y_{n-2})\right)\mapsto (\lambda x_0,\lambda x_1,\lambda\mu^{-1} x_2,\mu y_0,\cdots,\mu y_{n-2})\]
and denote by $[x_0\colon x_1 \colon x_2 \semicolon y_0 \colon \cdots \colon y_{n-2}]\in P$ the class of $(x_0,x_1,x_2,y_0,\cdots,y_{n-2})$. The birational morphism 
\[
\begin{tikzcd}[map]
P & \to & \p^n\\
{[x_0\colon x_1 \colon x_2 \semicolon y_0 \colon y_1 \colon y_2\colon \cdots \colon y_{n-2}]} & \mapsto & {[x_2y_0 \colon \cdots \colon x_2y_{n-2}:x_0:x_1]}
\end{tikzcd}\]
is the blow-up of $\ell$, so $\hat{X}$ is given by 
\[Ax_0^2+2Bx_0x_1+Cx_1^2+2Dx_2x_0+2Ex_2x_1+Fx_2^2=0,\]
which is then a conic bundle over $\p^2$. The discriminant of the curve gives a hypersurface $\Delta\subset \p^2$ of degree $5$, given by the determinant of 
$\left(\begin{smallmatrix} A & B & D \\
 B & C & E \\ D & E & F\end{smallmatrix}\right)$. 
\end{proof}

\begin{corollary}\label{cor:SmoothCubicBirXZ2}
For each $n\ge 4$ and each smooth cubic hypersurface $X\subset \p^n$, there exists a surjective group homomorphism $\Bir(X)\tto \bigoplus_\Z \Z/2$
\end{corollary}

\begin{proof}
Follows from the application of Theorem~\ref{TheoremBirXBnotsimple} to the conic bundle associated to blow-up of a line of $X$ (Lemma~\ref{Lem:SmoothCubicHypersurfaceLine}).
\end{proof}

Every smooth cubic threefold $X\subset \p^4$ is not rational, and moreover two such cubics are birational if and only if they are projectively equivalent, i.e.~equal up to an element of $\Aut(\p^4)=\PGL_5(\C)$ \cite{ClemensGriffiths}. 
We moreover get the following:

\begin{proposition}\label{Prop:Cubicfreeproduct}
Let $X\subset \p^4$ be a general smooth cubic hypersurface. We have a surjective group homomorphism 
$\Bir(X)\tto \bigast_J \Z/2$, where $J$ has the cardinality of~$\C$.
\end{proposition}

\begin{proof}
The map of Lemma~\ref{Lem:SmoothCubicHypersurfaceLine} associates to each smooth cubic threefold $X$ and each line $\ell\subset X$ a quintic curve $\Delta\subset \p^2$ and also a theta-characteristic; this induces a birational map between the pairs $(\ell,X)$ of lines on smooth cubic threefolds, up to $\PGL_5(\C)$, and the pairs $(\theta,\Delta)$, where $\Delta\subset \p^2$ is a smooth quintic and $\theta$ is a theta-characteristic, again up to $\PGL_3(\C)$ \cite[Theorem 4.1 and Proposition 4.2]{CasRob}. 
 
In particular, taking a general smooth cubic hypersurface $X\subset \p^4$ and varying the lines $\ell\subset X$ (which form a $2$-dimensional family), we obtain a family $J$ of dimension $2$ of smooth quintics $\Delta\subset \p^2$, not pairwise equivalent modulo $\PGL_3(\C)$. This yields conic bundles that are not pairwise equivalent, parametrised by a complex algebraic variety of dimension $2$. 
Applying the group homomorphism of Theorem~\ref{TheoremBirMori} and projecting on the corresponding factors provides a surjective group homomorphism 
$\Bir(X)\tto \bigast_J \Z/2$, similarly as in the proof of Theorem~\ref{TheoremManyCB}.
\end{proof}

\subsection{Fibrations graph} \label{sec:graph}

We explain how to get a natural graph structure from the set of rank~$r$ fibrations, similarly as in \cite{LZ17}. 

Let $Z$ be a variety birational to a Mori fibre space.
We construct a sequence of nested graphs $\Gl_n$, $n \ge 1$, as follows.
The set of vertices of $\Gl_n$ are rank~$r$ fibrations $X/B$, for any $r \le n$, with a choice of a birational map $\phi\colon Z \rat X$, and modulo $Z$-equivalence (Definition \ref{def:Tequivalent}).
We denote $(X/B, \phi)$ such an equivalence class.
We put an oriented edge from $(X/B, \phi)$ to $(X'/B', \phi')$ if $\rho(X'/B') = \rho(X/B) - 1$ and the birational maps from $Z$ induce a factorisation of $X/B$ through $X'/B'$, that is, if there is a morphism $B' \to B$ and a birational contraction $X \rat X'$ such that the following diagram commutes
\[
\begin{tikzcd}[link]
& Z  \ar[dr, dashed,"\phi'"] \ar[dl, dashed, "\phi",swap] \\
X \ar[rr,dashed] \ar[dd] && X'  \ar[dd] \\ \\
B  && B' \ar[ll]
\end{tikzcd}
\]
We call the graph $\Gl := \bigcup_n \Gl_n$ the \emph{fibrations graph} associated with $Z$.
The group $\Bir(Z)$ naturally acts on each graph $\Gl_n$, and so also on $\Gl$, by precomposition :
\[
g \cdot (X/B, \phi) := (X/B, \phi \circ g^{-1}).
\]

The fact that Sarkisov links generate $\BirMori(Z)$ is equivalent to the fact that $\Gl_2$ is a connected graph.
Lemma \ref{lem:2 rank 2} implies that $\Gl_3$ is the $1$-skeleton of a square complex, where each square has one vertex of rank 3, one vertex of rank 1 and two vertices of rank 2.
The fact that elementary relations generate all relations in $\BirMori(Z)$ is equivalent to the fact that this square complex is simply connected.

It is not clear to us if for $n \ge 4$ the graph $\Gl_n$ is still the $1$-skeleton of a cube complex.  
  
\begingroup \vbadness 10000\relax 
\bibliographystyle{myalpha}
\bibliography{biblio}

\newcommand{\etalchar}[1]{$^{#1}$}
\begin{thebibliography}{BCTSSD85}

\bibitem[AO18]{AhmOka}
H.~Ahmadinezhad \& T.~Okada.
\newblock Stable rationality of higher dimensional conic bundles.
\newblock {\em \'{E}pijournal Geom. Alg\'{e}brique}, 2:Art. 5, 13, 2018.

\bibitem[AZ16]{AZ}
H.~{Ahmadinezhad} \& F.~{Zucconi}.
\newblock {Mori dream spaces and birational rigidity of Fano 3-folds.}
\newblock {\em {Adv. Math.}}, 292:410--445, 2016.

\bibitem[AZ17]{AZ2017}
H.~Ahmadinezhad \& F.~Zucconi.
\newblock Circle of {S}arkisov links on a {F}ano 3-fold.
\newblock {\em Proc. Edinb. Math. Soc. (2)}, 60(1):1--16, 2017.

\bibitem[{Alb}02]{Carraminana}
M.~{Alberich-Carrami\~nana}.
\newblock {\em Geometry of the plane {C}remona maps}, volume 1769 of {\em
  Lecture Notes in Mathematics}.
\newblock Springer-Verlag, Berlin, 2002.

\bibitem[AR04]{AlzatiRusso}
A.~Alzati \& F.~Russo.
\newblock Some extremal contractions between smooth varieties arising from
  projective geometry.
\newblock {\em Proc. London Math. Soc. (3)}, 89(1):25--53, 2004.

\bibitem[AD09]{ArtebaniDolgachev}
M.~Artebani \& I.~Dolgachev.
\newblock The {H}esse pencil of plane cubic curves.
\newblock {\em Enseign. Math. (2)}, 55(3-4):235--273, 2009.

\bibitem[ADHL15]{ADHL_2015}
I.~Arzhantsev, U.~Derenthal, J.~Hausen \& A.~Laface.
\newblock {\em Cox rings}, volume 144 of {\em Cambridge Studies in Advanced
  Mathematics}.
\newblock Cambridge University Press, Cambridge, 2015.

\bibitem[BDE{\etalchar{+}}17]{BDELU}
F.~Bastianelli, P.~{De Poi}, L.~Ein, R.~Lazarsfeld \& B.~Ullery.
\newblock Measures of irrationality for hypersurfaces of large degree.
\newblock {\em Compos. Math.}, 153(11):2368--2393, 2017.

\bibitem[BCTSSD85]{BCSS}
A.~Beauville, J.-L. Colliot-Th\'{e}l\`ene, J.-J. Sansuc \& P.~Swinnerton-Dyer.
\newblock Vari\'{e}t\'{e}s stablement rationnelles non rationnelles.
\newblock {\em Ann. of Math. (2)}, 121(2):283--318, 1985.

\bibitem[Bir16]{BirkarS}
C.~Birkar.
\newblock Singularities of linear systems and boundedness of {F}ano varieties.
\newblock {\em \href{https://arxiv.org/pdf/1609.05543}{Preprint
  arXiv:1609.05543}}, 2016.

\bibitem[Bir19]{BirkarA}
C.~Birkar.
\newblock Anti-pluricanonical systems on {F}ano varieties.
\newblock {\em Ann. of Math. (2)}, 190(2):345--463, 2019.

\bibitem[BCHM10]{BCHM}
C.~{Birkar}, P.~{Cascini}, C.~D. {Hacon} \& J.~{McKernan}.
\newblock {Existence of minimal models for varieties of log general type.}
\newblock {\em {J. Am. Math. Soc.}}, 23(2):405--468, 2010.

\bibitem[Bla10]{Blanc2010}
J.~Blanc.
\newblock Groupes de {C}remona, connexit\'e et simplicit\'e.
\newblock {\em Ann. Sci. \'Ec. Norm. Sup\'er. (4)}, 43(2):357--364, 2010.

\bibitem[BFT17]{BFT}
J.~Blanc, A.~Fanelli \& R.~Terpereau.
\newblock Automorphisms of $\mathbb{P}^1$-bundles over rational surfaces.
\newblock {\em \href{https://arxiv.org/pdf/1707.01462}{Preprint
  arXiv:1707.01462}}, 2017.

\bibitem[BFT19]{BFT2}
J.~Blanc, A.~Fanelli \& R.~Terpereau.
\newblock Connected algebraic groups acting on 3-dimensional {M}ori fibrations.
\newblock {\em manuscript in preparation}, 2019.

\bibitem[BF13]{BF13}
J.~Blanc \& J.-P. Furter.
\newblock Topologies and structures of the {C}remona groups.
\newblock {\em Ann. of Math. (2)}, 178(3):1173--1198, 2013.

\bibitem[BH15]{BlancHeden}
J.~Blanc \& I.~Hed\'en.
\newblock The group of {C}remona transformations generated by linear maps and
  the standard involution.
\newblock {\em Ann. Inst. Fourier (Grenoble)}, 65(6):2641--2680, 2015.

\bibitem[BZ18]{BZ}
J.~Blanc \& S.~Zimmermann.
\newblock Topological simplicity of the {C}remona groups.
\newblock {\em Amer. J. Math.}, 140(5):1297--1309, 2018.

\bibitem[Bro06]{Brown}
R.~Brown.
\newblock {\em Topology and groupoids}.
\newblock BookSurge, LLC, Charleston, SC, 2006.
\newblock Third edition of \textit{Elements of modern topology} [McGraw-Hill,
  New York, 1968],.

\bibitem[Can14]{CantatMorphisms}
S.~Cantat.
\newblock Morphisms between {C}remona groups, and characterization of rational
  varieties.
\newblock {\em Compos. Math.}, 150(7):1107--1124, 2014.

\bibitem[CL13]{CantatLamy}
S.~Cantat \& S.~Lamy.
\newblock Normal subgroups in the {C}remona group.
\newblock {\em Acta Math.}, 210(1):31--94, 2013.
\newblock With an appendix by Yves de Cornulier.

\bibitem[CX18]{CantatXie}
S.~Cantat \& J.~Xie.
\newblock Algebraic actions of discrete groups: The $p$-adic method.
\newblock {\em Acta Math.}, 220(2):239--295, 2018.

\bibitem[CMF05]{CasRob}
S.~Casalaina-Martin \& R.~Friedman.
\newblock Cubic threefolds and abelian varieties of dimension five.
\newblock {\em J. Algebraic Geom.}, 14(2):295--326, 2005.

\bibitem[Cas01]{Cas}
G.~Castelnuovo.
\newblock Le trasformazioni generatrici del gruppo cremoniano nel piano.
\newblock {\em Atti della R. Accad. delle Scienze di Torino}, (36):861--874,
  1901.

\bibitem[CG72]{ClemensGriffiths}
C.~H. Clemens \& P.~A. Griffiths.
\newblock The intermediate {J}acobian of the cubic threefold.
\newblock {\em Ann. of Math. (2)}, 95:281--356, 1972.

\bibitem[Cor95]{Corti1995}
A.~Corti.
\newblock Factoring birational maps of threefolds after {S}arkisov.
\newblock {\em J. Algebraic Geom.}, 4(2):223--254, 1995.

\bibitem[Cor00]{Corti_explicit}
A.~Corti.
\newblock Singularities of linear systems and {$3$}-fold birational geometry.
\newblock In {\em Explicit birational geometry of 3-folds}, volume 281 of {\em
  London Math. Soc. Lecture Note Ser.}, pages 259--312. Cambridge Univ. Press,
  Cambridge, 2000.

\bibitem[CPR00]{CPR}
A.~{Corti}, A.~{Pukhlikov} \& M.~{Reid}.
\newblock {Fano 3-fold hypersurfaces.}
\newblock In {\em {Explicit birational geometry of 3-folds}}, pages 175--258.
  Cambridge: Cambridge University Press, 2000.

\bibitem[DGO17]{DGO}
F.~Dahmani, V.~Guirardel \& D.~Osin.
\newblock Hyperbolically embedded subgroups and rotating families in groups
  acting on hyperbolic spaces.
\newblock {\em Mem. Amer. Math. Soc.}, 245(1156), 2017.

\bibitem[Dan74]{Danilov}
V.~I. Danilov.
\newblock Non-simplicity of the group of unimodular automorphisms of an affine
  plane.
\newblock {\em Mat. Zametki}, 15:289--293, 1974.

\bibitem[Dem70]{Demazure}
M.~Demazure.
\newblock Sous-groupes alg\'ebriques de rang maximum du groupe de {C}remona.
\newblock {\em Ann. Sci. \'Ecole Norm. Sup. (4)}, 3:507--588, 1970.

\bibitem[D{\'e}s07]{Deserti}
J.~D{\'e}serti.
\newblock Le groupe de {C}remona est hopfien.
\newblock {\em C. R. Math. Acad. Sci. Paris}, 344(3):153--156, 2007.

\bibitem[D{\'e}s17]{Des17}
J.~D{\'e}serti.
\newblock Cremona maps and involutions.
\newblock {\em \href{https://arxiv.org/pdf/1708.01569}{Preprint
  arXiv:1708.01569}}, 2017.

\bibitem[ELM{\etalchar{+}}06]{Ein_al}
L.~Ein, R.~Lazarsfeld, M.~Musta{\c t}{\u a}, M.~Nakamaye \& M.~Popa.
\newblock Asymptotic invariants of base loci.
\newblock {\em Ann. Inst. Fourier (Grenoble)}, 56(6):1701--1734, 2006.

\bibitem[Enr95]{Enriques}
F.~Enriques.
\newblock {\em Conferenze di Geometria: fundamenti di una geometria
  iperspaziale}.
\newblock Bologna, 1895.

\bibitem[Fuj99]{Fujino1999}
O.~Fujino.
\newblock Applications of {K}awamata's positivity theorem.
\newblock {\em Proc. Japan Acad. Ser. A Math. Sci.}, 75(6):75--79, 1999.

\bibitem[Fuj15]{Fujino2015}
O.~Fujino.
\newblock Some remarks on the minimal model program for log canonical pairs.
\newblock {\em J. Math. Sci. Univ. Tokyo}, 22(1):149--192, 2015.

\bibitem[FL10]{FurterLamy}
J.-P. Furter \& S.~Lamy.
\newblock Normal subgroup generated by a plane polynomial automorphism.
\newblock {\em Transform. Groups}, 15(3):577--610, 2010.

\bibitem[GK19]{GouKou}
F.~Gounelas \& A.~Kouvidakis.
\newblock Measures of irrationality of the {F}ano surface of a cubic threefold.
\newblock {\em Trans. Amer. Math. Soc.}, 371(10):7111--7133, 2019.

\bibitem[HM07]{HMcK2007}
C.~D. Hacon \& J.~Mckernan.
\newblock On {S}hokurov's rational connectedness conjecture.
\newblock {\em Duke Math. J.}, 138(1):119--136, 2007.

\bibitem[HM13]{HMcK}
C.~D. Hacon \& J.~McKernan.
\newblock The {S}arkisov program.
\newblock {\em J. Algebraic Geom.}, 22(2):389--405, 2013.

\bibitem[Har77]{Hartshorne}
R.~Hartshorne.
\newblock {\em Algebraic geometry}.
\newblock Springer-Verlag, New York-Heidelberg, 1977.
\newblock Graduate Texts in Mathematics, No. 52.

\bibitem[Hat02]{Hatcher}
A.~Hatcher.
\newblock {\em Algebraic topology}.
\newblock Cambridge University Press, Cambridge, 2002.

\bibitem[HK00]{HK}
Y.~Hu \& S.~Keel.
\newblock Mori dream spaces and {GIT}.
\newblock {\em Michigan Math. J.}, 48:331--348, 2000.

\bibitem[Hud27]{Hudson}
H.~P. Hudson.
\newblock {\em Cremona transformations in plane and space}.
\newblock Cambridge, University Press, 1927.

\bibitem[Isk91]{Isk1991}
V.~A. Iskovskikh.
\newblock Generators of the two-dimensional {C}remona group over a nonclosed
  field.
\newblock {\em Trudy Mat. Inst. Steklov.}, 200:157--170, 1991.

\bibitem[Isk96]{Isk1996}
V.~A. Iskovskikh.
\newblock Factorization of birational mappings of rational surfaces from the
  point of view of {M}ori theory.
\newblock {\em Uspekhi Mat. Nauk}, 51(4(310)):3--72, 1996.

\bibitem[IKT93]{IKT}
V.~A. Iskovskikh, F.~K. Kabdykairov \& S.~L. Tregub.
\newblock Relations in a two-dimensional {C}remona group over a perfect field.
\newblock {\em Izv. Ross. Akad. Nauk Ser. Mat.}, 57(3):3--69, 1993.
\newblock Translation in Russian Acad. Sci. Izv. Math. 42 (1994), no. 3,
  427--478.

\bibitem[Isk87]{Encyclo}
V.~A. Iskovskikh.
\newblock Cremona group.
\newblock In {\em
  {\href{https://www.encyclopediaofmath.org/index.php/Cremona_group}{Encyclopedia
  of Mathematics}}}. 1987.

\bibitem[Kal13]{Kaloghiros}
A.-S. Kaloghiros.
\newblock Relations in the {S}arkisov program.
\newblock {\em Compos. Math.}, 149(10):1685--1709, 2013.

\bibitem[KKL16]{KKL}
A.-S. Kaloghiros, A.~K\"uronya \& V.~Lazi\'c.
\newblock Finite generation and geography of models.
\newblock In {\em Minimal models and extremal rays ({K}yoto, 2011)}, volume~70
  of {\em Adv. Stud. Pure Math.}, pages 215--245. 2016.

\bibitem[Kat87]{Katz}
S.~Katz.
\newblock The cubo-cubic transformation of {${\bf P}^3$} is very special.
\newblock {\em Math. Z.}, 195(2):255--257, 1987.

\bibitem[Kol86]{Kollar_1986}
J.~Koll\'{a}r.
\newblock Higher direct images of dualizing sheaves. {I}.
\newblock {\em Ann. of Math. (2)}, 123(1):11--42, 1986.

\bibitem[Kol93]{Kollar93}
J.~Koll\'ar.
\newblock Effective base point freeness.
\newblock {\em Math. Ann.}, 296(4):595--605, 1993.

\bibitem[Kol96]{Kollar_rational}
J.~Koll\'ar.
\newblock {\em Rational curves on algebraic varieties}, volume~32.
\newblock Springer-Verlag, Berlin, 1996.

\bibitem[Kol97]{Kollar_SantaCruz}
J.~Koll\'ar.
\newblock Singularities of pairs.
\newblock In {\em Algebraic geometry---{S}anta {C}ruz 1995}, volume~62 of {\em
  Proc. Sympos. Pure Math.}, pages 221--287. Amer. Math. Soc., Providence, RI,
  1997.

\bibitem[Kol17]{KollarCB}
J.~Koll\'{a}r.
\newblock Conic bundles that are not birational to numerical {C}alabi-{Y}au
  pairs.
\newblock {\em \'{E}pijournal Geom. Alg\'{e}brique}, 1:Art. 1, 14, 2017.

\bibitem[KM92]{KollarMori1992}
J.~Koll\'ar \& S.~Mori.
\newblock Classification of three-dimensional flips.
\newblock {\em J. Amer. Math. Soc.}, 5(3):533--703, 1992.

\bibitem[KM98]{KM}
J.~Koll\'ar \& S.~Mori.
\newblock {\em Birational geometry of algebraic varieties}, volume 134 of {\em
  Cambridge Tracts in Mathematics}.
\newblock Cambridge University Press, Cambridge, 1998.
\newblock With the collaboration of C. H. Clemens and A. Corti, Translated from
  the 1998 Japanese original.

\bibitem[Kra96]{KraftBourbaki}
H.~Kraft.
\newblock Challenging problems on affine {$n$}-space.
\newblock {\em Ast\'erisque}, (237):Exp.\ No.\ 802, 5, 295--317, 1996.
\newblock S\'eminaire Bourbaki, Vol. 1994/95.

\bibitem[LZ20]{LZ17}
S.~Lamy \& S.~Zimmermann.
\newblock Signature morphisms from the {C}remona group over a non-closed field.
\newblock {\em JEMS}, 22(10):3133--3173, 2020.

\bibitem[Lan87]{lang}
S.~Lang.
\newblock {\em Elliptic functions}, volume 112 of {\em Graduate Texts in
  Mathematics}.
\newblock Springer-Verlag, New York, second edition, 1987.
\newblock With an appendix by J. Tate.

\bibitem[Laz04]{LazarsfeldI}
R.~Lazarsfeld.
\newblock {\em Positivity in algebraic geometry. {I}}.
\newblock Springer-Verlag, Berlin, 2004.
\newblock Classical setting: line bundles and linear series.

\bibitem[Lon16]{Lonjou}
A.~Lonjou.
\newblock Non simplicit\'e du groupe de {C}remona sur tout corps.
\newblock {\em Ann. Inst. Fourier (Grenoble)}, 66(5):2021--2046, 2016.

\bibitem[MW90]{MasserWustholz}
D.~W. Masser \& G.~W\"{u}stholz.
\newblock Estimating isogenies on elliptic curves.
\newblock {\em Invent. Math.}, 100(1):1--24, 1990.

\bibitem[MO15]{MinasyanOsin}
A.~Minasyan \& D.~Osin.
\newblock Acylindrical hyperbolicity of groups acting on trees.
\newblock {\em Math. Ann.}, 362(3-4):1055--1105, 2015.

\bibitem[MM85]{MoriMukai}
S.~Mori \& S.~Mukai.
\newblock Classification of {F}ano {$3$}-folds with {$B_2\geq 2$}. {I}.
\newblock In {\em Algebraic and topological theories ({K}inosaki, 1984)}, pages
  496--545. Kinokuniya, Tokyo, 1985.

\bibitem[Pan99]{Pan}
I.~Pan.
\newblock Une remarque sur la g\'en\'eration du groupe de {C}remona.
\newblock {\em Bol. Soc. Brasil. Mat. (N.S.)}, 30(1):95--98, 1999.

\bibitem[PS15]{PanSimis}
I.~Pan \& A.~Simis.
\newblock Cremona maps of de {J}onqui\`eres type.
\newblock {\em Canad. J. Math.}, 67(4):923--941, 2015.

\bibitem[Poo07]{Poonen}
B.~Poonen.
\newblock Gonality of modular curves in characteristic {$p$}.
\newblock {\em Math. Res. Lett.}, 14(4):691--701, 2007.

\bibitem[Pro11]{Pro2011}
Y.~Prokhorov.
\newblock {$p$}-elementary subgroups of the {C}remona group of rank 3.
\newblock In {\em Classification of algebraic varieties}, EMS Ser. Congr. Rep.,
  pages 327--338. Eur. Math. Soc., Z\"urich, 2011.

\bibitem[Pro12]{Pro2012}
Y.~Prokhorov.
\newblock Simple finite subgroups of the {C}remona group of rank 3.
\newblock {\em J. Algebraic Geom.}, 21(3):563--600, 2012.

\bibitem[Pro14]{Pro2014}
Y.~Prokhorov.
\newblock 2-elementary subgroups of the space {C}remona group.
\newblock In {\em Automorphisms in birational and affine geometry}, volume~79
  of {\em Springer Proc. Math. Stat.}, pages 215--229. 2014.

\bibitem[PS16]{ProShra}
Y.~Prokhorov \& C.~Shramov.
\newblock Jordan property for {C}remona groups.
\newblock {\em Amer. J. Math.}, 138(2):403--418, 2016.

\bibitem[Sar82]{sarkisov}
V.~G. Sarkisov.
\newblock On conic bundle structures.
\newblock {\em Izv. Akad. Nauk SSSR Ser. Mat.}, 46(2):371--408, 432, 1982.

\bibitem[Sch19]{Schneider}
J.~Schneider.
\newblock Relations in the {C}remona group over perfect fields.
\newblock {\em \href{https://arxiv.org/pdf/1906.05265}{Preprint
  arXiv:1906.05265}}, 2019.

\bibitem[Sch73]{Schupp}
P.~E. Schupp.
\newblock A survey of {SQ}-universality.
\newblock pages 183--188. Lecture Notes in Math., Vol. 319, 1973.

\bibitem[SU04]{ShestakovUmirbaev}
I.~P. Shestakov \& U.~U. Umirbaev.
\newblock The tame and the wild automorphisms of polynomial rings in three
  variables.
\newblock {\em J. Amer. Math. Soc.}, 17(1):197--227, 2004.

\bibitem[Sil09]{silverman}
J.~H. Silverman.
\newblock {\em The arithmetic of elliptic curves}, volume 106 of {\em Graduate
  Texts in Mathematics}.
\newblock Springer, Dordrecht, second edition, 2009.

\bibitem[Ume85]{Ume}
H.~Umemura.
\newblock On the maximal connected algebraic subgroups of the {C}remona group.
  {II}.
\newblock In {\em Algebraic groups and related topics ({K}yoto/{N}agoya,
  1983)}, volume~6 of {\em Adv. Stud. Pure Math.}, pages 349--436.
  North-Holland, Amsterdam, 1985.

\bibitem[Zim18]{Zimmermann}
S.~Zimmermann.
\newblock The {A}belianization of the real {C}remona group.
\newblock {\em Duke Math. J.}, 167(2):211--267, 2018.

\end{thebibliography}
\endgroup

\end{document}